\documentclass[12pt,reqno]{amsart}
\usepackage{amsmath}
\usepackage{amssymb}
\usepackage{verbatim}
\usepackage{mathrsfs}
\usepackage{graphicx} 
\usepackage[usenames,dvipsnames]{pstricks}
\usepackage{pst-plot,pstricks-add}
\usepackage{wavelet}

\usepackage[top=0.6in,bottom=0.6in,left=0.7in,right=0.7in]{geometry}

\usepackage{enumitem}
\usepackage{multirow}
\usepackage{subcaption}

\newtheorem{lemma}{Lemma}
\newtheorem{theorem}[lemma]{Theorem}
\newtheorem{cor}[lemma]{Corollary}
\newtheorem{example}{Example}
\newtheorem{algorithm}[lemma]{Algorithm}

\renewcommand{\pth}{\fth}
\newcommand{\pE}{\mathsf{E}}
\newcommand{\pF}{\mathsf{F}}
\newcommand{\ldeg}{\operatorname{ldeg}}
\newcommand{\sig}{\operatorname{sig}}
\newcommand{\sign}{\operatorname{sign}}

\numberwithin{equation}{section}

\begin{document}

\title[Generalized Matrix Spectral Factorization and Quasi-tight Framelets]{Generalized Matrix Spectral Factorization and Quasi-tight Framelets with Minimum Number of Generators}

\author{Chenzhe Diao and Bin Han}

\address{Department of Mathematical and Statistical Sciences,
University of Alberta, Edmonton,\quad Alberta, Canada T6G 2G1.
\quad {\tt diao@ualberta.ca},\quad {\tt bhan@ualberta.ca}
}

\thanks{
Research was supported in part by NSERC Canada under Grant RGP 228051.
}

\makeatletter \@addtoreset{equation}{section} \makeatother
\begin{abstract}
As a generalization of orthonormal wavelets in $\Lp{2}$, tight framelets (also called tight wavelet frames) are of importance in wavelet analysis and applied sciences due to their many desirable properties in applications such as image processing and numerical algorithms.
Tight framelets are often derived from particular refinable functions satisfying certain stringent conditions. Consequently, a large family of refinable functions cannot be used to construct tight framelets. This motivates us to introduce the notion of a quasi-tight framelet, which is a dual framelet but behaves almost like a tight framelet.
It turns out that the study of quasi-tight framelets is intrinsically linked to
the problem of the generalized matrix spectral factorization for matrices of Laurent polynomials.
In this paper, we provide a systematic investigation on the generalized matrix spectral factorization problem
and compactly supported quasi-tight framelets. As an application of our results on generalized matrix spectral factorization for matrices of Laurent polynomials,
we prove in this paper that from any arbitrary compactly supported refinable function in $\Lp{2}$, we can always construct a compactly supported one-dimensional quasi-tight framelet having the minimum number of generators and the highest possible order of vanishing moments. Our proofs are constructive and supplemented by step-by-step algorithms.
Several examples of quasi-tight framelets will be provided to illustrate the theoretical results and algorithms developed in this paper.
\end{abstract}


 \maketitle

\pagenumbering{arabic}


\section{Introduction and Motivations}

Due to their many desirable properties such as sparse multiscale representations and fast transforms, orthogonal wavelets have been employed in many applications such as signal/image processing and numerical algorithms (\cite{daubook}). As a generalization of an orthogonal wavelet, a tight framelet (also called a tight wavelet frame) preserves almost all the desirable properties of an orthogonal wavelet and offer many extra new features such as directionality and redundant representations in applications (e.g., \cite{chs02,dhrs03,ds10,hz14,shen10} and many references therein).
Before explaining our motivations of this paper, let us recall the definition of tight framelets.
For a function $f$ defined on the real line $\R$, we shall adopt the following notation:
\[
f_{\gl;k}:=|\gl|^{1/2}f(\gl\cdot-k),\qquad \gl \in \R\bs\{0\}, k\in \R.
\]
For square integrable functions $\eta,\psi^1,\ldots,\psi^s\in \Lp{2}$,
we say that $\{\eta;\psi^1,\ldots,\psi^s\}$ is \emph{a tight framelet} in $\Lp{2}$ if
every function $f\in \Lp{2}$ has the following multiscale representation:
\be \label{tf}
f=\sum_{k\in \Z} \la f, \eta(\cdot-k)\ra \eta(\cdot-k)+
\sum_{j=0}^\infty \sum_{\ell=1}^s \sum_{k\in \Z} \la f, \psi^\ell_{2^j;k}\ra \psi^\ell_{2^j;k}
\ee
with the series converging unconditionally in $\Lp{2}$. Moreover, if $\{\eta;\psi^1,\ldots,\psi^s\}$ is a tight framelet in $\Lp{2}$, then $\{\psi^1,\ldots,\psi^s\}$ is \emph{a homogeneous tight framelet} in $\Lp{2}$ (e.g. see \cite[Proposition~4]{han12} and \cite{han97,rs97}), i.e.,
\be \label{htf}
f=
\sum_{j\in \Z} \sum_{\ell=1}^s \sum_{k\in \Z} \la f, \psi^\ell_{2^j;k}\ra \psi^\ell_{2^j;k},\qquad \forall\, f\in \Lp{2}
\ee
with the series converging unconditionally in $\Lp{2}$. By $\lp{0}$ we denote the space of all finitely supported sequences on $\Z$.
In this paper we are interested in compactly supported generating framelet functions $\psi^1,\ldots,\psi^s$, which are derived from a compactly supported \emph{refinable function} $\phi$ satisfying
\be \label{reffunc}
\phi=2\sum_{k\in \Z} a(k) \phi(2\cdot-k)
\ee
for some finitely supported sequence/filter $a\in \lp{0}$.
For a filter $a=\{a(k)\}_{k\in \Z} \in \lp{0}$, we define its associated Laurent polynomial to be $\pa(z):=
\sum_{k\in \Z} a(k)z^k$ for $z\in \C\bs\{0\}$. Suppose that a filter $a\in \lp{0}$ satisfies $\sum_{k\in \Z} a(k)=1$, i.e., $\pa(1)=1$.
Using the Fourier transform, we obtain a refinable function/distribution $\phi$ through
\be \label{phi}
\wh{\phi}(\xi):=\prod_{j=1}^\infty \pa(e^{-i 2^{-j}\xi}),\qquad \xi\in \R,
\ee
where the Fourier transform used in this paper is defined to be $\wh{f}(\xi):=\int_{\R} f(x) e^{-ix\xi} dx$ for $f\in \Lp{1}$
and can be naturally extended to square integrable functions and tempered distributions.
It is trivial to check that $\wh{\phi}(2\xi)=\pa(e^{-i\xi})\wh{\phi}(\xi)$, which is equivalent to \eqref{reffunc}.

Suppose that the refinable function $\phi$ associated with low-pass filter $a$ belongs to $\Lp{2}$.
A general procedure called oblique extension principle (OEP) has been introduced in \cite{dhrs03} and independently in \cite{chs02} for constructing compactly supported tight framelets from the refinable function $\phi$. For $\theta,b_1,\ldots,b_s\in \lp{0}$, we define
\be \label{eta:psi}
\eta:=\sum_{k\in \Z} \theta(k) \phi(\cdot-k)
\quad \mbox{and}\quad
\psi^\ell:=2\sum_{k\in \Z} b_\ell(k)\phi(2\cdot-k),\qquad \ell=1,\ldots,s.
\ee
Since $\phi\in \Lp{2}$ and all filters are finitely supported, we have $\eta,\psi^1,\ldots,\psi^s\in \Lp{2}$.
Then $\{\eta;\psi^1,\ldots,\psi^s\}$ is a tight framelet in $\Lp{2}$ (e.g., see \cite[Theorem~6.4.2]{hanbook} and \cite{chs02,dhrs03,dh04,han03,han12}) if and only if
$\pTh(1)=1$ and $\{a;b_1,\ldots,b_s\}_\Theta$ is \emph{an (OEP-based) tight framelet filter bank} satisfying
\begin{align}
&\pTh(z^2) \pa(z) \pa^\star(z)+\pb_1(z) \pb_1^\star(z)+\cdots+\pb_s(z)\pb_s^\star(z)=\pTh(z),
\qquad z\in \C\bs\{0\}, \label{tffb:1}\\
&\pTh(z^2) \pa(z) \pa^\star(-z)+\pb_1(z) \pb_1^\star(-z)+\cdots+\pb_s(z)\pb_s^\star(-z)=0, \qquad z\in \C\bs\{0\},\label{tffb:0}
\end{align}
where $\pTh(z):=\pth(z)\pth^\star(z)$. Here we define $\pu^\star(z):=\sum_{k\in \Z} \ol{u(k)}^\tp z^{-k}$ for a finitely supported (matrix-valued) sequence $u:=\{u(k)\}_{k\in \Z}: \Z\rightarrow \C^{m\times n}$. Notice that $ \pu^\star(z) = [\pu(z)]^\star := \overline{\pu(z)}^\tp $ for all $ z\in \T $. Therefore, the task of constructing a tight framelet is reduced to constructing a tight framelet filter bank.
In fact, it is known in \cite[Theorem~4.5.4]{hanbook} that every tight framelet $\{\eta;\psi^1,\ldots,\psi^s\}$ in $\Lp{2}$ must come from a refinable function $\phi$ through the refinable structure in \eqref{eta:psi}.
One-dimensional tight framelets and tight framelet filter banks have been extensively investigated and constructed in the literature, to only mentioned a few, see \cite{ch00,chs02,daubook,dhrs03,DonShe:2007Pseudo-splines,han03,han15,hanbook,HanMo:2005Symmetric,jiang03,js15,sel01} and references therein.

One of the most important features of wavelets is the sparse multiscale representations in \eqref{tf} and \eqref{htf}. The sparsity of the representations in \eqref{tf} and \eqref{htf} come from the vanishing moments of the framelet/wavelet generators $\psi^1,\ldots,\psi^s$ in \eqref{eta:psi}, e.g., see \cite{daubook}. For a compactly supported function $\psi\in \Lp{2}$, we say that $\psi$ has \emph{$m$ vanishing moments} if $\int_{\R} x^j \psi(x) dx=0$ for all $j=0,\ldots, m-1$.
If in addition $\psi=2\sum_{k\in \Z} b(k)\phi(2\cdot-k)$ with $b\in \lp{0}$ and $\wh{\phi}(0)\ne 0$, then one can easily deduce that $\psi$ has $m$ vanishing moments if and only if the filter $b$ has $m$ vanishing moments, i.e., $\sum_{k\in \Z} k^j b(k)=0$ for all $j=0,\ldots,m-1$. We define $\vmo(b):=m$ with $m$ being the largest such integer. For convenience, we also define $\vmo(\pb(z)):=\vmo(b)$. The notion of vanishing moments is closely related to sum rules. For a filter $a\in \lp{0}$, we say that $a$ has \emph{$n$ sum rules} (\cite{daubook}) if
\be \label{sr}
\sum_{k\in \Z} a(2k)(2k)^j=\sum_{k\in \Z} a(1+2k) (1+2k)^j, \qquad \forall\, j=0,\ldots,n-1.
\ee
Note that $a$ has $n$ sum rules if and only if $\pa(z)=(1+z)^n \pu(z)$ for some Laurent polynomial $\pu$.
We define $\sr(\pa(z)):=\sr(a):=n$ with $n$ being the largest such integer.
If $\{a;b_1,\ldots,b_s\}_\Theta$ is a tight framelet filter bank with $ \pTh(1)\pa(1)\neq 0 $, then one can easily deduce from \eqref{tffb:1} and \eqref{tffb:0}
(e.g., see \cite[Proposition 3.3.1]{hanbook} and \cite{chs02,dhrs03,han15}) that
\be \label{tf:vm:sr}
\min(\vmo(b_1),\ldots,\vmo(b_s))=\min(\sr(a), \tfrac{1}{2}\vmo(\pTh(z)-\pTh(z^2)\pa(z)\pa^\star(z))).
\ee
For a given low-pass filter $a$, the role of the filter $\Theta$ is to increase the vanishing moments of $\pTh(z)-\pTh(z^2)\pa(z)\pa^\star(z)$ so that all the high-pass filters $b_1,\ldots,b_s$ have high orders of vanishing moments.

Note that the equations in \eqref{tffb:1} and \eqref{tffb:0} for a tight framelet filter bank $\{a;b_1,\ldots,b_s\}_\Theta$ can be equivalently expressed in the matrix form:
\be \label{tffb}
\left[\begin{matrix}
\pb_1(z) &\cdots &\pb_s(z)\\
\pb_1(-z) &\cdots &\pb_s(-z)\end{matrix}\right]
\left[\begin{matrix}
\pb_1(z) &\cdots &\pb_s(z)\\
\pb_1(-z) &\cdots &\pb_s(-z)\end{matrix}\right]^\star
=\cM_{\pa,\pTh}(z)
\ee
with
\be \label{cM}
\cM_{\pa,\pTh}(z):=
\left[ \begin{matrix} \pTh(z)-\pTh(z^2) \pa(z)\pa^\star(z) &-\pTh(z^2)\pa(z)\pa^\star(-z)\\
-\pTh(z^2) \pa(-z)\pa^\star(z) &\pTh(-z)-\pTh(z^2)\pa(-z)\pa^\star(-z)\end{matrix}
\right].
\ee
Recall that an $ n\times n $ Hermite matrix $ A $ is called \emph{positive semidefinite}, denoted by $A\ge 0$, if and only if $ x^\star A x \geqslant 0 $ for all $ x\in \C^n $.
Obviously, \eqref{tffb} implies $\cM_{\pa,\pTh}(z)\ge 0$ for all $z\in \T:=\{\zeta \in \C \setsp |\zeta|=1\}$, which is known (see \cite{chs02,dhrs03}, \cite[Lemma~6]{han15} and \cite[Lemma~1.4.5]{hanbook}) to be equivalent to that $\Theta(z)\ge 0$ for all $ z\in \T $ and
\be \label{det:cM}
\det(\cM_{\pa,\pTh}(z))=
\pTh(z)\pTh(-z)-\pTh(z^2)[\pTh(-z)\pa(z)\pa^\star(z)+
\pTh(z)\pa(-z)\pa^\star(-z)]\ge 0,\qquad \forall\, z\in \T.
\ee
Consequently, by the Fej\'er-Riesz lemma, there exists a Laurent polynomial $\pth(z)$ such that $\pth(z)\pth^\star(z)=\pTh(z)$ so that we can define the function $\eta$ in \eqref{eta:psi}.
One often can construct a filter $\Theta\in \lp{0}$ so that $\pTh(z)\ge 0$ for all $z\in \T$ and $\vmo(\pTh(z)-\pTh(z^2)\pa(z)\pa^\star(z))$ is reasonably high. However, many pairs of filters $a, \Theta\in \lp{0}$ do not satisfy the condition in \eqref{det:cM}.
Let us provide an example here for the most popular choice of $\pTh(z)=1$.
Let $u\in \lp{0}$ be a filter such that $|\pu(z_0)|^2+|\pu(z_0)|^2<1$ for some $z_0\in \T$. Let $a\in \lp{0}$ be an arbitrary dual filter of $u$, that is, $\pa(z)\pu^\star(z)+\pa(-z)\pu^\star(-z)=1$ for all $z\in \T$.
Consequently, by the Cauchy-Schwarz inequality, we have $1\le (|\pa(z_0)|^2+|\pa(-z_0)|^2)
(|\pu(z_0)|^2+|\pu(-z_0)|^2)$, from which we have
$|\pa(z_0)|^2+|\pa(-z_0)|^2 \ge
(|\pu(z_0)|^2+|\pu(-z_0)|^2)^{-1}>1$.
Therefore, by
$\pTh(z)=1$, we have $\det(\cM_{a,\Theta}(z))=1-\pa(z)\pa^\star(z)-\pa(-z)\pa^\star(-z)$
but $\det(\cM_{a,\Theta}(z_0))<0$.
This shows that the condition in \eqref{det:cM} fails for many filters $a$ even with the most popular and simplest choice of $\pTh(z)=1$.
Hence, a tight framelet cannot be derived from the refinable function associated with the filter $a$ and $\pTh(z)=1$.
Also, some papers try to design general $ \pTh(z) $ to guarantee $ \cM_{\pa, \pTh}(z) \geqslant 0 $ for all $ z \in \T $. However, in order to prove the existence of such $ \pTh(z) $, they have to put additional assumptions on the spectral radius of the transition operator associated with the low-pass filter $ a $, or the stability of the integer shifts of the refinable function $ \phi $, e.g., see \cite{chs02,dhrs03,HanMo:2005Symmetric}.

This motivates us to introduce the notion of quasi-tight framelet filter banks. Let $\Theta,a,b_1,\ldots,b_s\in \lp{0}$ and $\eps_1,\ldots,\eps_s\in \{-1,1\}$. We say that $\{a;b_1,\ldots,b_s\}_{\Theta, (\epsilon_1,\ldots,\epsilon_s)}$ is \emph{a quasi-tight framelet filter bank} if
\be \label{qtffb}
\left[\begin{matrix}
\pb_1(z) &\cdots &\pb_s(z)\\
\pb_1(-z) &\cdots &\pb_s(-z)\end{matrix}\right]
\left[\begin{matrix} \eps_1 & &\\
&\ddots &\\ & &\eps_s\end{matrix}\right]
\left[\begin{matrix}
\pb_1(z) &\cdots &\pb_s(z)\\
\pb_1(-z) &\cdots &\pb_s(-z)\end{matrix}\right]^\star
=\cM_{\pa,\pTh}(z),
\ee
where $\cM_{\pa,\pTh}$ is defined in \eqref{cM}. Hence, a tight framelet filter bank is a special case of
a quasi-tight framelet filter bank with $\eps_1=\cdots=\eps_s = 1$.
We call $ \epsilon_\ell \in \{-1, 1\} $ the \emph{signature} of the filter $ b_\ell $, $ \ell = 1,\ldots,s $.
Moreover, it is straightforward to observe that
$\{a;b_1,\ldots,b_s\}_{\Theta, (\epsilon_1,\ldots,\epsilon_s)}$ is a quasi-tight framelet filter bank if and only if $\{a;b_1,\ldots,b_s\}_{-\Theta, (-\epsilon_1, \ldots, -\epsilon_s)}$
is a quasi-tight framelet filter bank.
Assume that $\pa(1)=1$ and $\phi\in \Lp{2}$ with $\phi$ being defined in \eqref{phi}.
Write $\pTh(z)=\tilde{\pth}(z)\pth^\star(z)$ for some $\tilde{\theta}, \theta\in \lp{0}$. Define $\eta,\psi^1,\ldots,\psi^s$ as in \eqref{eta:psi} and $\tilde{\eta}:=\sum_{k\in\Z} \tilde{\theta}(k) \phi(\cdot-k)$. If in addition $\pTh(z)\ge 0$ for all $z\in \T$, then by Fej\'er-Riesz lemma we can always choose $\tilde{\pth}(z)=\pth(z)$ so that $\tilde{\eta}=\eta$.
If $\pTh(1)=1$ and $\pb_1(1)=\cdots=\pb_s(1)=0$,
by \cite[Theorems~4.1.9 and 6.4.1]{hanbook} and \cite[Theorem~2.3]{han03},
then
$\{\eta,\tilde{\eta};\psi^1,\ldots,\psi^s\}_{ (\eps_1,\ldots,\eps_s)}$ is \emph{a quasi-tight framelet} in $\Lp{2}$, that is, for all $ f \in \Lp{2} $,
\be \label{qtf}
f=\sum_{k\in \Z} \la f, \tilde{\eta}(\cdot-k)\ra \eta(\cdot-k)+
\sum_{j=0}^\infty \sum_{\ell=1}^s \sum_{k\in \Z} \eps_\ell \la f, \psi^\ell_{2^j;k}\ra \psi^\ell_{2^j;k}
\ee
with the series converging unconditionally in $\Lp{2}$ and the underlying system being a Bessel sequence in $\Lp{2}$.
By  \cite[Proposition~4]{han12}, it follows directly from \eqref{qtf} that $\{\psi^1,\ldots,\psi^s\}_{(\eps_1,\ldots,\eps_s)}$ is
\emph{a homogeneous quasi-tight framelet} in $\Lp{2}$, that is,
\be \label{hqtf}
f=\sum_{j\in \Z} \sum_{\ell=1}^s \sum_{k\in \Z} \eps_\ell \la f, \psi^\ell_{2^j;k}\ra \psi^\ell_{2^j;k},\qquad \forall\, f\in \Lp{2}
\ee
with the series converging unconditionally in $\Lp{2}$ and the underlying system being a Bessel sequence in $\Lp{2}$.
The multiscale representations in \eqref{qtf} and \eqref{hqtf} using a quasi-tight framelet are very similar to those in \eqref{tf} and \eqref{htf} under a tight framelet. Therefore, a quasi-tight framelet is a special class of dual framelets in $\Lp{2}$ but behaves almost identically to a tight framelet with the exception of possible sign changes of framelet coefficients. An example of quasi-tight framelets and quasi-tight framelet filter banks was probably first observed in \cite[Example~3.2.2]{hanbook} and was obtained by applying the general algorithm in \cite{han15} for constructing dual framelet filter banks.
The equations in \eqref{qtffb} for a quasi-tight framelet filter bank are intrinsically linked to the problem of matrix spectral factorization for which we shall extensively study in this paper. Moreover, similar to the identity in \eqref{tf:vm:sr} for a tight framelet filter bank, if
$\{a;b_1,\ldots,b_s\}_{\Theta,(\epsilon_1,\ldots, \epsilon_s)}$ is a quasi-tight framelet filter bank, then we have
\be \label{qtf:vm:sr}
\min(\vmo(b_1),\ldots,\vmo(b_s))\le \min(\sr(a), \tfrac{1}{2}\vmo(\pTh(z)-\pTh(z^2)\pa(z)\pa^\star(z))).
\ee
That is, the highest possible order of vanishing moments achieved by
a quasi-tight framelet filter bank derived from given filters $a,\Theta\in \lp{0}$ is $\min(\sr(a), \tfrac{1}{2}\vmo(\pTh(z)-\pTh(z^2)\pa(z)\pa^\star(z)))$.

As demonstrated in \cite[Theorem~7]{han15} and \cite[Theorem~1.4.7]{hanbook}, for general filters $a,\Theta\in \lp{0}$, $\det(\cM_{a,\Theta}(z))$ is often not identically zero and
the minimum number $s$ of high-pass filters in a quasi-tight framelet filter bank is at least $2$. Given a Laurent polynomial $ \pp(z) $, for simplicity, we use $ \pp(z)\equiv 0 $ ($ \pp(z)\not\equiv 0 $) to indicate that $ \pp(z) $ is (is not) identically zero.

For an $n\times n$ square matrix $\pA(z)$ of Laurent polynomials, its spectrum $\sigma(\pA)$ is defined to be
\be \label{sigmaA}
\sigma(\pA):=\{z\in\C\setminus \{0\} \setsp \det(\pA(z))=0\}.
\ee
If $ \pA^\star(z) = \pA(z) $, then $ \pA(z) $ is a Hermite matrix for all $ z\in \T $ and we call such $ \pA(z) $ \emph{a Hermite matrix of Laurent polynomials}.
In this case, for all $ z\in \T $, all the eigenvalues of $\pA(z)$ are real numbers and hence,
we define $\nu_+(\pA(z))$ to be the number of positive eigenvalues of the matrix $\pA(z)$, and define $\nu_-(\pA(z))$ to be the number of negative eigenvalues of the matrix $\pA(z)$. In particular, for filters $a,\Theta\in \lp{0}$, we define
\be \label{saTheta}
s_{a,\Theta}^+ := \max_{z\in \T}\nu_+(\cM_{\pa,\pTh}(z)), \qquad
s_{a,\Theta}^- := \max_{z\in \T}\nu_-(\cM_{\pa,\pTh}(z)), \qquad
\mbox{and}\quad
s_{a,\Theta} := s_{a,\Theta}^+ + s_{a,\Theta}^-,
\ee
where the matrix $\cM_{\pa,\pTh}(z)$ is defined in \eqref{cM}.

Through the study of the generalized matrix spectral factorization in \eqref{qtffb}, we now state the main result obtained in this paper on quasi-tight framelets with the minimum number of generators and the highest possible order of vanishing moments derived from any arbitrarily given filters $a,\Theta\in \lp{0}$.

\begin{theorem}\label{thm:qtf}
Let $a,\Theta\in \lp{0}\bs\{0\}$ be two finitely supported
not-identically-zero filters such that $\pTh^\star=\pTh$.
Let $n_b$ be any positive integer satisfying
\be \label{nb}
1\le n_b\le \min(\sr(a), \tfrac{1}{2} \vmo(\pTh(z)-\pTh(z^2) \pa(z)\pa^\star(z))).
\ee
Let $\cM_{\pa,\pTh}(z)$ be defined in \eqref{cM} and
the quantities $s_{a,\Theta}^+, s_{a,\Theta}^-, s_{a,\Theta}$ be defined in \eqref{saTheta}.
Define $s:=s_{a,\Theta}$.
Then there exist $b_1,\ldots,b_s\in \lp{0}$
and $ \eps_1=\ldots= \eps_{s_{a,\Theta}^+} = 1 $,  $\eps_{s_{a,\Theta}^+ +1}=\ldots=\eps_s=-1$ such that $\{a;b_1,\ldots,b_s\}_{\Theta, (\epsilon_1,\ldots, \epsilon_s)}$ is a quasi-tight framelet filter bank with
$ \min\{\vmo(b_1),\ldots, \vmo(b_s)\}\geqslant n_b $.
Moreover, for $1\le s<s_{a,\Theta}$, there does not exist a quasi-tight framelet filter bank
$\{a;b_1,\ldots,b_s\}_{\Theta, (\epsilon_1, \ldots, \epsilon_s)}$ with $b_1,\ldots,b_s\in \lp{0}$ and $\eps_1,\ldots,\eps_s\in\{-1,1\}$.
Furthermore, if $\pa(1)=\pTh(1)=1$ and $\phi\in \Lp{2}$ with $\phi$ being defined in \eqref{phi}, then
$\{\eta,\tilde{\eta};\psi^1,\ldots,\psi^s\}_
{(\eps_1,\ldots,\eps_s)}$ is a quasi-tight framelet in $\Lp{2}$, where $\eta,\psi^1,\ldots,\psi^s\in \Lp{2}$ are defined in \eqref{eta:psi} and $\tilde{\eta}:=\sum_{k\in \Z} \tilde{\theta}(k)\phi(\cdot-k)$ with $\tilde{\pth}(z)\pth^\star(z)=\pTh(z)$.
\end{theorem}

Since quasi-tight framelets preserve most desirable properties of tight framelets and enjoy great flexibility as demonstrated in Theorem~\ref{thm:qtf}, we expect that quasi-tight framelets will be as useful as tight framelets in applications.
We also mention that our investigation on quasi-tight framelets is much involved than the study of tight framelets in \cite{chs02,dhrs03,han15,sel01} and the approach taken in these papers for tight framelets does not carry over to general quasi-tight framelets.
Our proof of Theorem~\ref{thm:qtf} is constructive and we shall provide an algorithm to construct the filters in Theorem~\ref{thm:qtf}.

To prove Theorem~\ref{thm:qtf} on quasi-tight framelet filter banks,
we shall establish two main results on generalized matrix spectral factorizations.
If $ A $ is an $ n\times n $ Hermite matrix, its \emph{signature} $ \sig(A) $ is defined as
\[
\sig(A) := \nu_+(A) - \nu_-(A),
\]
where $\nu_+(A)$ and $\nu_-(A)$ are the numbers of its positive and negative eigenvalues, respectively.
For a Hermite matrix $ \pA(z) $ of Laurent polynomials, we say that it has constant signature if $ \sig(\pA(z)) $ is constant for all $ z\in \T\bs\sigma(\pA)$. In this situation, we can easily see that $ \nu_+(\pA(z)) $ and $ \nu_-(\pA(z)) $ remain constant for all $ z\in \T\bs\sigma(\pA) $.
For Hermite matrices of Laurent polynomials with constant signature, we have the following result on the generalized spectral factorization problem.

\begin{theorem} \label{thm:const-sig}
Let $\pA(z)$ be an $n \times n$ Hermite matrix of Laurent polynomials such that
$\det(\pA(z))$ is not identically zero. If $\nu_+(\pA(z)) = \nu_+$ and $\nu_-(\pA(z)) = \nu_-$ for all $ z\in \T\setminus\sigma(\pA)$ for some nonnegative integers $ \nu_+ $ and $ \nu_- $,
then there exists an $n \times n$ matrix $\pU(z)$ of Laurent polynomials such that
$\pA(z)=\pU(z)\pD\pU^\star(z)$,
where $\pD:=\diag(\mathbf{I}_{\nu_+}, -\mathbf{I}_{\nu_-})$ is an $n \times n$ constant diagonal matrix.
\end{theorem}

If $\pA(z)\ge 0$ for all $z\in \T$, then it is trivial that $\nu_+(\pA(z))=n$ and $\nu_-(\pA(z))=0$ for all $z\in \T\bs \sigma(\pA)$. Therefore, for the special case $\pA(z)\ge 0$ for all $z\in \T$,
Theorem~\ref{thm:const-sig} reduces to the standard result on matrix spectral factorization (also known as Matrix-valued Fej\'er-Riesz Lemma) for nonnegative Hermite matrices of Laurent polnomials, which has been extensively studied in the literature, e.g., see \cite{RosRov:1985Hardy,HarHogSun:2004The-matrix-valued,EphJanLag:2009A-simple} and many references therein. This classical result on matrix spectral factorization plays a key role in the construction of tight framelets and tight framelet filter banks with two (non-symmetric) high-pass filters, e.g., see \cite{chs02,dhrs03,sel01} and references therein.

For a general Hermite matrix of Laurent polynomials, we have

\begin{theorem} \label{thm:nonconst-sig}
Let $\pA(z)$ be an $n \times n$ Hermite matrix of Laurent polynomials such that
$\det(\pA(z))$ is not identically zero.
Then there exists some $n \times m$ matrix $\pU(z)$ of Laurent polynomials
such that $\pA(z)=\pU(z)\pD\pU^\star(z)$ holds
with
$\pD =\diag( \mathbf{I}_{m_1}, -\mathbf{I}_{m_2})$ and $m:=m_1 + m_2$
if and only if
\begin{equation} \label{eq:LargeSig}
m_1 \geqslant \max_{z\in\T}\nu_+(\pA(z)), \qquad
m_2 \geqslant \max_{z\in\T}\nu_-(\pA(z)).
\end{equation}
\end{theorem}

The above Theorems~\ref{thm:const-sig} and~\ref{thm:nonconst-sig} play a key role in our proof of Theorem~\ref{thm:qtf} and our study on quasi-tight framelets and quasi-tight framelet filter banks. Moreover, our proofs to Theorems~\ref{thm:const-sig} and~\ref{thm:nonconst-sig} are constructive and supplemented by step-by-step algorithms.
We also mention that the generalized matrix spectral factorization problem for matrices of polynomials has been extensively investigated in the literature of engineering, for example, see \cite{GohLanRod:1980Spectral,GohLanRod:1982Factorization,RanRod:1994Factorization,RanZiz:1997On-self-adjoint} and many references therein. However, there are barely any references on the generalized matrix spectral factorization problem for matrices of Laurent polynomials. 
Although the proofs of our construction share some similarities to the polynomial results \cite{GohLanRod:1982Factorization,RanRod:1994Factorization},
indeed, many new ideas and techniques are needed in order to handle the generalized matrix spectral factorization problem for matrices of Laurent polynomials. 

The structure of the paper is as follows. In Section~2 we shall
prove Theorem~\ref{thm:qtf} using Theorems~\ref{thm:const-sig} and~\ref{thm:nonconst-sig} on generalized matrix spectral factorization.
In Section~3 we shall provide a few examples of quasi-tight framelet filter banks and quasi-tight framelets in $\Lp{2}$ to illustrate our main results on quasi-tight framelets. In Section~4 we shall prove Theorem~\ref{thm:const-sig}
on generalized matrix spectral factorization with constant signature.
For improved readability, a few technical results for proving Theorem~\ref{thm:const-sig} are presented in the Appendix.
In Section~5, we shall prove Theorem~\ref{thm:nonconst-sig}.
Finally, in Section~6 we shall briefly discuss some extension of our results to one-dimensional quasi-tight framelets with a general dilation factor.

\section{Proof of Theorem~\ref{thm:qtf} on Quasi-tight Framelets}

In this section, we shall prove Theorem~\ref{thm:qtf} using Theorems~\ref{thm:const-sig} and \ref{thm:nonconst-sig} on generalized matrix spectral factorization. The proofs of Theorems~\ref{thm:const-sig} and \ref{thm:nonconst-sig} will be presented in Sections~4 and 5.

Before proving Theorem~\ref{thm:qtf}, we need the following lemma.

\begin{lemma} \label{lem:maxminT}
Let $ \pA(z) $ be an $ n\times n $ Hermite matrix of Laurent polynomials. Then
\[
\max_{z\in \T}\nu_+(\pA(z)) = \max_{z\in \T \bs B}\nu_+(\pA(z)), \qquad
\max_{z\in \T}\nu_-(\pA(z)) = \max_{z\in \T \bs B}\nu_-(\pA(z))
\]
for any finite subset $B$ of $\T$.
\end{lemma}

\begin{proof}
Define $ n_+ := \max_{z\in \T} \nu_+(\pA(z)) $. Then there exists some $ z_0 \in \T $, such that $ \nu_+(\pA(z_0)) = n_+ $.
Since $ \pA(z) $ is an $ n\times n $ Hermite matrix of Laurent polynomials, its $ n $ eigenvalues $ \lambda_1(z), \ldots, \lambda_n(z) $, which are all the roots of the polynomial $ \det(\lambda\mathbf{I_n} - \pA(z)) $, can be chosen as real-valued continuous functions on $ \T $. (They are actually algebraic functions which are globally analytic.)
Therefore, there exists a neighborhood $ U(z_0) $ of $ z_0 $ on $ \T $, such that $ \nu_+(\pA(z)) = n_+ $ for all $ z\in U(z_0) $. As $ U(z_0) $ contains infinitely many points, the set $ U(z_0)\setminus B $ must be nonempty.
This implies that
\[
\max_{z\in \T \bs B} \nu_+(\pA(z)) \geqslant n_+ = \max_{z\in \T} \nu_+(\pA(z)).
\]
Since $ \T\bs B $  is a subset of $ \T $,
we trivially have $\max_{z\in \T \bs B} \nu_+(\pA(z))\le n_+$.
This proves
$\max_{z\in \T}\nu_+(\pA(z)) = \max_{z\in \T \bs B}\nu_+(\pA(z))$.
The identity $ \max_{z\in \T}\nu_-(\pA(z)) = \max_{z\in \T \bs B}\nu_-(\pA(z)) $ can be proved similarly.
\end{proof}

For a Laurent polynomial $ \pp(z)\not\equiv 0 $ and $ z_0\in \C\setminus\{0\} $, we define
$\mz(\pp(z), z_0) $ to be the multiplicity of the root of $ \pp(z) $ at $ z_0 $. That is, $ \mz(\pp(z), z_0) $ is the nonnegative integer such that $ (z-z_0)^{\mz(\pp(z), z_0)} \mid \pp(z) $ but $ (z-z_0)^{\mz(\pp(z), z_0)+1} \nmid \pp(z) $.
Hence, the orders of vanishing moments and sum rules of a Laurent polynomial $ \pp(z) $ can be equivalently expressed by
\[ \vmo(\pp(z)) = \mz(\pp(z), 1) \qquad
\mbox{and} \quad \sr(\pp(z)) = \mz(\pp(z), -1). \]

Also, recall that for a finitely supported sequence $ u \in \lp{0} $ and $ \gamma \in \Z $, \emph{its $ \gamma$-coset sequence $ u^{[\gamma]} $} is defined to be $ u^{[\gamma]}:=\{u(\gamma + 2k)\}_{k\in \Z} $. In terms of Laurent polynomials, we have $ \pu(z) = \pu^{[0]}(z^2) + z \pu^{[1]}(z^2) $.
Moreover,
\[ \begin{bmatrix}
\pb_1(z) & \cdots & \pb_s(z) \\
\pb_1(-z) & \cdots & \pb_s(-z)
\end{bmatrix}
=
\begin{bmatrix}
1 & z \\
1 & -z
\end{bmatrix}
\begin{bmatrix}
\pb_1^{[0]}(z^2) & \cdots & \pb_s^{[0]}(z^2) \\
\pb_1^{[1]}(z^2) & \cdots & \pb_s^{[1]}(z^2)
\end{bmatrix}, \]
where the last $ 2\times s $ matrix is called
the \emph{polyphase matrix} of the filter bank $ \{\pb_1, \ldots, \pb_s\} $.

We now prove Theorem~\ref{thm:qtf} using Theorems~\ref{thm:const-sig} and~\ref{thm:nonconst-sig}.

\begin{proof}[Proof of Theorem~\ref{thm:qtf}]
Since all high-pass filters must have at least $n_b$ vanishing moments, we can write
\begin{equation} \label{eq:bCoset}
\pb_\ell(z) = (1-z^{-1})^{n_b}\mathring{\pb}_\ell(z), \qquad \ell = 1,\ldots, s
\end{equation}
for some Laurent polynomials $\mathring{\pb}_1(z), \ldots, \mathring{\pb}_s(z)$.
Then
$\{a;b_1,\ldots,b_s\}_{\Theta,(\eps_1,\ldots,\eps_s)}$
is a quasi-tight framelet filter bank
satisfying \eqref{qtffb} and \eqref{eq:bCoset} if and only if
\begin{equation} \label{eq:qtfbnb}
\begin{bmatrix}
\mathring{\pb}_1(z) & \cdots & \mathring{\pb}_s(z) \\
\mathring{\pb}_1(-z) & \cdots & \mathring{\pb}_s(-z)
\end{bmatrix}
\begin{bmatrix}
\epsilon_1 & & \\
 & \ddots & \\
 & & \epsilon_s
\end{bmatrix}
\begin{bmatrix}
\mathring{\pb}_1(z) & \cdots & \mathring{\pb}_s(z) \\
\mathring{\pb}_1(-z) & \cdots & \mathring{\pb}_s(-z)
\end{bmatrix}^\star
= \cM_{\pa, \pTh|n_b}(z),
\end{equation}
where
\begin{align} \label{eq:VMfactor}
\cM_{\pa, \pTh|n_b}(z)
:=&
\begin{bmatrix}
(1-z^{-1})^{-n_b} & \\
 & (1+z^{-1})^{-n_b}
\end{bmatrix}
\cM_{\pa, \pTh}(z)
\begin{bmatrix}
(1-z)^{-n_b} & \\
 & (1+z)^{-n_b}
\end{bmatrix}      \\
=&
\begin{bmatrix}
\pA(z) & \pB(z) \\
\pB(-z) & \pA(-z)
\end{bmatrix}, \notag
\end{align}
with
\begin{equation} \label{eq:DefAB}
\pA(z):= \frac{\pTh(z) - \pTh(z^2)\pa(z)\pa^\star(z)}{(1-z)^{n_b}(1-z^{-1})^{n_b}}, \qquad
\pB(z):= \frac{-\pTh(z^2)\pa(z)\pa^\star(-z)}{(1+z)^{n_b}(1-z^{-1})^{n_b}}.
\end{equation}
Note that according to \eqref{nb}, we have
$ 2n_b \leqslant \mz(\pTh(z) - \pTh(z^2)\pa(z)\pa^\star(z), ~1)  $ and
$ n_b \leqslant \mz(\pa(z), ~ -1)  =  \mz(\pa^\star(-z), ~ 1)  $. Hence
$ \pA(z) $ and $ \pB(z) $ are well-defined Laurent polynomials.
Using the coset sequences, we know that \eqref{eq:qtfbnb} is equivalent to
\begin{equation} \label{eq:PRpolyphase}
\begin{bmatrix}
\mathring{\pb}_1^{[0]}(z) & \cdots & \mathring{\pb}_s^{[0]}(z) \\
\mathring{\pb}_1^{[1]}(z) & \cdots & \mathring{\pb}_s^{[1]}(z)
\end{bmatrix}
\begin{bmatrix}
\epsilon_1 & & \\
 & \ddots & \\
 & & \epsilon_s
\end{bmatrix}
\begin{bmatrix}
\mathring{\pb}_1^{[0]}(z) & \cdots & \mathring{\pb}_s^{[0]}(z) \\
\mathring{\pb}_1^{[1]}(z) & \cdots & \mathring{\pb}_s^{[1]}(z)
\end{bmatrix}^\star
=
\cN_{\pa, \pTh|n_b}(z),
\end{equation}
where $\cN_{\pa, \pTh|n_b}(z)$ is calculated from:
\begin{equation} \label{eq:DefN1}
\cM_{\pa, \pTh|n_b}(z) =
\begin{bmatrix}
1 & z \\
1 & -z
\end{bmatrix}
\cN_{\pa, \pTh|n_b}(z^2)
\begin{bmatrix}
1 & z \\
1 & -z
\end{bmatrix}^\star.
\end{equation}
That is,
\begin{equation} \label{eq:DefN2}
\cN_{\pa, \pTh|n_b}(z):=\frac{1}{2}
\begin{bmatrix}
\pA^{[0]}(z) + \pB^{[0]}(z) & z(\pA^{[1]}(z)-\pB^{[1]}(z)) \\
\pA^{[1]}(z) + \pB^{[1]}(z) & \pA^{[0]}(z) - \pB^{[0]}(z)
\end{bmatrix},
\end{equation}
where $\pA(z)$ and $\pB(z)$ are defined in \eqref{eq:DefAB}.
Hence, the existence of a quasi-tight framelet filter bank
$ \{a; b_1, \ldots, b_s\}_{\Theta, (\epsilon_1, \ldots, \epsilon_s)} $ with $ n_b $ vanishing moments necessarily implies a generalized spectral factorization in \eqref{eq:PRpolyphase} for the matrix $ \cN_{\pa, \pTh|n_b}(z) $ of Laurent polynomials.

According to Theorem~\ref{thm:nonconst-sig}, the existence of the generalized spectral factorization in \eqref{eq:PRpolyphase} implies that the number $ s_+ $ of times that $ ``+1"$ appears in $ \{\epsilon_1,\ldots,\epsilon_s\} $ and the number $ s_- $ of times that $ ``-1"$ appears in $ \{\epsilon_1,\ldots,\epsilon_s\} $ must satisfy
\begin{equation} \label{eq:s+s-}
s_+ \geqslant \max_{z\in \T}\nu_+(\cN_{\pa, \pTh|n_b}(z)) , \qquad \mbox{and} \quad
s_- \geqslant \max_{z\in \T}\nu_-(\cN_{\pa, \pTh|n_b}(z)).
\end{equation}
By \eqref{eq:VMfactor} and \eqref{eq:DefN1}, we know that
\[
\cM_{\pa, \pTh}(z)  =
\pP(z)  \cN_{\pa, \pTh|n_b}(z^2) \pP^\star(z)\quad
\mbox{with} \quad
\pP(z) := \begin{bmatrix}
(1-z^{-1})^{n_b} & \\
& (1+z^{-1})^{n_b}
\end{bmatrix}
\begin{bmatrix}
1 & z \\
1 & -z
\end{bmatrix}.
\]
Since $ \det(\pP(z)) = -2 z (1-z^{-1})^{n_b} (1+z^{-1})^{n_b} $, we observe
$ \sigma(\pP) \subseteq \{-1, 1\} $. Hence, $ \sigma(\pP) $ is a finite set.
For $ z\in \T \bs \sigma(\pP) $, the matrix $ \pP(z) $ is a nonsingular matrix.
By Sylvester's law of inertia, we get from $ \cM_{\pa, \pTh}(z) = \pP(z)\cN_{\pa, \pTh|n_b}(z^2)\pP^\star(z) $ that
$$ \nu_+(\cM_{\pa, \pTh}(z)) = \nu_+(\cN_{\pa, \pTh|n_b}(z^2)), \qquad
\nu_-(\cM_{\pa, \pTh}(z)) = \nu_-(\cN_{\pa, \pTh|n_b}(z^2)), \qquad
\forall z \in \T\bs \sigma(\pP). $$
According to Lemma~\ref{lem:maxminT}, we have
\begin{align*}
s_{a, \Theta}^+
=&
\max_{z\in \T} \nu_+(\cM_{\pa, \pTh}(z))
= \max_{z\in \T \bs\sigma(\pP)} \nu_+(\cM_{\pa, \pTh}(z)) \\
=& \max_{z\in \T \bs\sigma(\pP)} \nu_+(\cN_{\pa, \pTh|n_b}(z^2))
= \max_{z\in \T} \nu_+(\cN_{\pa, \pTh|n_b}(z^2))
= \max_{z\in \T} \nu_+(\cN_{\pa, \pTh|n_b}(z)).
\end{align*}
Similarly, $ s_{a, \Theta}^- = \max_{z\in \T} \nu_-(\cN_{\pa, \pTh|n_b}(z)) $.
Therefore, from \eqref{eq:s+s-} we know that the generalized spectral factorization in \eqref{eq:PRpolyphase} implies
\begin{equation} \label{eq:s+s-2}
s_+ \geqslant s_{a, \Theta}^+ , \qquad
s_- \geqslant s_{a, \Theta}^-, \qquad
\mbox{and} \quad
s = s_+ + s_-
\geqslant s_{a, \Theta}^+ + s_{a, \Theta}^- = s_{a, \Theta} .
\end{equation}
Hence, by Theorem~\ref{thm:nonconst-sig},
for $1\le s<s_{a,\Theta}$, there does not exist a quasi-tight framelet filter bank
$\{a;b_1,\ldots,b_s\}_{\Theta, (\epsilon_1, \ldots, \epsilon_s)}$ with $b_1,\ldots,b_s\in \lp{0}$ and $\eps_1,\ldots,\eps_s\in\{-1,1\}$.

On the other hand, given filters $ a, \Theta \in \lp{0}\bs\{0\} $, $ \Theta^\star = \Theta $, and a positive integer $ n_b $ satisfying \eqref{nb}, we can calculate the matrix $ \cN_{\pa, \pTh|n_b}(z) $ of Laurent polynomials from \eqref{eq:DefAB} and \eqref{eq:DefN2}.
By $ \Theta^\star(z) = \Theta(z) $, we deduce from \eqref{eq:DefAB} that
$ \pA^\star(z) = \pA(z) $ and $ \pB^\star(z) = \pB(-z) $.
Plugging these identities into
$ \pA^{[0]}(z^2) = \tfrac{1}{2}\left(\pA(z) + \pA(-z)\right) $,
$ \pA^{[1]}(z^2) = \tfrac{1}{2z}\left(\pA(z) - \pA(-z)\right)$,
$ \pB^{[0]}(z^2) = \tfrac{1}{2}\left(\pB(z) + \pB(-z)\right) $, and $ \pB^{[1]}(z^2) = \tfrac{1}{2z}\left(\pB(z)-\pB(-z)\right)$, we can easily verify that
\[
\pA^{[0]^\star}(z) = \pA^{[0]}(z),\qquad
\pA^{[1]^\star}(z) = z\pA^{[1]}(z), \qquad
\pB^{[0]^\star}(z) = \pB^{[0]}(z),\qquad
\mbox{and}\quad
\pB^{[1]^\star}(z) = -z\pB^{[1]}(z).
\]
Using the above four equations, we deduce from \eqref{eq:DefN2} that $ \cN_{\pa, \pTh|n_b}^\star(z) = \cN_{\pa, \pTh|n_b}(z) $. That is, $ \cN_{\pa, \pTh|n_b}(z) $ is a Hermite matrix of Laurent polynomials.
As we calculated, $ s_{a, \Theta}^+ = \max_{z\in \T} \nu_+(\cN_{\pa, \pTh|n_b}(z)) $
and
$ s_{a, \Theta}^- = \max_{z\in \T} \nu_-(\cN_{\pa, \pTh|n_b}(z)) $.
Take $ s := s_{a, \Theta} = s_{a, \Theta}^+ + s_{a, \Theta}^- $.
According to Theorem~\ref{thm:nonconst-sig},
we can choose
$ \eps_1=\cdots= \eps_{s_{a,\Theta}^+} = 1 $, $\eps_{s_{a,\Theta}^+ +1}=\cdots=\eps_{s}=-1$,
and find a generalized spectral factorization of $ \cN_{\pa, \pTh|n_b}(z) $ as
$ \cN_{\pa, \pTh|n_b}(z) = \pU(z)
\mbox{diag}(\eps_1,\ldots,\eps_s) \pU^\star(z)$,
where $ \pU(z) $ is a $ 2\times s $ matrix of Laurent polynomials.
Define Laurent polynomials $ \mathring{\pb}_1(z), \ldots, \mathring{\pb}_{s}(z) $ by
$
\begin{bmatrix}
\mathring{\pb}_1^{[0]}(z) & \cdots & \mathring{\pb}_{s}^{[0]}(z) \\
\mathring{\pb}_1^{[1]}(z) & \cdots & \mathring{\pb}_{s}^{[1]}(z)
\end{bmatrix} := \pU(z)
$.
Thus, \eqref{eq:PRpolyphase} holds.
Multiplying
$ \begin{bmatrix}
1 & z \\
1 & -z
\end{bmatrix} $ and
$ \begin{bmatrix}
1 & z \\
1 & -z
\end{bmatrix}^\star $ on the left and right side of $ \cN_{\pa, \pTh|n_b}(z^2) $ respectively, we see that \eqref{eq:PRpolyphase} is equivalent to \eqref{eq:qtfbnb} with $ \cM_{\pa, \pTh|n_b}(z) $ being defined in \eqref{eq:VMfactor}.
Define Laurent polynomials $ \pb_1(z)\ldots, \pb_s(z) $ as \eqref{eq:bCoset}, we conclude from \eqref{eq:qtfbnb} that $ \{a; b_1, \ldots, b_s\}_{\Theta, (\epsilon_1, \ldots, \epsilon_s)} $ is a quasi-tight framelet filter bank with
$ \min\{\vmo(b_1),\ldots, \vmo(b_s)\}\geqslant n_b $. This proves the existence of quasi-tight framelet filter bank with minimum number of high-pass filters and high vanishing moments.
\end{proof}

By Theorem~\ref{thm:qtf}, we see that the minimum numbers of high-pass filters with positive and negative signatures in a quasi-tight framelet filter bank are just $s_{a, \Theta}^+ $ and $ s_{a, \Theta}^-$, which are defined in \eqref{saTheta}.
We now explicitly present such quantities in the following for any given filters $a,\Theta \in \lp{0}\bs\{0\}$. Note that the matrix $\cM_{a,\Theta}$ cannot be identically zero.

If $\det(\cM_{a,\Theta}(z))$ is identically zero, then
one of the following two cases must happen:
\begin{enumerate}
\item[(1)] $ \pTh(z)\ge 0 $ for all $z\in \T$ if and only if $ s_{a, \Theta}^+ = 1 $ and $ s_{a, \Theta}^- = 0 $;
\item[(2)] $\pTh(z)\le 0 $ for all $z\in \T$ if and only if $ s_{a, \Theta}^+ = 0 $ and $ s_{a, \Theta}^- = 1$.
\end{enumerate}

Since $ \det(\cM_{\pa, \pTh}(z))\equiv 0$ and $ \cM_{\pa, \pTh}(z) = -\cM_{\pa, -\pTh}(z) $, by \cite[Lemma~1.4.5]{hanbook} or \cite[Lemma~6]{han15}, we conclude that
$\pTh(z)\ge 0$ (or $\pTh(z)\le 0$) for all $z\in \T$ if and only if
$\cM_{\pa, \pTh}(z)\ge 0$ (or $\cM_{\pa, \pTh}(z)\le 0$) for all $z\in \T$.
Note that $0$ must be an eigenvalue of $\cM_{\pa, \pTh}(z)$ by $\det(\cM_{\pa, \pTh}(z))=0$. Hence, if $\cM_{\pa, \pTh}(z)\ge 0$ (or $\cM_{\pa, \pTh}(z)\le 0$) for all $z\in \T$, then the other eigenvalue of $\cM_{\pa, \pTh}(z)$ must be nonnegative (or non-positive) and cannot be identically zero, since $\cM_{a,\pTh}$ cannot be identically zero. This proves items (1) and (2).
We now prove that $\pTh$ cannot change signs on $\T$.
By our assumptions $\pTh^\star=\pTh$ and
\begin{equation}\label{detM}
\det(\cM_{a,\Theta}(z))= \pTh(z)\pTh(-z)-\pTh(z^2)[\pTh(-z)\pa(z)\pa^\star(z)+
	\pTh(z)\pa(-z)\pa^\star(-z)]=0,
\end{equation}
we conclude (see \cite[Theorem~1.4.7]{hanbook} and \cite[Theorem~7]{han15}) that $\pTh(z)\in \R$ for $z\in \T$ and
$\pTh(z)\pTh(-z) = \lambda \pTh(z^2)$
for some nonzero real number $\gl$. Consequently, we have $\fth(z)\fth(-z)=\fth(z^2)$ with $\fth(z):=\pTh(z)/\gl$ and the above identity in \eqref{detM} is equivalent to
\[
\fth(-z)\pu(z)+\fth(z)\pu(-z)=1\qquad \mbox{with}\quad \pu(z):=\pa(z)\pa^\star(z).
\]
Since $\pu(z)\ge 0$ for all $z\in \T$, by the above identity, if $\fth(z_0)<0$ for some $z_0\in \T$, then we must have $\fth(-z_0)>0$ and consequently $\fth(z_0^2)=\fth(z_0)\fth(-z_0)<0$.
By induction, for any $z_0\in \T$, if $\fth(z_0)<0$, then we must have $\fth(z_0^{2^j})<0$ for all $j\in \N$.
If $\pTh$ changes signs on $\T$, then
$\fth(e^{-i\xi})<0$ for some $\xi \in (c,d)$ with $c<d$. Then the above argument shows that $\fth(e^{-i\xi})<0$ for all $\xi\in (2^jc,2^jd)$.
Therefore, we must have $\gl^{-1}\pTh(z) = \fth(z) <0$ for all $z\in \T$, a contradiction to our assumption. This proves that $\pTh$ cannot change signs on $\T$.

If $\det(\cM_{a,\Theta}(z))$ is not identically zero, then
one of the following four cases must happen:
\begin{enumerate}
\item[(3)] $\pTh(z)\ge 0$ and $\det(\cM_{a,\Theta}(z))\ge 0$ for all $z\in \T$ if and only if $ s_{a, \Theta}^+ = 2 $ and $ s_{a, \Theta}^- = 0 $;

\item[(4)] $\pTh(z)\le 0$ and $\det(\cM_{a,\Theta}(z))\ge 0$ for all $z\in \T$ if and only if $ s_{a, \Theta}^+ = 0 $ and $ s_{a, \Theta}^- = 2 $;

\item[(5)] $\det(\cM_{a,\Theta}(z))\le 0$ for all $z\in \T$ if and only if $ s_{a, \Theta}^+ = 1 $ and $ s_{a, \Theta}^- = 1 $;

\item[(6)] Otherwise (i.e., beyond the above three cases in items (3)--(5)), $s_{a,\Theta} = s_{a,\Theta}^+ + s_{a,\Theta}^->2$.

\end{enumerate}
Since $ \cM_{\pa, \pTh}(z) = -\cM_{\pa, -\pTh}(z)$,
items (3) and (4) are direct consequence of \cite[Theorem~1.4.5]{hanbook} and \cite[Theorem~7]{han15}.
Since $ \det(\cM_{\pa, \pTh}(z)) $ is the product of its two eigenvalues, we know that $\det(\cM_{a,\Theta}(z))\le 0$ for all $z\in \T$ if and only if for all $ z\in \T\bs\sigma(\cM_{\pa, \pTh}) $, $ \nu_+(\cM_{\pa, \pTh}(z)) = \nu_-(\cM_{\pa, \pTh}(z)) = 1 $. Because $\sigma(\cM_{\pa, \pTh}) $ is a finite set, we conclude from Lemma~\ref{lem:maxminT} that this is equivalent to  $ s_{a, \Theta}^+ = s_{a, \Theta}^- = 1 $. This proves item (5).
Hence, items (1) and (2) characterize all the cases for $s_{a,\Theta}=1$, while items (3)--(5) characterizes all the cases for $s_{a,\Theta}=2$.
Note that items (1) and (3)
lead to tight framelet filter bank $\{a; b_1, \ldots, b_{s_{a,\Theta}}\}_\Theta$, and items (2) and (4) lead to tight framelet filter banks $ \{a; b_1, \ldots, b_{s_{a,\Theta}}\}_{-\Theta}$ with $s\in \{1,2\}$.
Items (5) and (6) lead to quasi-tight framelet filter banks which cannot be changed into tight framelet filter banks.

For the special most popular choice of $\pTh(z) = 1$,  according to the above discussion, one of the following four cases must happen:
\begin{enumerate}
\item[(i)] $ \det(\cM_{\pa, 1}(z)) \equiv 0 $ for all $ z \in \T $ if and only if $ s_{a, \Theta}^+ = 1 $ and $ s_{a, \Theta}^- = 0 $;
	
\item[(ii)] $ \det(\cM_{\pa, 1}(z)) \geqslant 0$ for all $ z \in \T $ with $ \det(\cM_{\pa, 1}(z)) \not\equiv 0 $ if and only if $ s_{a, \Theta}^+ = 2 $ and $ s_{a, \Theta}^- = 0 $;
	
\item[(iii)] $ \det(\cM_{\pa, 1}(z)) \leqslant 0$ for all $ z \in \T $ with $ \det(\cM_{\pa, 1}(z)) \not\equiv 0 $ if and only if $ s_{a, \Theta}^+ = 1 $ and $ s_{a, \Theta}^- = 1 $;
	
\item[(iv)] $\det(\cM_{\pa, 1}(z)) $ changes signs on $\T$ if and only if $ s_{a, \Theta}^+ = 2 $ and $ s_{a, \Theta}^- = 1 $.
\end{enumerate}

\section{Examples of Quasi-tight Framelets and Quasi-tight Framelet Filter Banks}

In this section, we provide some examples for quasi-tight framelet filter banks and quasi-tight framelets.
Since tight framelet filter banks have been extensively studied and constructed in the literature, according to our discussion at the end of Section~2,
we only provide examples for cases (5) and (6) in Section~2 (i.e., either $\det(\cM_{a,\Theta}(z))\le 0$ or it changes signs on $\T$) which lead to truly quasi-tight framelet filter banks.

In order to obtain a quasi-tight framelet in $ \Lp{2} $, we have to check the technical condition that the refinable function $ \phi $ (defined in \eqref{phi}) associated with the low-pass filter $ a $ is in $ \Lp{2} $.
Let $ a \in \lp{0} $ with $ \pa(1) = 1 $ and $ m:= \sr(\pa) $, the order of the sum rules of the low-pass filter $a$.
Then we can write $ \pa(z) = (1+z)^m \mathring{\pa}(z) $, where $ \mathring{\pa}(-1)\neq 0 $.
Let $ w\in \lp{0} $ be the sequence determined by $ \pw(z):= \mathring{\pa}(z)\mathring{\pa}^\star(z) $, whose highest and lowest degrees are $ K $ and $ -K $ respectively.
We now recall a technical quantity (e.g., see \cite[(2.0.7)]{hanbook}):
\begin{equation}\label{smaM}
\sm(a):=-\tfrac{1}{2}-\log_{2} \sqrt{\rho(a)},
\end{equation}
where $ \rho(a) $ denotes the spectral radius of the square matrix
$ (w(2j - k))_{-K \leqslant j,k \leqslant K} $.
Let $\phi$ be defined in \eqref{phi}. If $\sm(a)>0$, then $\phi\in \Lp{2}$ and moreover, $\int_{\R} |\wh{\phi}(\xi)|^2 (1+|\xi|^2)^\tau d\xi<\infty$ for all $0\le \tau<\sm(a)$.

The following example shows that for some low-pass filters $a$, one can never obtain a finitely supported tight framelet filter bank, but one can easily construct a quasi-tight framelet filter bank.

\begin{example} \label{ex:ThetaVM1}
{\rm
Consider a low-pass filter $ a $ given by
\[
\pa(z) = \tfrac{1}{26}z^{-2}(z+1)(z^2 - z + 1)(9z^2-5z+9).
\]
Note that $ |\pa(e^{-i2\pi/3})| = \frac{14}{13} > 1 $ and $ |\pa(e^{-i2\pi/3})| \not\in\{ 2^j\setsp j \in \N \} $. By \cite[Proposition 4.4]{HanMo:2005Symmetric}, there does not exist a (rational) Laurent polynomial $ \Theta $ with real coefficients such that $ \cM_{a,\Theta}(z)\ge 0$ for all $ z\in \T $.
Therefore, using Oblique Extension Principle, one cannot construct a real-valued tight framelet filter bank from such low-pass filter $a$.
Note that $\sr(\pa) =1$ and $\vmo(1-\pa\pa^\star)=2$.
Taking $ \Theta(z) = 1 $ and $ n_b = 1 $, we see from Figure~\ref{fig:ThetaVM1} that $ \det(\cM_{\pa, 1}(z))$ changes signs on $\T$. Hence, $ s_{a, \Theta}^+ = 2 $ and $ s_{a, \Theta}^- = 1 $.
We have a quasi-tight framelet filter bank $ \{a; b_1, b_2, b_3\}_{\Theta, (1, 1, -1)} $ as follows:
\begin{align*}
\pb_1(z) &= \tfrac{1}{52}z^{-2}(z-1)(63z^4+28z^3+100z^2+28z+63),\\
\pb_2(z) &= \tfrac{\sqrt{2}}{97344}z^{-2}(z-1)(9z^2+4z+9)(3645z^4 - 1034z^2+3645), \\
\pb_3(z) &= \tfrac{\sqrt{2}}{97344}z^{-2}(z-1)(9z^2+4z+9)(3645z^4 + 9782z^2+3645),
\end{align*}
with $ \vmo(b_1) = \vmo(b_2) = \vmo(b_3) = 1 $. Since $ \sm(a) \approx 0.7693 $, the refinable function $ \phi $ defined in \eqref{phi} belongs to $ \Lp{2} $.
Therefore, $\{\phi, \phi; \psi^1,\psi^2,\psi^3\}_{(1,1,-1)}$ is a quasi-tight framelet in $\Lp{2}$ and
$\{\psi^1, \psi^2, \psi^3\}_{(1, 1, -1)} $ is a homogeneous quasi-tight framelet in $\Lp{2}$, where $ \psi^1, \psi^2, \psi^3 $ are defined in \eqref{eta:psi} and have at least one vanishing moment.
}\end{example}

\begin{figure}[hbt]
\centering
\begin{subfigure}[]{0.18\textwidth}			 \includegraphics[width=\textwidth, height=0.8\textwidth]{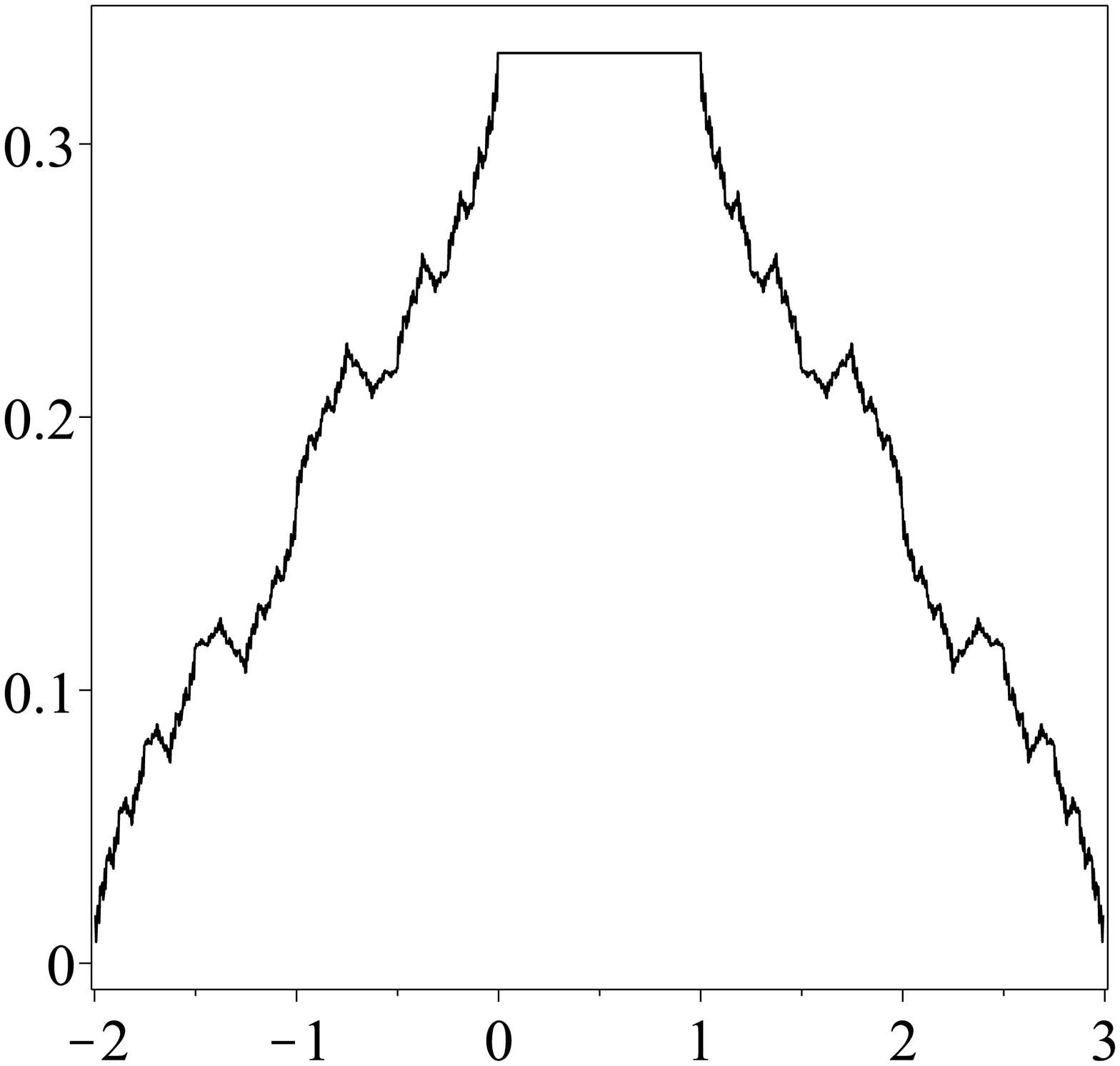}
\caption{$\phi$}\end{subfigure}
\begin{subfigure}[]{0.18\textwidth}			 \includegraphics[width=\textwidth, height=0.8\textwidth]{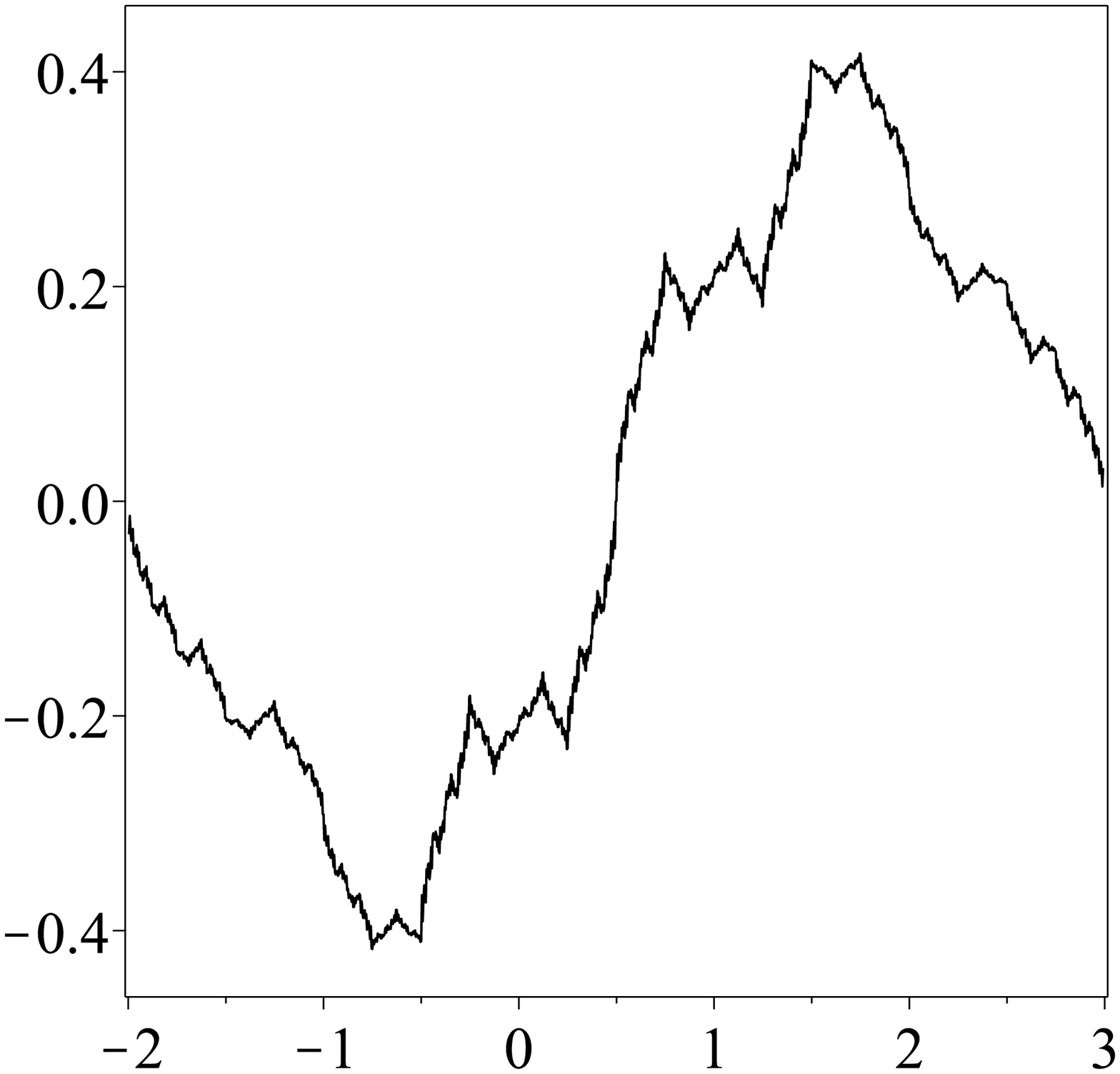}	
\caption{$\psi^1$}\end{subfigure}
\begin{subfigure}[]{0.18\textwidth}			 \includegraphics[width=\textwidth, height=0.8\textwidth]{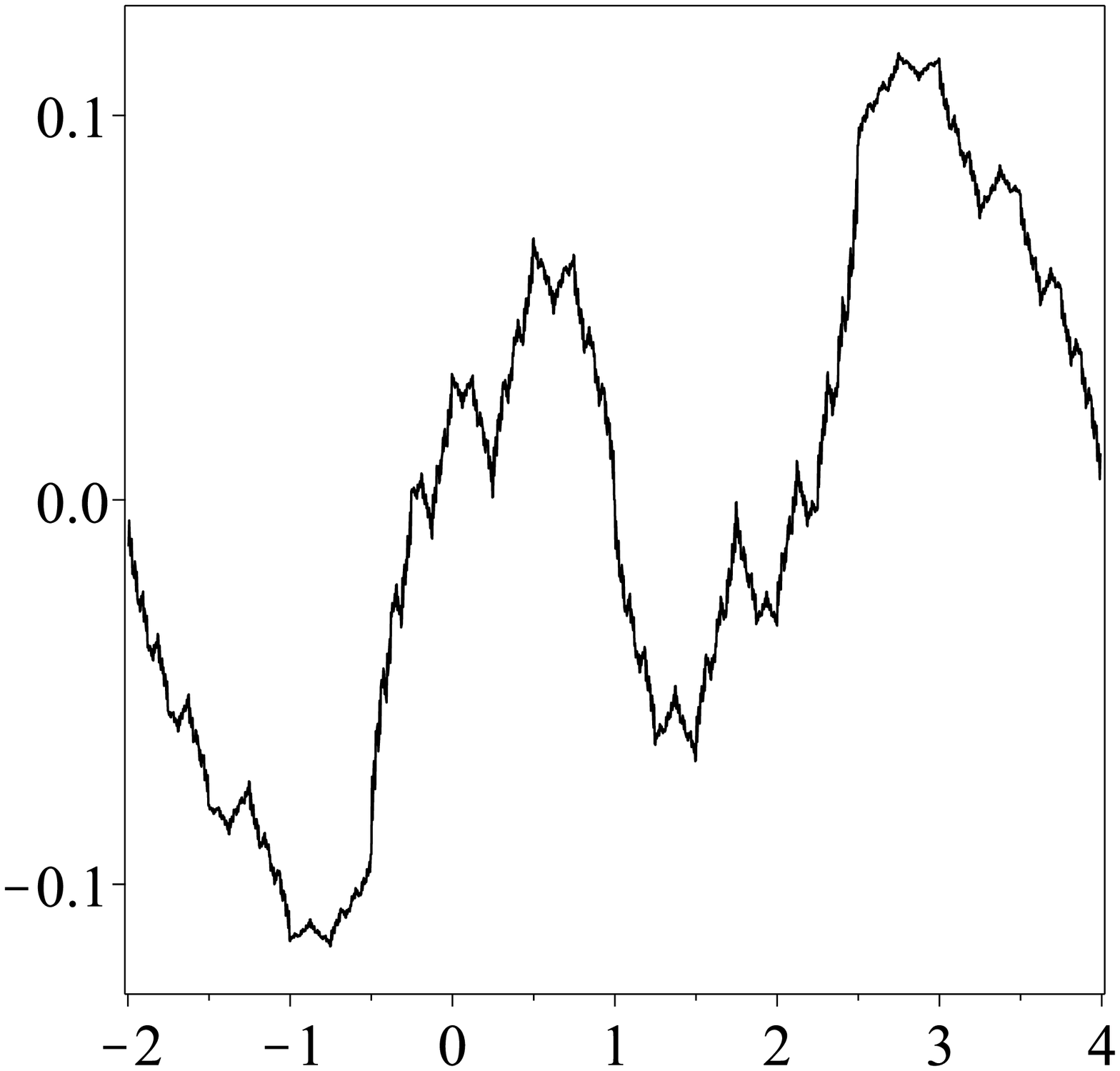}
\caption{$\psi^2$}
\end{subfigure} \begin{subfigure}[]{0.18\textwidth}
\includegraphics[width=\textwidth, height=0.8\textwidth]{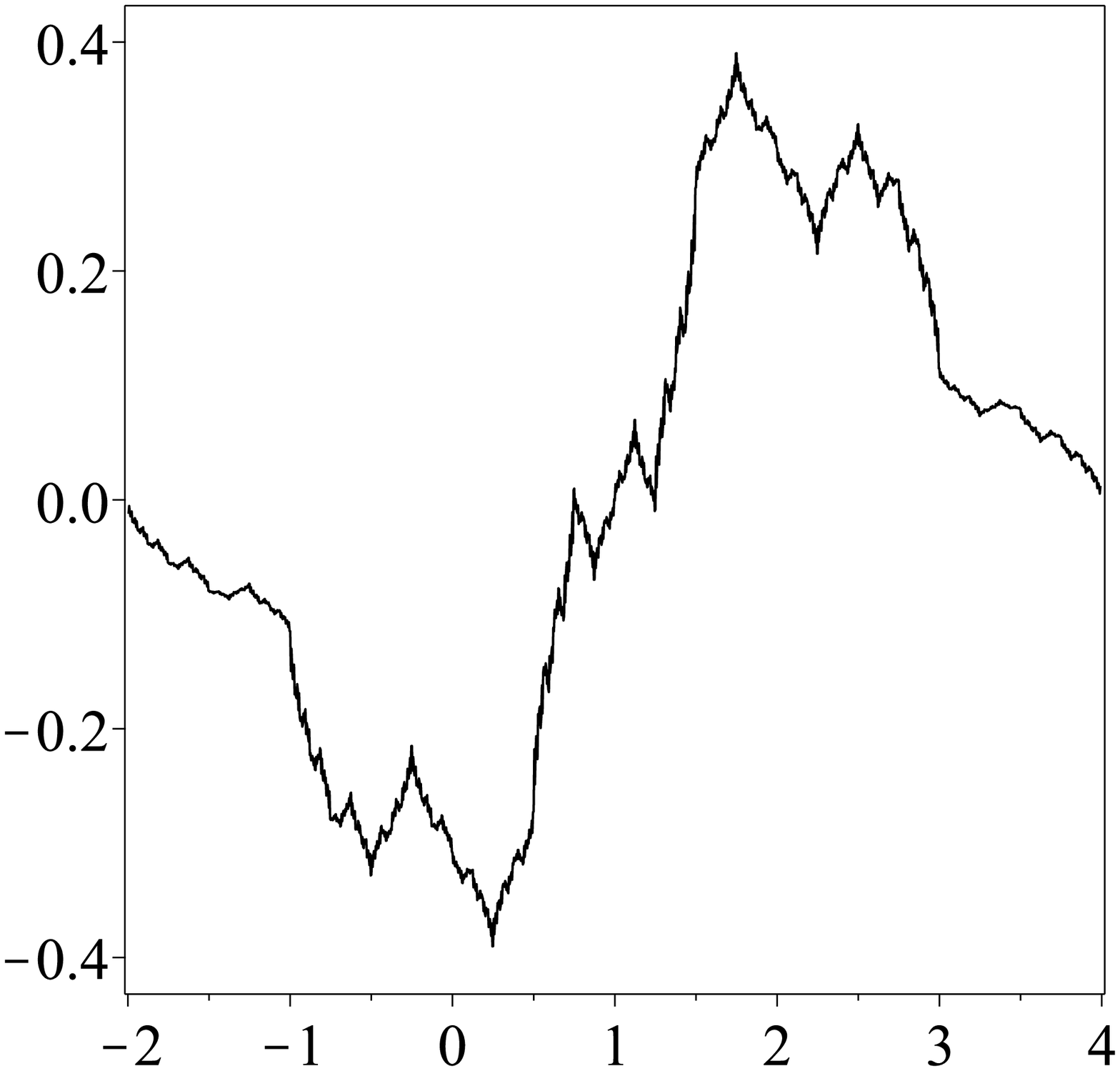}
\caption{$ \psi^3 $ }
\end{subfigure}
\begin{subfigure}[]{0.18\textwidth}
\includegraphics[width=\textwidth, height=0.8\textwidth]{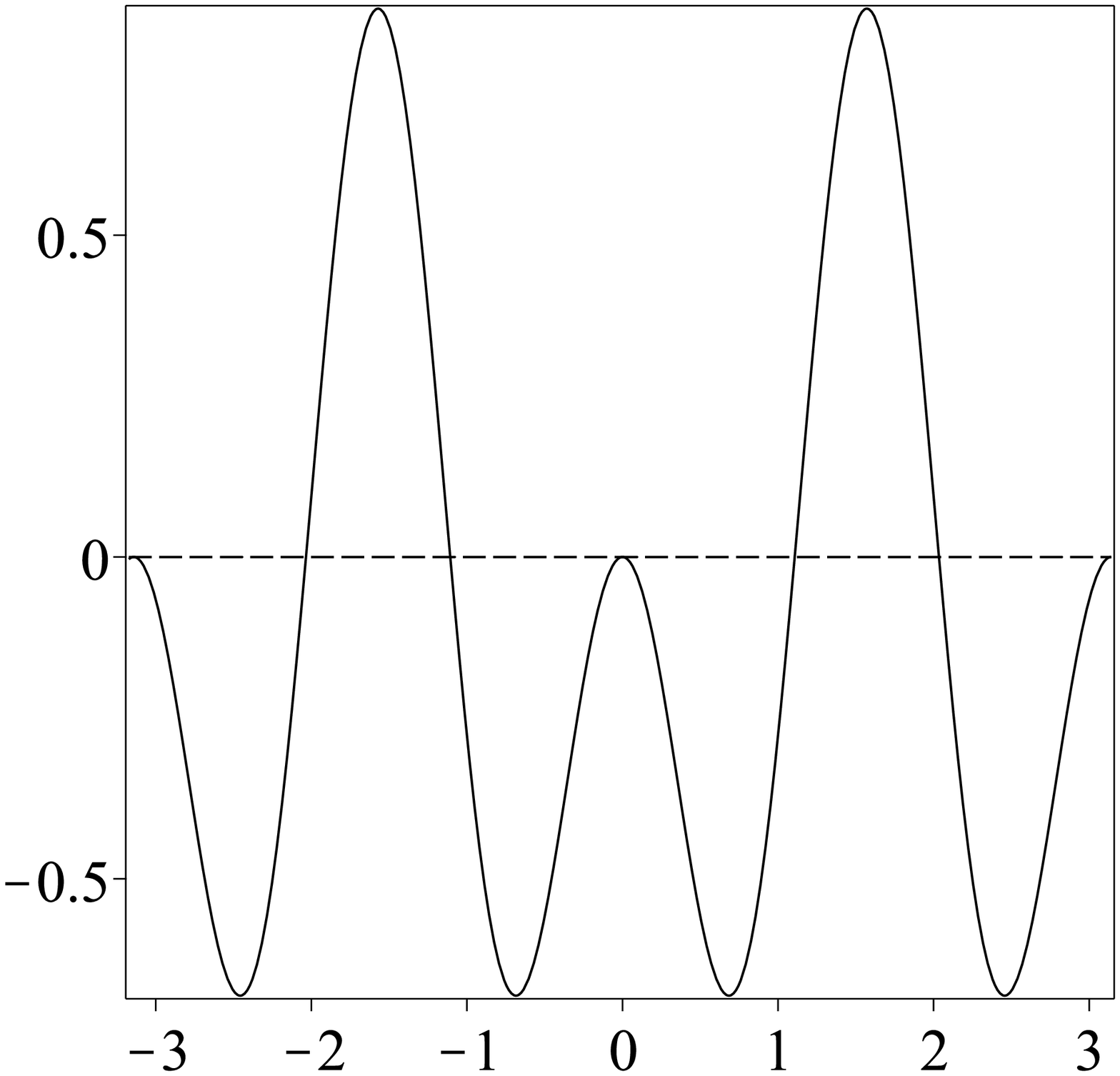}
\caption{$\det(\cM_{\pa,1})$}\end{subfigure}
\caption{
The quasi-tight framelet $\{\phi,\phi; \psi^1,\psi^2,\psi^3\}_{(1,1,-1)}$ and the homogeneous quasi-tight framelet $\{\psi^1,\psi^2,\psi^3\}_{(1,1,-1)}$ in $\Lp{2}$ obtained in Example~\ref{ex:ThetaVM1}.
(A) is the refinable function $\phi\in \Lp{2}$. (B) --(D) are the framelet functions $\psi^1$, $ \psi^2 $ and $ \psi^3$.
(E) is $\det(\cM_{\pa, 1}(e^{-i\xi}))$ for $ \xi\in  [-\pi, \pi] $, where the dashed line is the horizontal axis.
}\label{fig:ThetaVM1}
\end{figure}

\begin{example} \label{ex:ThetaInterp}
{\rm Consider $\Theta(z)=\tfrac{1}{2}(z+\tfrac{1}{z}) $ and the interpolatory low-pass filter
\[
\pa(z) = \tfrac{1}{2} + \tfrac{3}{8} (z + z^{-1}) - \tfrac{1}{8}(z^3 + z^{-3}).
\]
We see from Figure~\ref{fig:ThetaInterp} that
$\det(\cM_{\pa, \pTh}(z)) \leqslant 0$ for all $ z\in \T $. Therefore, $ s_{a, \Theta}^+ = s_{a, \Theta}^- = 1 $.
	Note that $ \sr(\pa) = 2 $ and
	$ \vmo(\pTh(z) - \pTh(z^2)\pa(z)\pa^\star(z)) = 4 $.
	Hence, the maximum order of vanishing moments is two. Taking $ n_b = 2 $, we obtain a quasi-tight framelet filter bank $ \{a; b_1, b_2\}_{\Theta, (1, -1)} $ as follows:
\begin{align*}
	\pb_1(z) &= \tfrac{1}{32}z^{-3}(z-1)^2 (z^6+2z^5-4z^4-14z^3-23z^2-16z-8), \\
	\pb_2(z) &= -\tfrac{1}{32}z^{-3}(z-1)^2 (z^4+2z^3+4z^2+2z+9),
	\end{align*}
	with $ \vmo(b_1) = \vmo(b_2) = 2 $. Since $ \sm(a) = 1 $,
the refinable function $ \phi $ defined in \eqref{phi} belongs to $ \Lp{2} $.
Define $\tilde{\eta}:=(\phi(\cdot+1)+\phi(\cdot-1))/2$.
Therefore, $\{\phi,\tilde{\eta};\psi^1,\psi^2\}_{(1,-1)}$ a quasi-tight framelet in $\Lp{2}$ and
$ \{\psi^1, \psi^2\}_{(1, -1)} $ is a homogeneous quasi-tight framelet in $ \Lp{2} $, where $ \psi^1, \psi^2 $ are defined in \eqref{eta:psi} and have at least two vanishing moments.
}\end{example}

\begin{figure}[htb!]
\centering
\begin{subfigure}[]{0.18\textwidth}			 \includegraphics[width=\textwidth, height=0.8\textwidth]{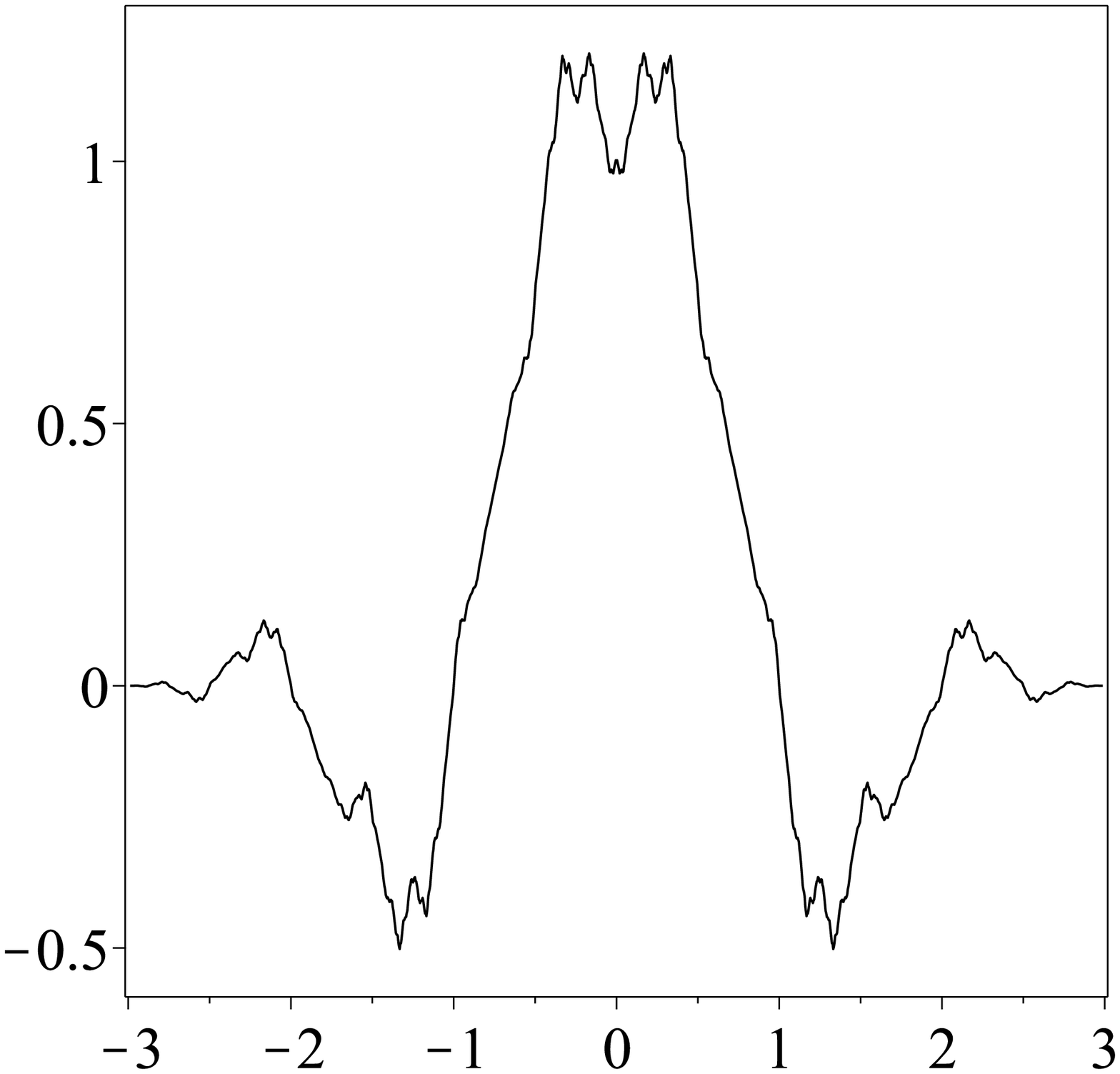}
\caption{$\phi$}
\end{subfigure}		
\begin{subfigure}[]{0.18\textwidth}			 \includegraphics[width=\textwidth, height=0.8\textwidth]{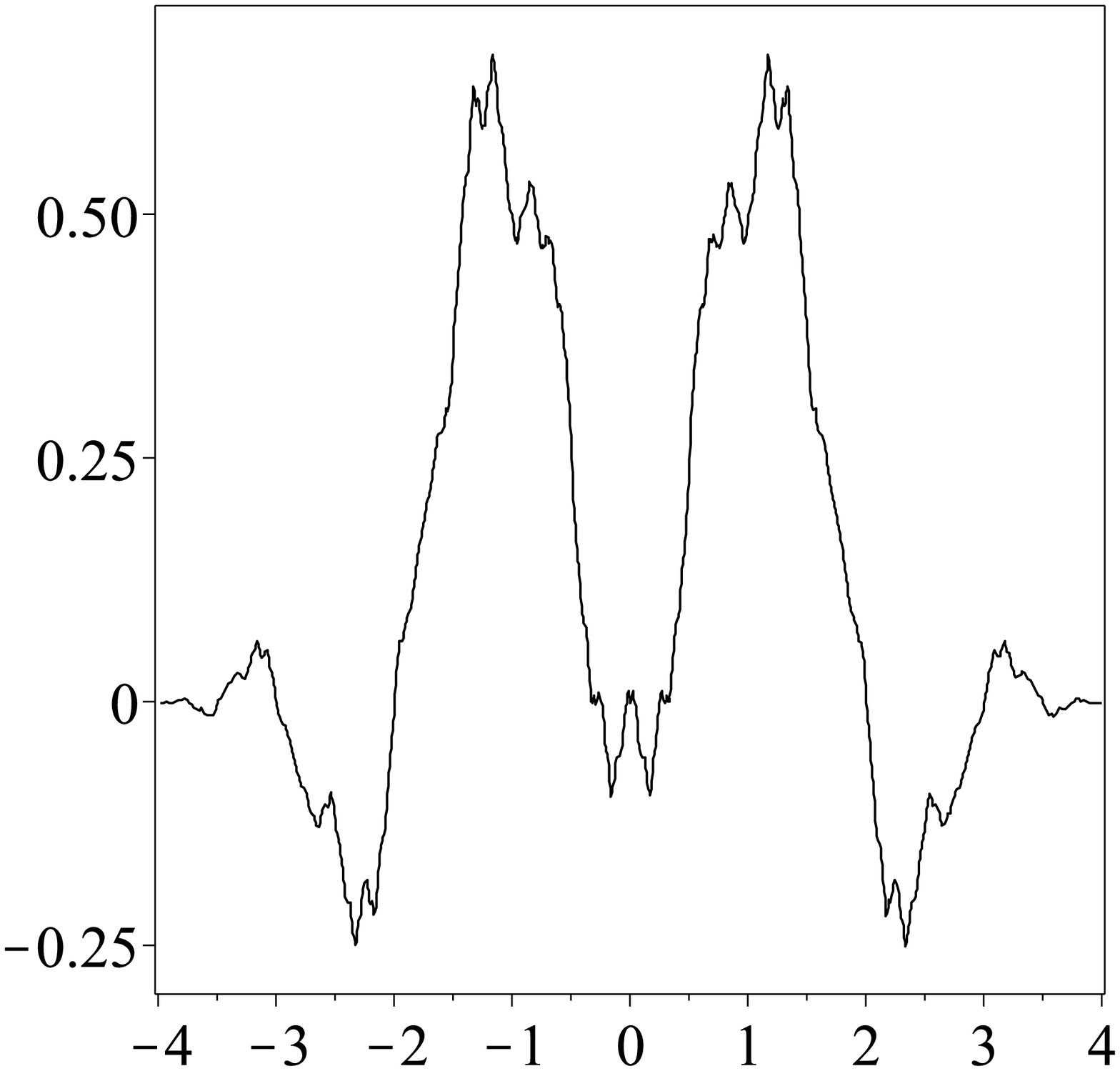}
\caption{$\tilde{\eta}$}
\end{subfigure}	
 \begin{subfigure}[]{0.18\textwidth}
\includegraphics[width=\textwidth, height=0.8\textwidth]{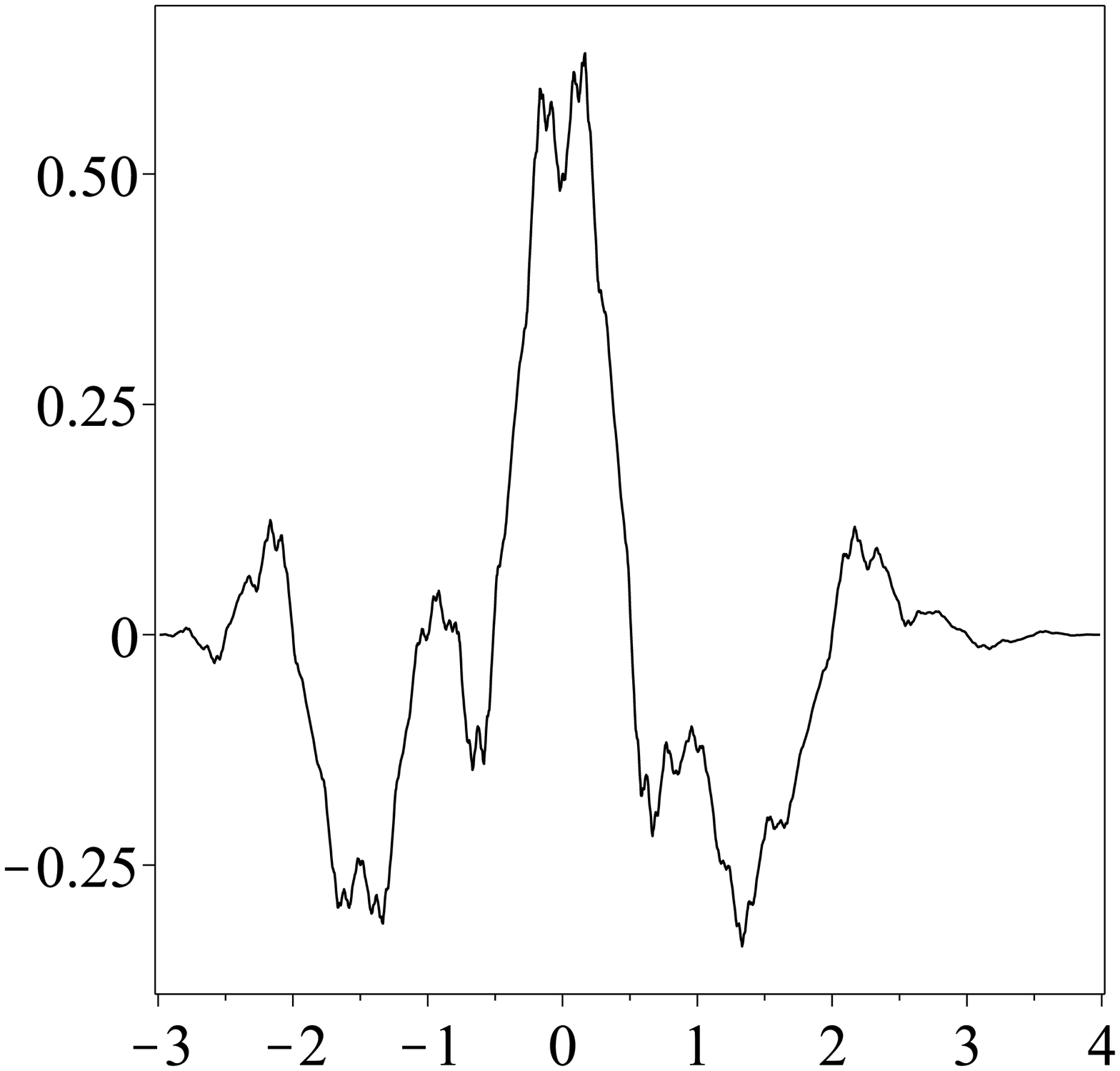}
			\caption{$\psi^1$}
		\end{subfigure}	 \begin{subfigure}[]{0.18\textwidth}
\includegraphics[width=\textwidth, height=0.8\textwidth]{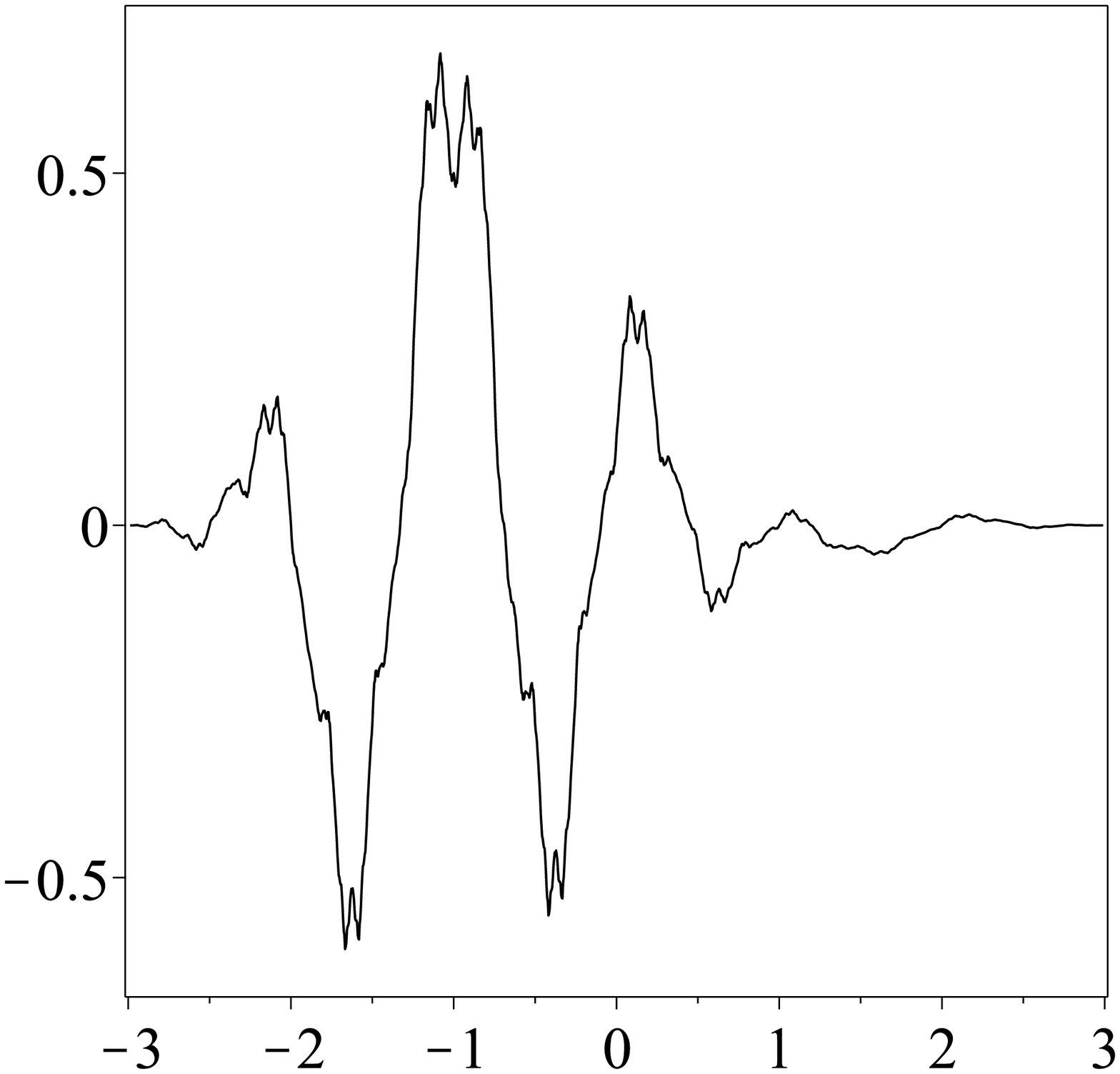}
			\caption{$\psi^2$}
		\end{subfigure}		 \begin{subfigure}[]{0.18\textwidth}
\includegraphics[width=\textwidth, height=0.8\textwidth]{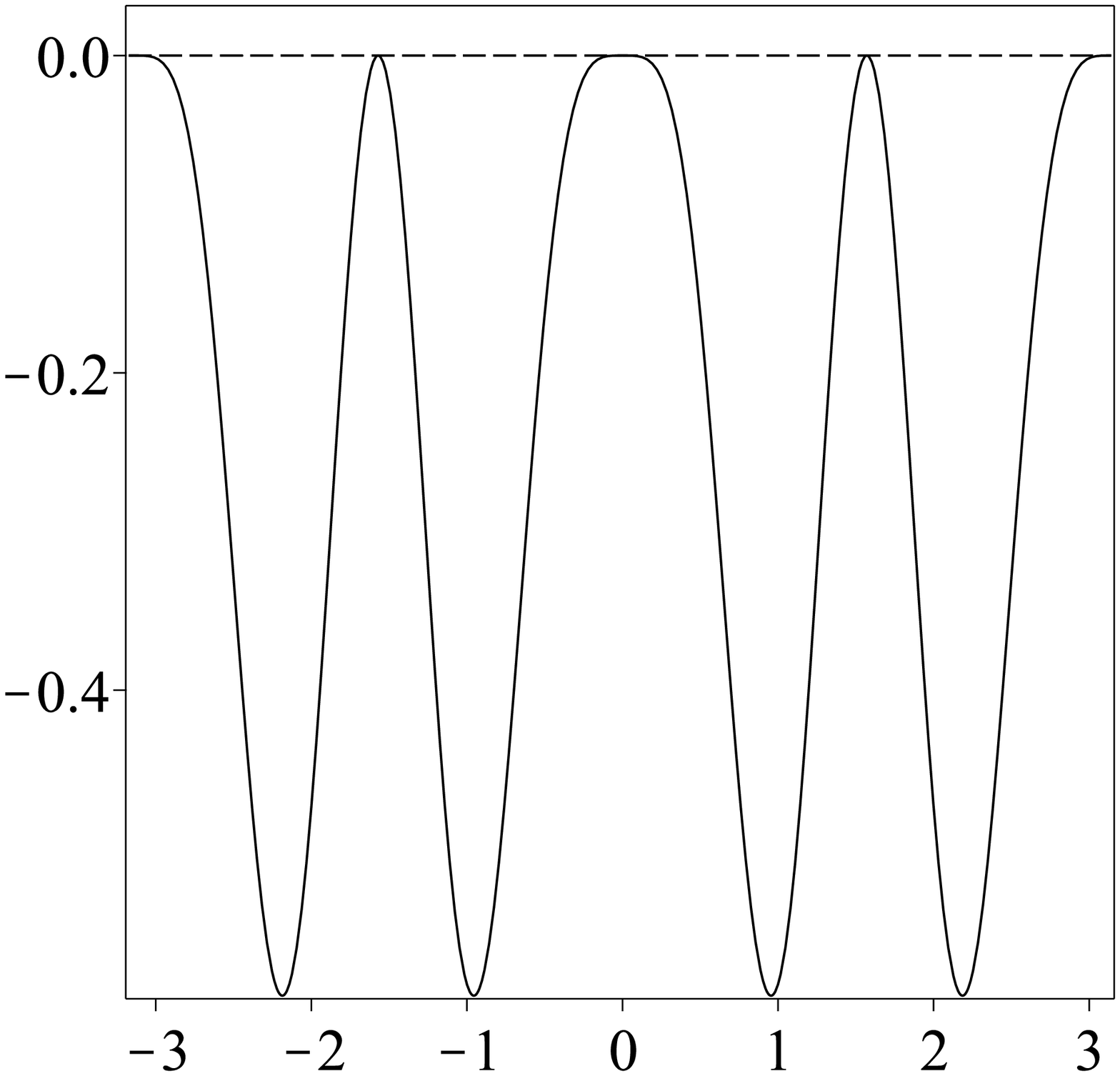}
			\caption{$ \det(\cM_{\pa,\pTh}) $ }
		\end{subfigure}
		\caption{
The quasi-tight framelet $\{\phi,\tilde{\eta}; \psi^1,\psi^2\}_{(1,-1)}$ and the homogeneous quasi-tight framelet $\{\psi^1,\psi^2\}_{(1,-1)}$ in $\Lp{2}$ obtained in Example~\ref{ex:ThetaInterp}.
(A) is the refinable function $\phi\in \Lp{2}$.
(B) is the function $\tilde{\eta}:=(\phi(\cdot+1)+\phi(\cdot-1))/2$.
(C) and (D) are the framelet functions $\psi^1$ and $ \psi^2 $.
(E) is $\det(\cM_{\pa, 1}(e^{-i\xi}))$ for $ \xi\in  [-\pi, \pi] $.
}
\label{fig:ThetaInterp}
\end{figure}

\begin{example} \label{ex:Two32}
{\rm Consider $\pTh(z) = 1$ and the low-pass filter
\[
\pa(z) = -\tfrac{1}{16}z^{-2} (z+1)^3
	( z^2 - 4z + 1 ).
\]
We see from Figure~\ref{fig:Two32} that $\det(\cM_{\pa, 1}(z))\leqslant 0 $ for all $ z\in \T $. Hence, $ s_{a, \Theta}^+ = s_{a, \Theta}^- = 1 $.
Note that $ \sr(\pa)=3 $ and $ \vmo(1-\pa\pa^\star)=2 $.
	Taking $n_b=1$, we obtain a quasi-tight framelet filter bank
	$\{\ta; \tb_1, \tb_2\}_{\Theta, \{1, -1\}}$ as follows:
\begin{align*}
\pb_1(z) =& \tfrac{-\sqrt{2}}{512z^2}
(z-1)(16z^5-271z^4+16z^3+16z^2-1), \\
\pb_2(z) =& \tfrac{\sqrt{2}}{512z^2}
(z-1)(16z^5+241z^4+16z^3+16z^2-1),
\end{align*}
with $ \vmo(\tb_1) =  \vmo(\tb_2) = 1 $.
Since $\sm(a) \approx 1.4408$, the refinable function $ \phi $ defined in \eqref{phi} belongs to $ \Lp{2} $. Therefore, $\{\phi,\phi; \psi^1,\psi^2\}_{(1,-1)}$ is a quasi-tight framelet in $\Lp{2}$ and
$\{\psi^1, \psi^2\}_{(1, -1)} $ is a homogeneous quasi-tight framelet in $\Lp{2} $, where $ \psi^1, \psi^2 $ are defined in \eqref{eta:psi} and have one vanishing moment.
} \end{example}

\begin{figure}[h!]
\centering		 \begin{subfigure}[]{0.24\textwidth}		 \includegraphics[width=\textwidth, height=0.8\textwidth]{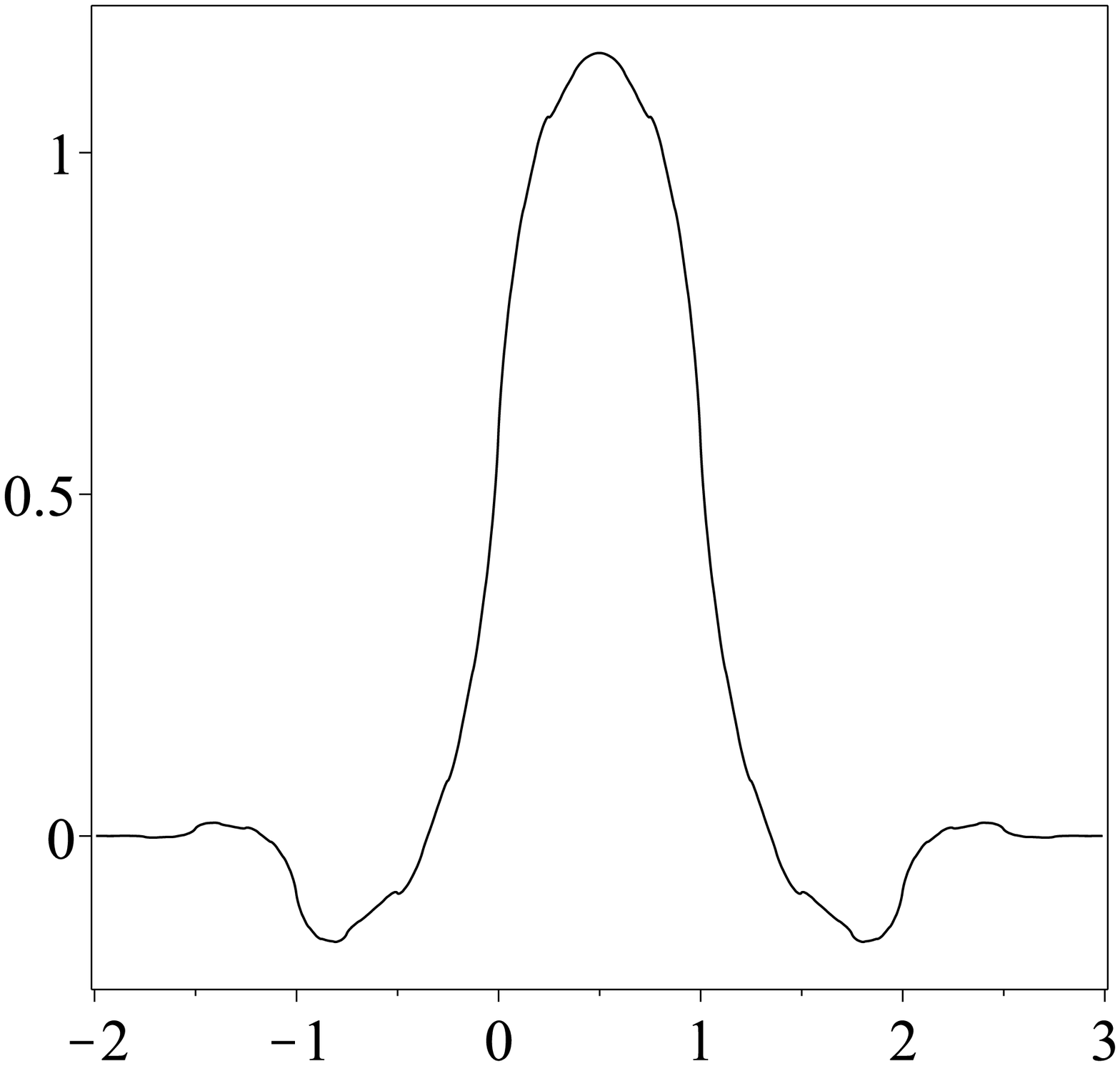}
\caption{$\phi$}
\end{subfigure}
\begin{subfigure}[]{0.24\textwidth}			 \includegraphics[width=\textwidth, height=0.8\textwidth]{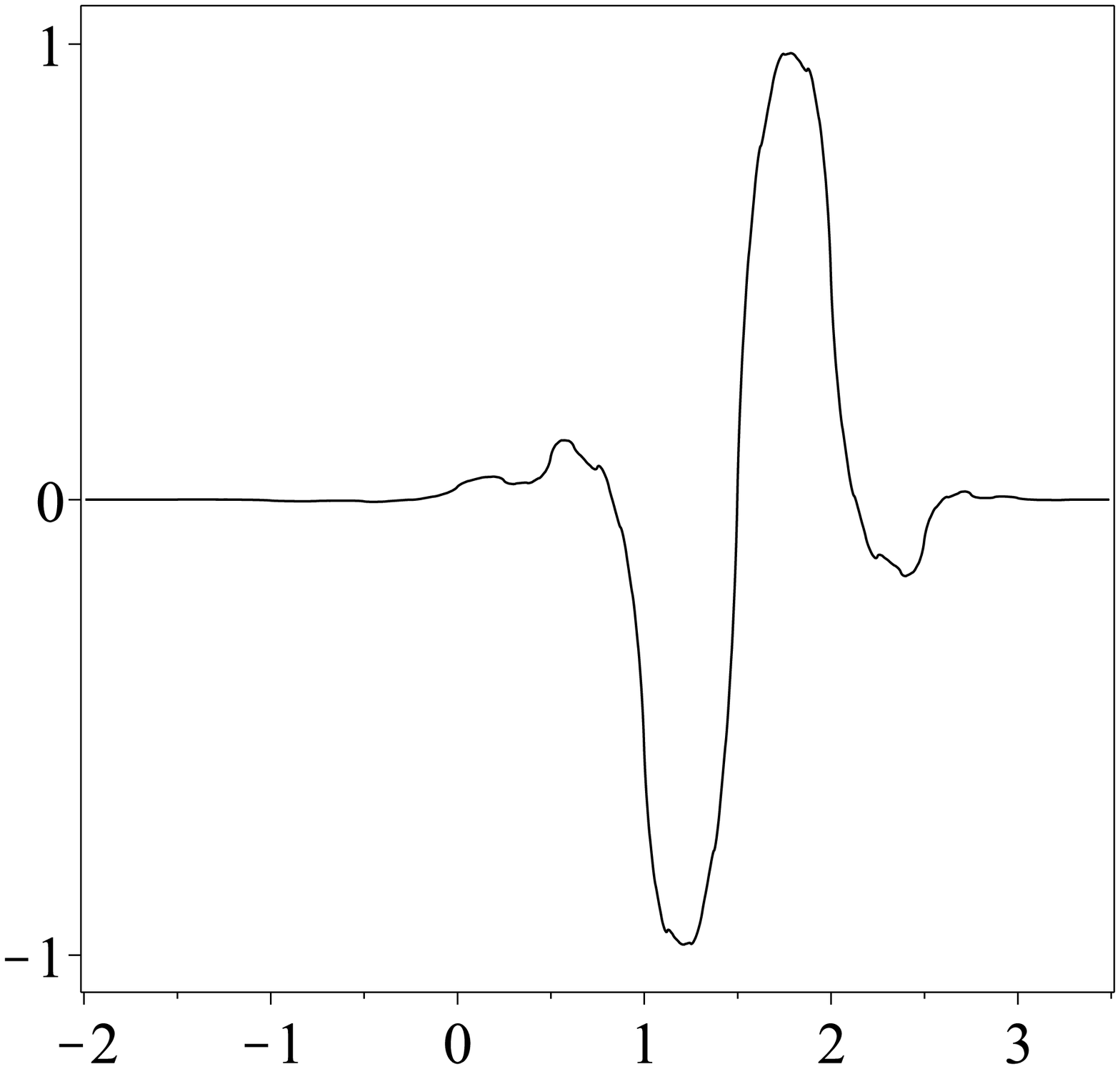}
\caption{$\psi^1$}
\end{subfigure}		 \begin{subfigure}[]{0.24\textwidth}
\includegraphics[width=\textwidth, height=0.8\textwidth]{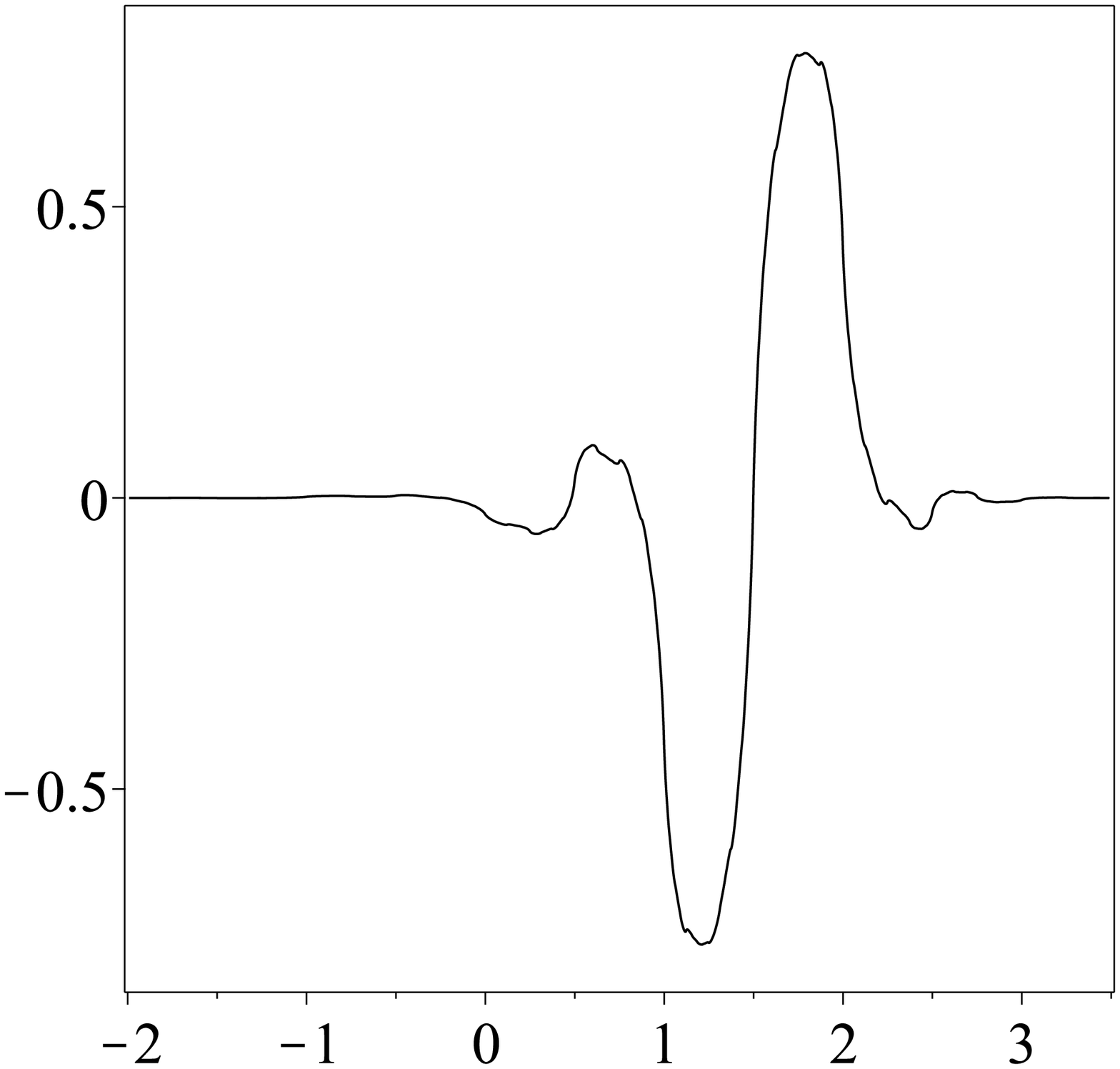}
\caption{$\psi^2$}
\end{subfigure}		 \begin{subfigure}[]{0.24\textwidth}
\includegraphics[width=\textwidth, height=0.8\textwidth]{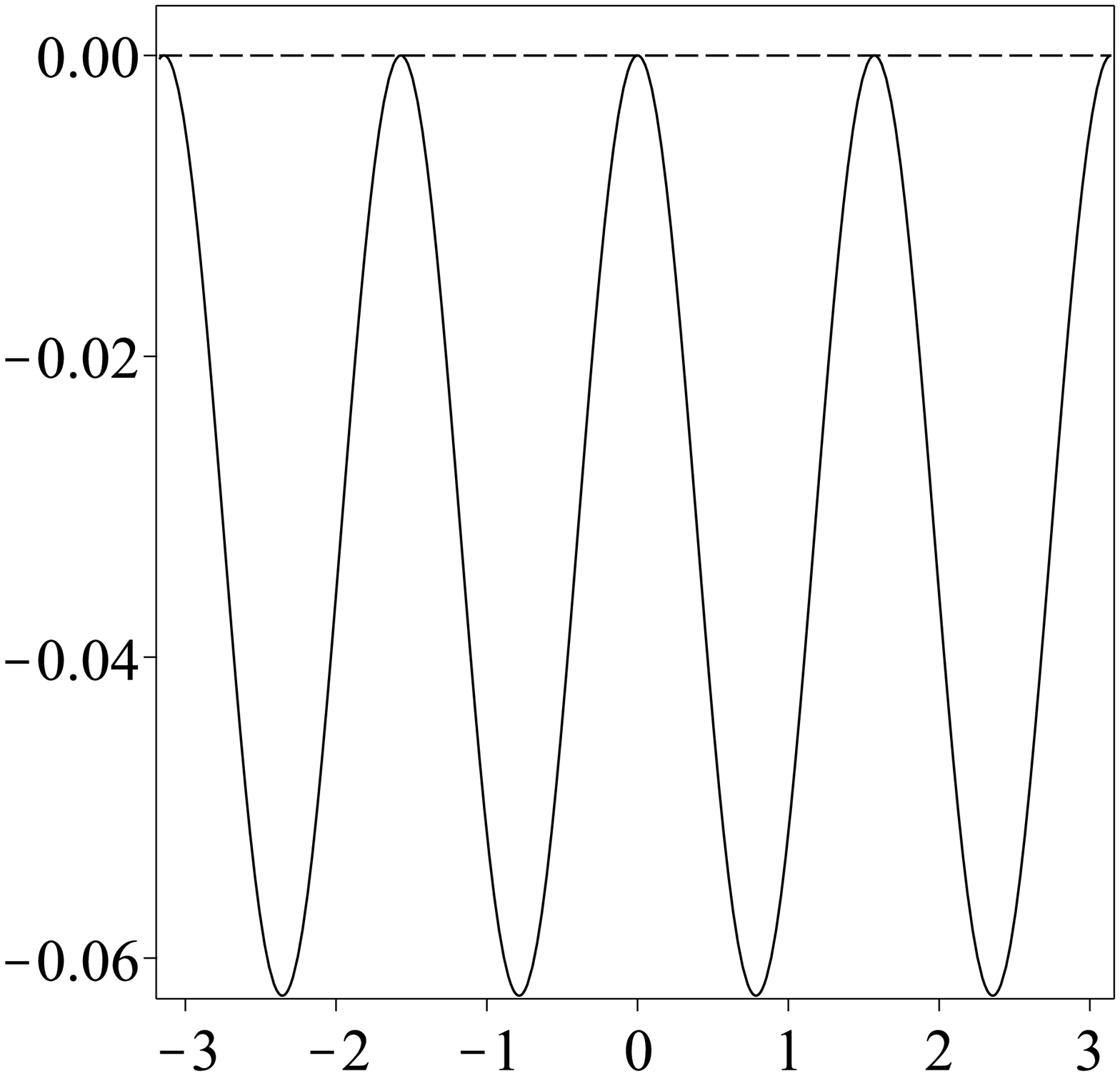}
\caption{$ \det(\cM_{\pa,1}(e^{-i\xi})) $ }
\end{subfigure}
\caption{
The quasi-tight framelet $\{\phi,\phi; \psi^1,\psi^2\}_{(1,-1)}$ and the homogeneous quasi-tight framelet $\{\psi^1,\psi^2\}_{(1,-1)}$ in $\Lp{2}$ obtained in Example~\ref{ex:Two32}.
(A) is the refinable function $\phi\in \Lp{2}$.
(B) and (C) are the framelet functions $\psi^1$ and $ \psi^2 $.
(D) is $\det(\cM_{\pa, 1}(e^{-i\xi}))$ for $ \xi\in  [-\pi, \pi] $.
}\label{fig:Two32}
\end{figure}

\begin{example} \label{ex:Two44}
{\rm	Consider $\pTh(z) = 1$ and the low-pass filter
	$$ \pa(z) = \tfrac{1}{64}z^{-4}(z+1)^4 (z^4 - 6z^3 + 14 z^2 - 6z + 1). $$
We see from Figure~\ref{fig:Two44} that $ \det(\cM_{\pa, 1}(z))\leqslant 0 $ for all $ z\in \T $. Hence, $ s_{a, \Theta}^+ = s_{a, \Theta}^- = 1 $.
Note that $ \sr(\pa)=4 $ and $ \vmo(1-\pa\pa^\star)=4 $.
	Hence, the maximum order of vanishing moments is two.
Taking $n_b=2$, we obtain a quasi-tight framelet filter bank
	$\{\ta; \tb_1, \tb_2\}_{\Theta, \{1, -1\}}$ as follows:
\begin{align*}
	\pb_1(z) =& \tfrac{\sqrt{2}}{16z^2}(z-1)^2
	\left( (2+\sqrt{3})z^4 +2-\sqrt{3} \right), \\
	\pb_2(z) =& \tfrac{1}{64z^4} (z-1)^2
		(z^6+11z^4+8z^3+11z^2+1),
\end{align*}
with $ \vmo(\tb_1) =  \vmo(\tb_2) = 2 $.
Since $\sm(a) \approx 1.6297$,  the refinable function $ \phi $ defined in \eqref{phi} belongs to $ \Lp{2} $. Therefore,
$\{\phi,\phi;\psi^1,\psi^2\}_{(1,-1)}$ is a quasi-tight framelet in $\Lp{2}$ and
$ \{\psi^1, \psi^2\}_{(1, -1)} $ is a homogeneous quasi-tight framelet in $ \Lp{2} $, where $ \psi^1, \psi^2 $ are defined in \eqref{eta:psi} and have two vanishing moments.
} \end{example}

\begin{figure}[htb!]
		\centering
\begin{subfigure}[]{0.24\textwidth}		 \includegraphics[width=\textwidth, height=0.8\textwidth]{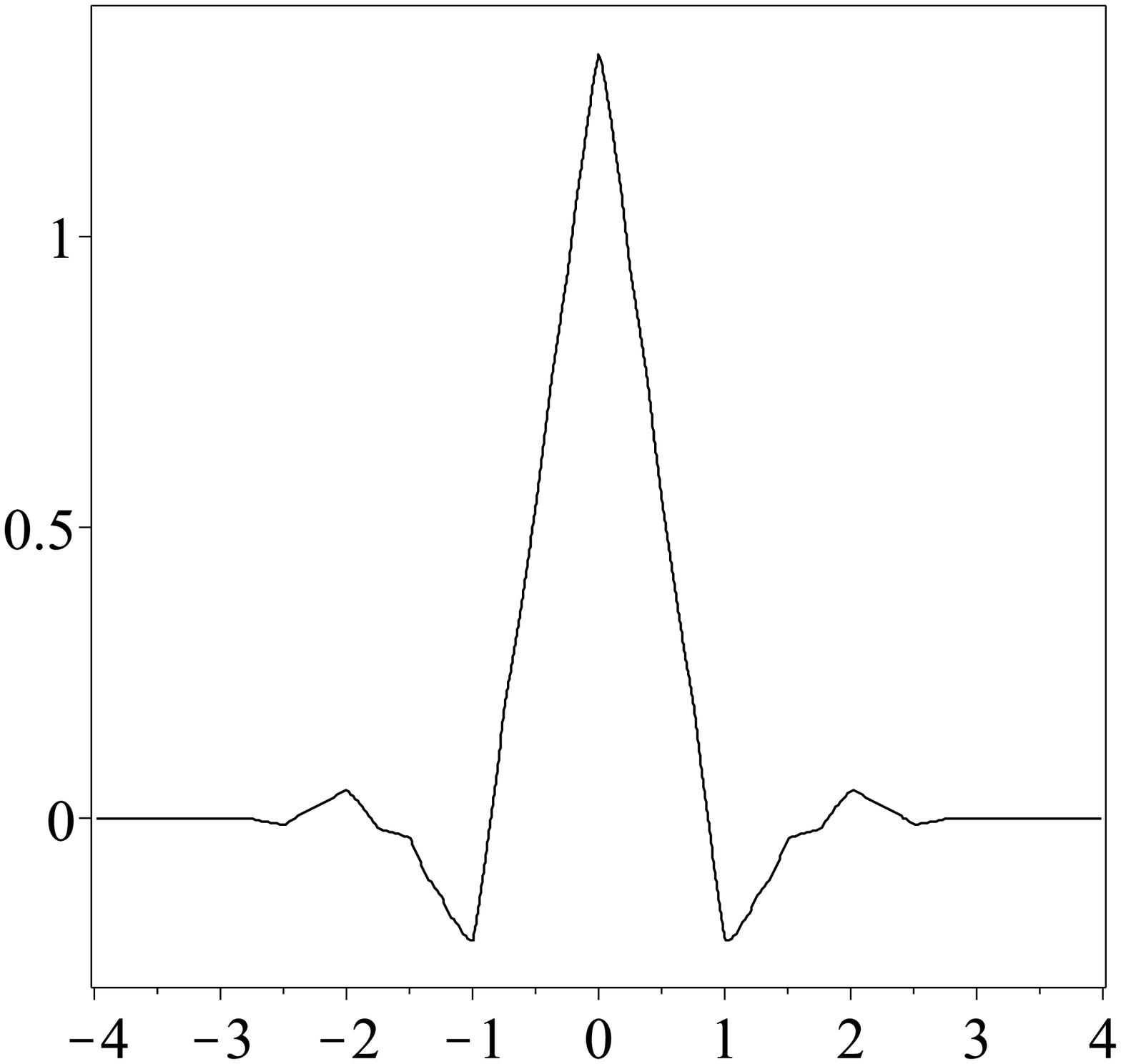}
			\caption{$\phi$}
		\end{subfigure}		 \begin{subfigure}[]{0.24\textwidth}
\includegraphics[width=\textwidth, height=0.8\textwidth]{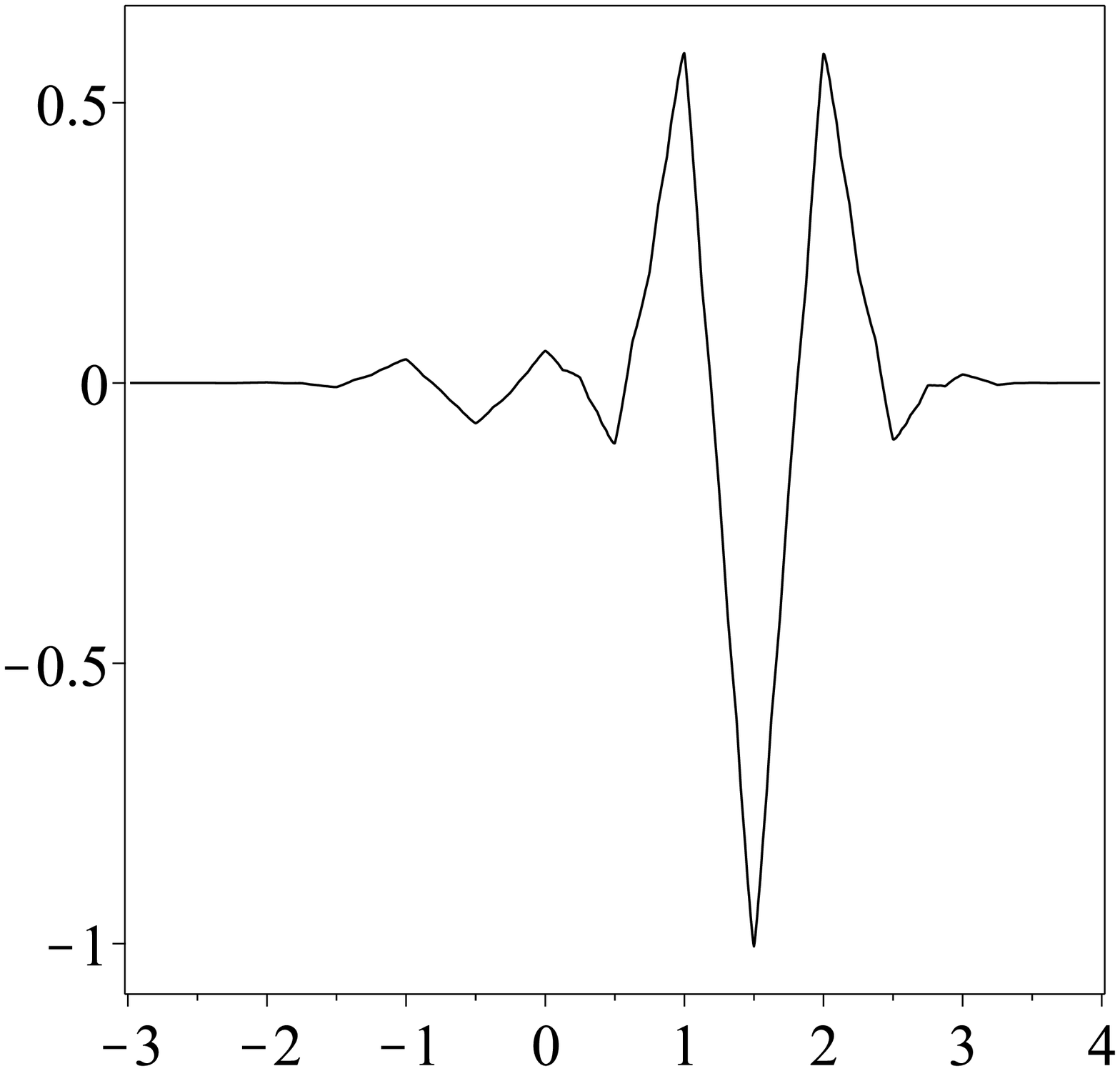}
			\caption{$\psi^1$}
		\end{subfigure}		 \begin{subfigure}[]{0.24\textwidth}
\includegraphics[width=\textwidth, height=0.8\textwidth]{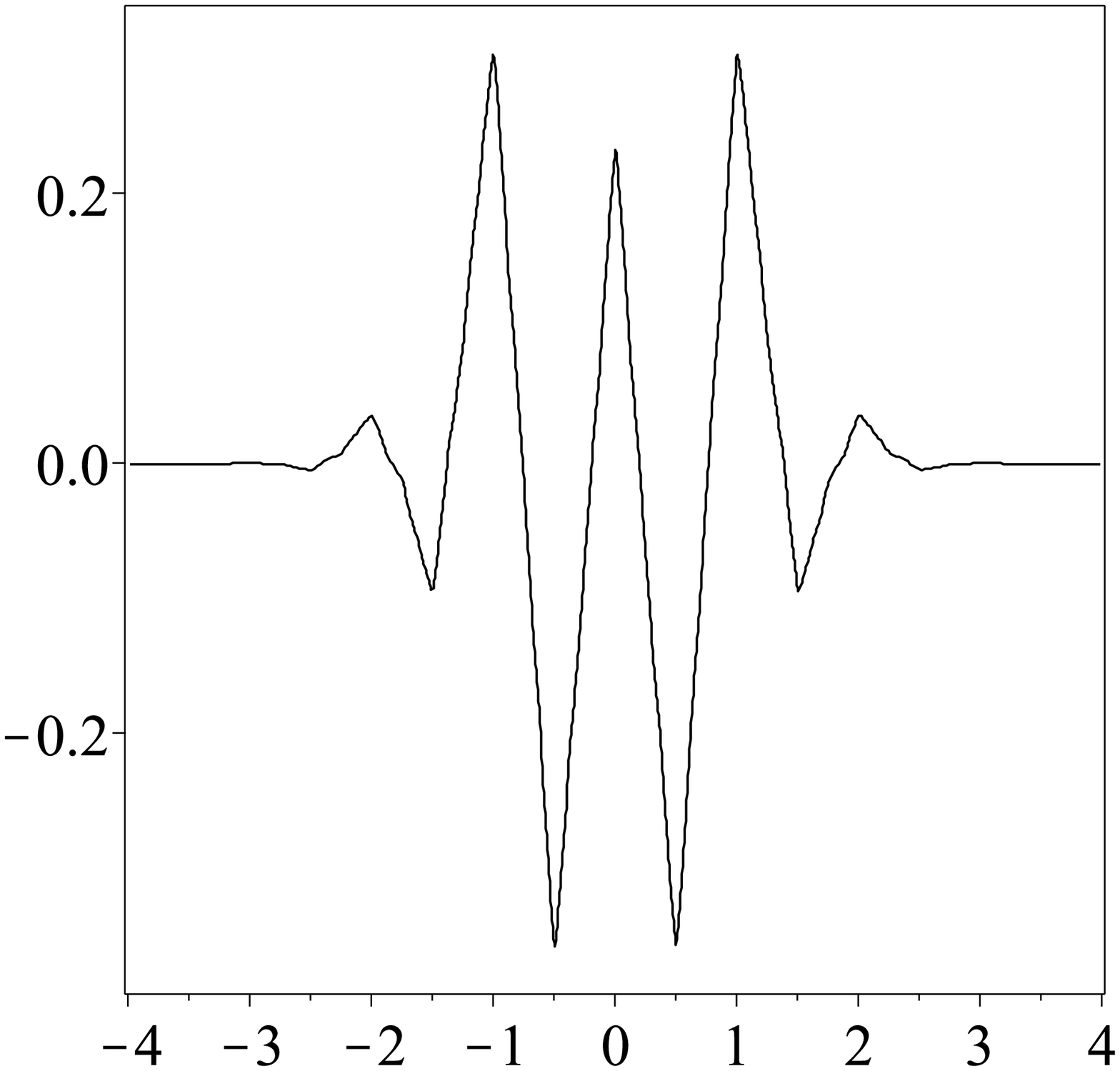}
			\caption{$\psi^2$}
		\end{subfigure}		 \begin{subfigure}[]{0.24\textwidth}
\includegraphics[width=\textwidth, height=0.8\textwidth]{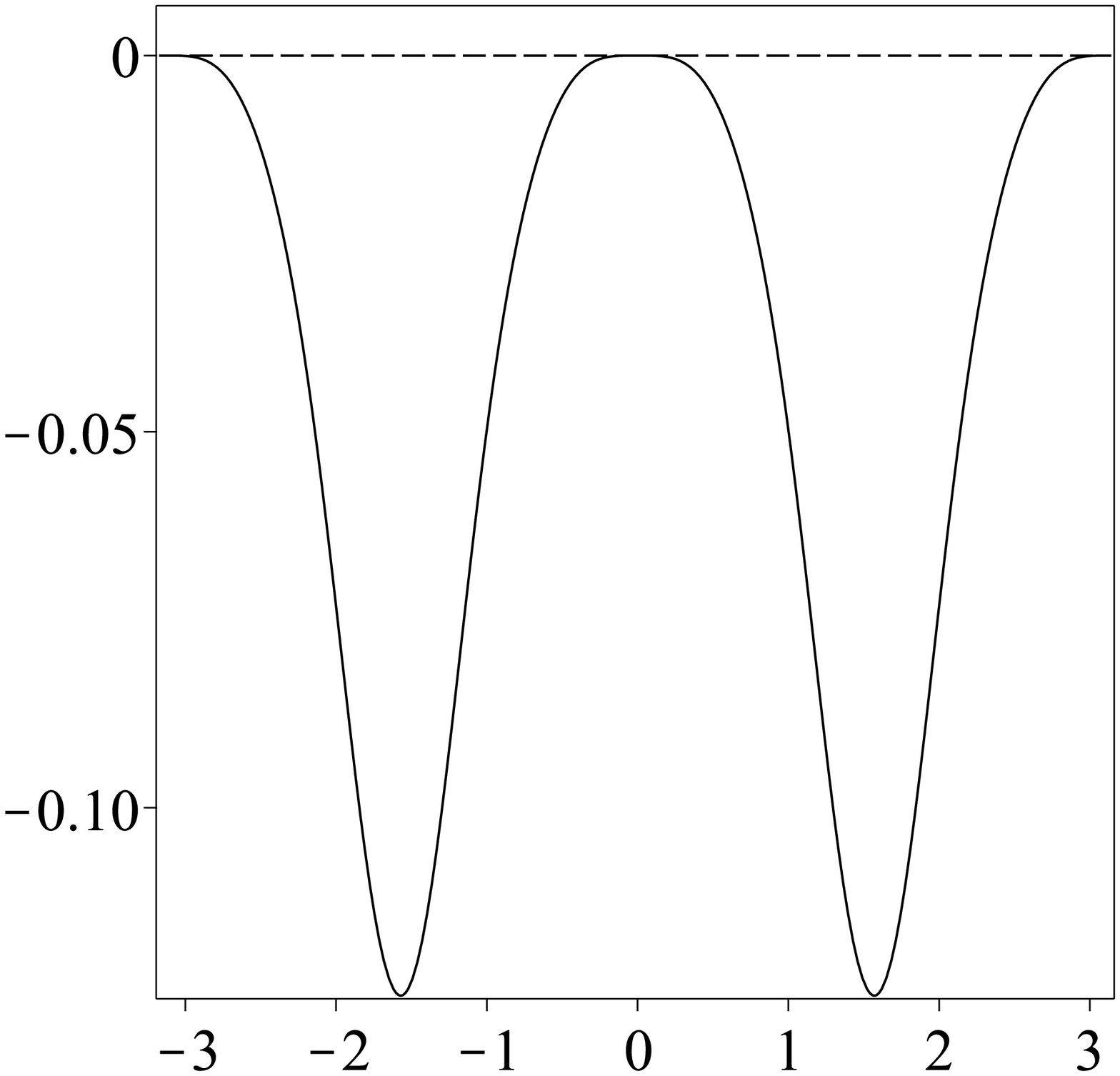}
			\caption{$ \det(\cM_{\pa,1}(e^{-i\xi})) $}
		\end{subfigure}
\caption{
The quasi-tight framelet $\{\phi,\phi; \psi^1,\psi^2\}_{(1,-1)}$ in $\Lp{2}$ and the homogeneous quasi-tight framelet $\{\psi^1,\psi^2\}_{(1,-1)}$ in $\Lp{2}$ obtained in Example~\ref{ex:Two44}.
(A) is the refinable function $\phi\in \Lp{2}$.
(B) and (C) are the framelet functions $\psi^1$ and $ \psi^2 $.
(D) is $\det(\cM_{\pa, 1}(e^{-i\xi}))$ for $ \xi\in  [-\pi, \pi] $.
}\label{fig:Two44}
\end{figure}

\begin{example} \label{ex:ThreeHigh22}
{\rm Consider $\pTh(z) = 1$ and the low-pass filter
\[
\pa(z) = \tfrac{2-\sqrt{6}}{8}z^{-2}(z+1)^2(z^2 - (4+\sqrt{6})z + 1).
\]
We see from Figure~\ref{fig:Three22} that $ \det(\cM_{\pa, 1}(z)) $ changes sign on $\T $. Hence $ s_{a, \Theta}^+ = 2 $ and $ s_{a, \Theta}^- = 1 $.
Note that $ \sr(\pa)=2 $ and $ \vmo(1-\pa\pa^\star)=2 $. Therefore, the maximum order of vanishing moments is one.
Taking $n_b=1$, we obtain a quasi-tight framelet filter bank
	$\{\ta; \tb_1, \tb_2, \tb_3\}_{\Theta, \{1, 1, -1\}}$ as follows:
\begin{align*}
	\pb_1(z) =& \tfrac{\sqrt{10}}{40}(z-1) z^{-2} \left(
	(\sqrt{6}-3)z^3 + (2\sqrt{6}-3)z^2 + (2\sqrt{6}-1)z - \sqrt{6}-1
	\right),\\
	\pb_2(z) =& \tfrac{\sqrt{10}}{40}(z-1) z^{-2} \left(
	(\sqrt{6}-2)z^3 + (\sqrt{6}-4)z^2 + (\sqrt{6}-2)z + \sqrt{6}-4
	\right),\\
	\pb_3(z) =& \tfrac{\sqrt{10}}{40}(z-1)^2 z^{-2} \left(
	(2-\sqrt{6})z^2 + (6-2\sqrt{6}) z + 4 - \sqrt{6}
	\right), \\
	\end{align*}
with $ \vmo(\tb_1) = \vmo(\tb_2) = 1 $ and $ \vmo(\tb_3) = 2 $.
Since $\sm(a) \approx 0.9382$, the refinable function $ \phi $ defined in \eqref{phi} belongs to $ \Lp{2} $. Therefore,
$\{\phi,\phi;\psi^1,\psi^2,\psi^3\}_{(1,1,-1)}$ is a quasi-tight framelet in $\Lp{2}$ and
$ \{\psi^1, \psi^2, \psi^3\}_{(1, 1, -1)} $ is a homogeneous quasi-tight framelet in $ \Lp{2} $, where $ \psi^1, \psi^2, \psi^3 $ are defined in \eqref{eta:psi} and have one vanishing moment.
} \end{example}
	
\begin{figure}[h!]
\centering	
\begin{subfigure}[]{0.18\textwidth}			 \includegraphics[width=\textwidth, height=0.8\textwidth]{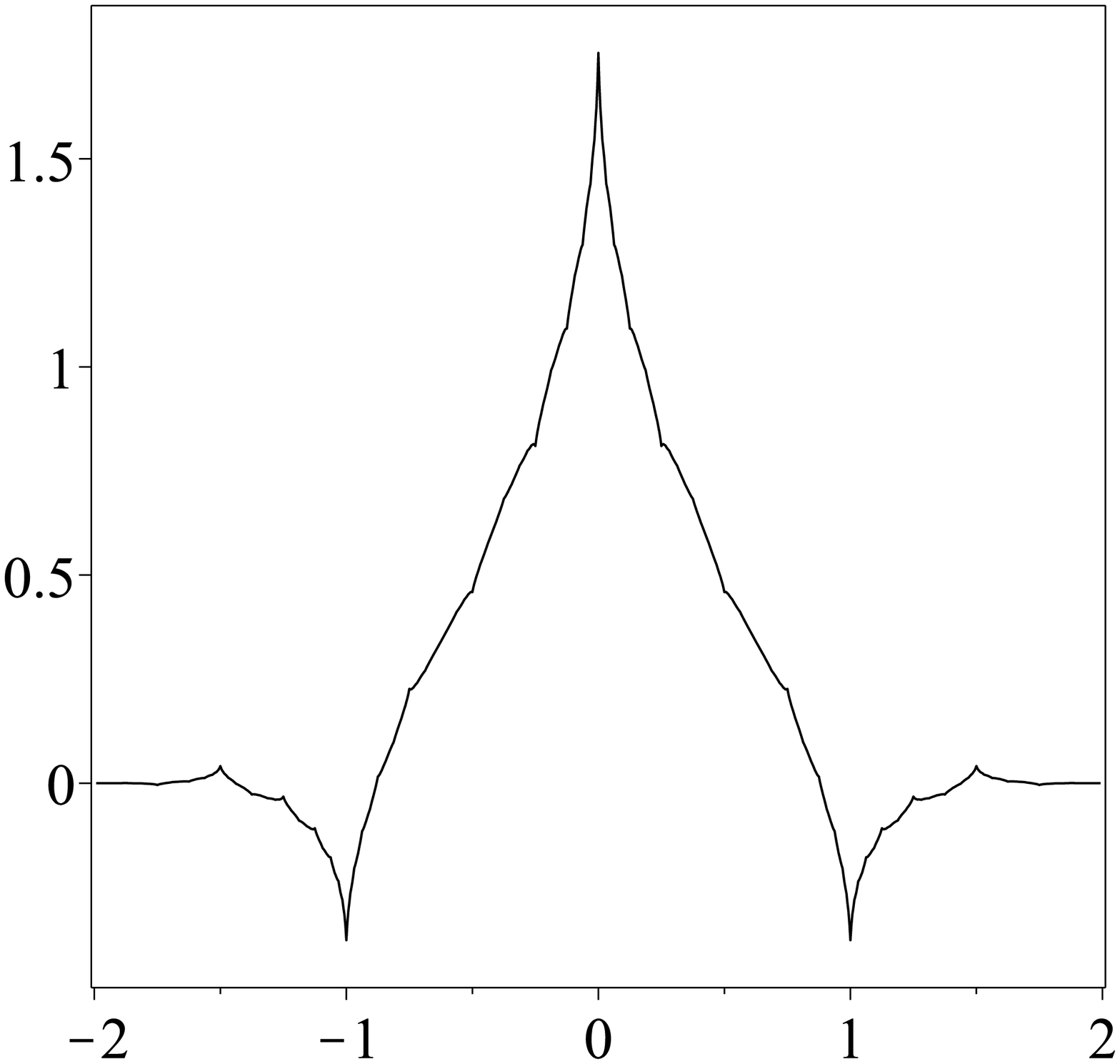}
			\caption{$\phi$}
		\end{subfigure}
\begin{subfigure}[]{0.18\textwidth}		 \includegraphics[width=\textwidth, height=0.8\textwidth]{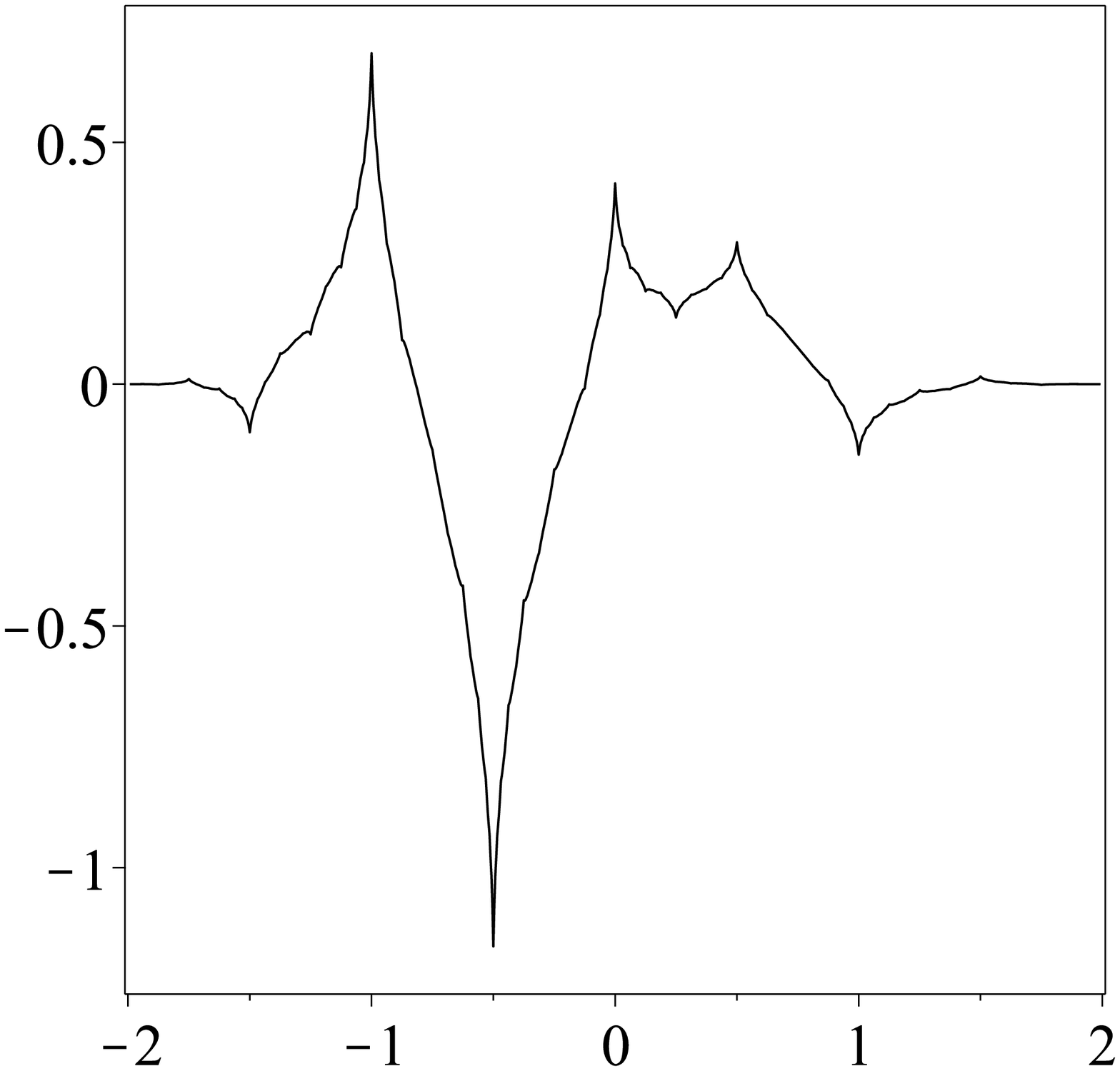}
			\caption{$\psi^1$}
		\end{subfigure}
\begin{subfigure}[]{0.18\textwidth}		 \includegraphics[width=\textwidth, height=0.8\textwidth]{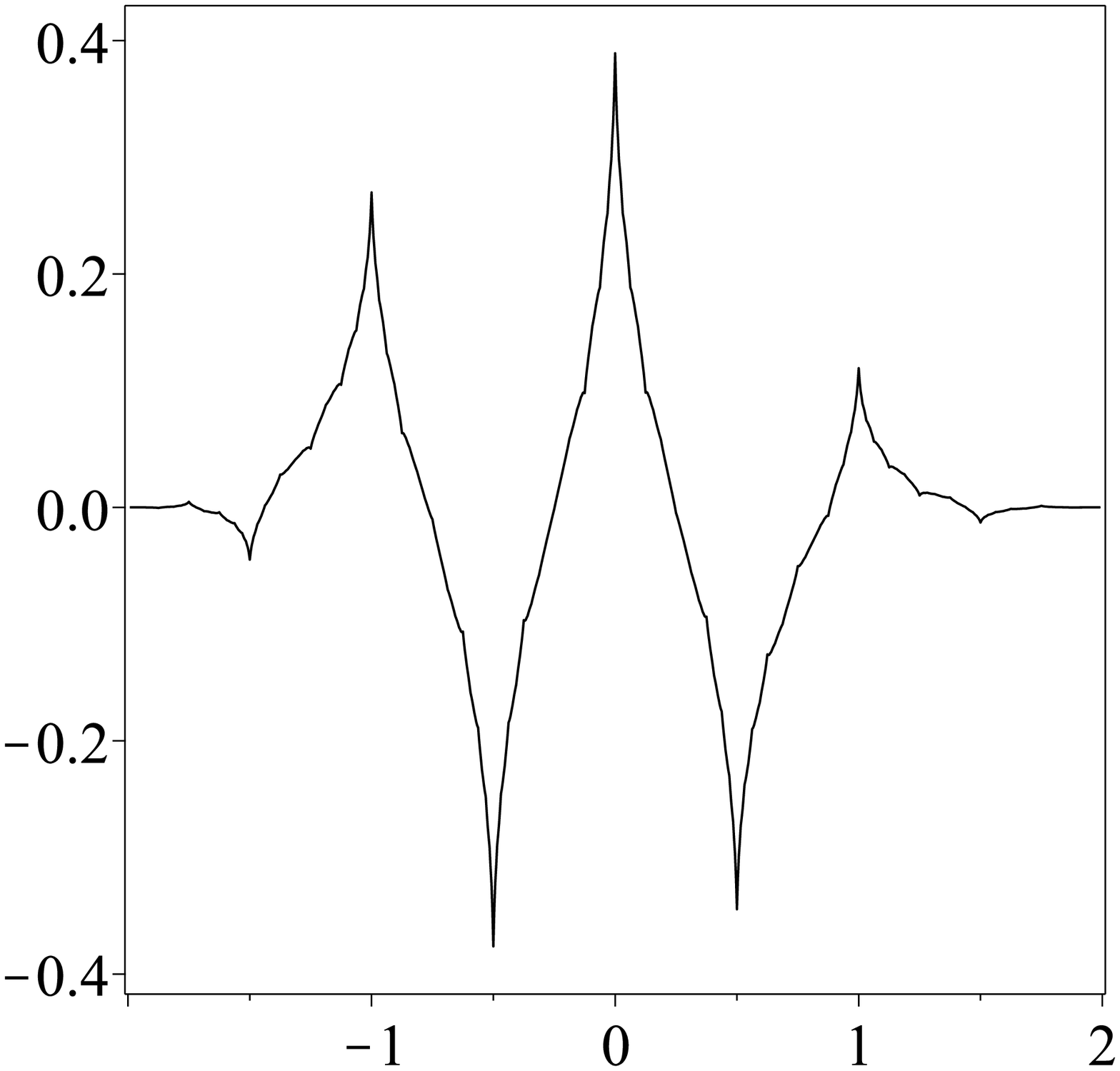}
			\caption{$\psi^2$}
		\end{subfigure}	
\begin{subfigure}[]{0.18\textwidth}		 
\includegraphics[width=\textwidth, height=0.8\textwidth]{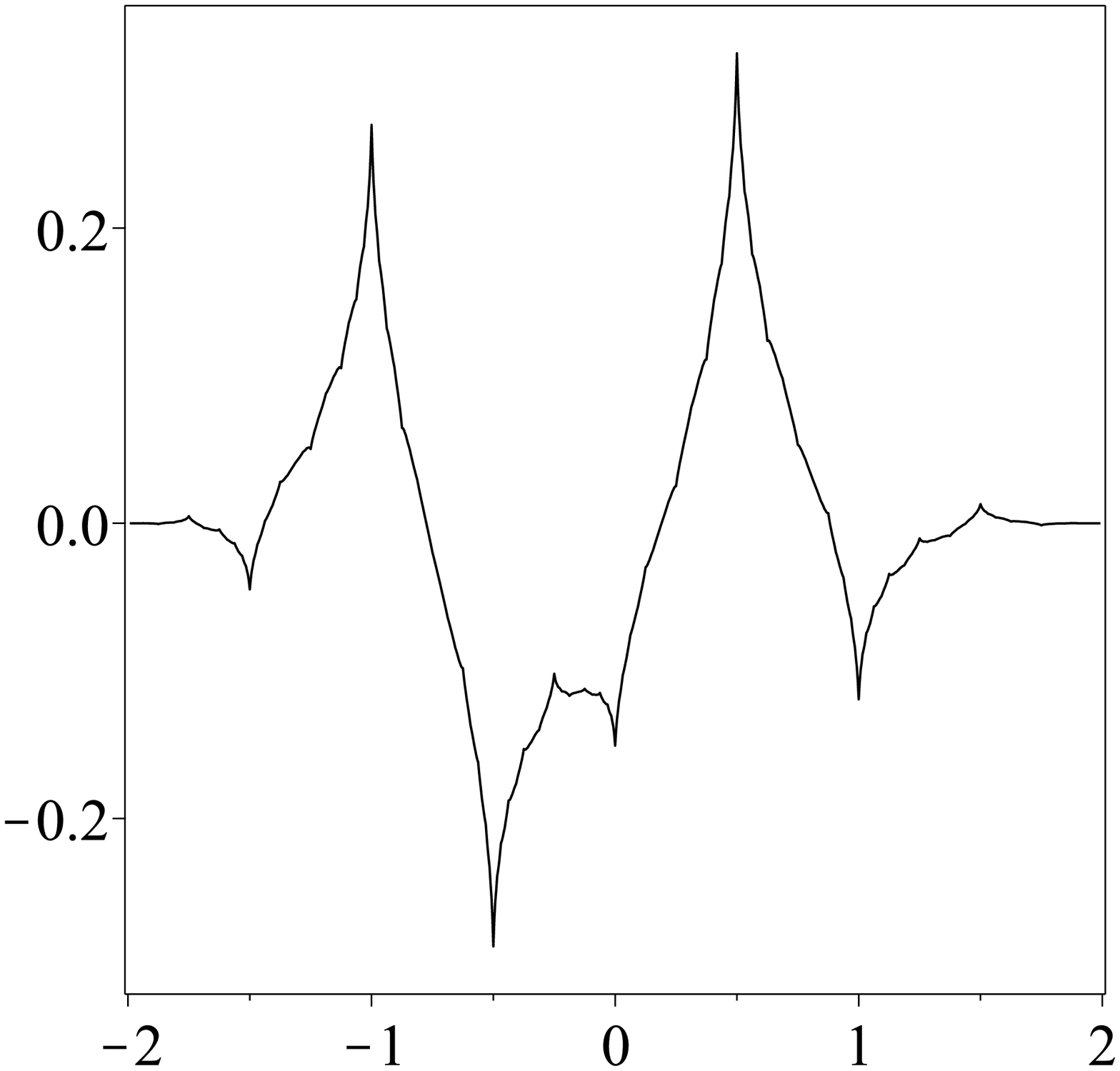}
			\caption{$\psi^3 $ }
		\end{subfigure}
\begin{subfigure}[]{0.18\textwidth}
\includegraphics[width=\textwidth, height=0.8\textwidth]{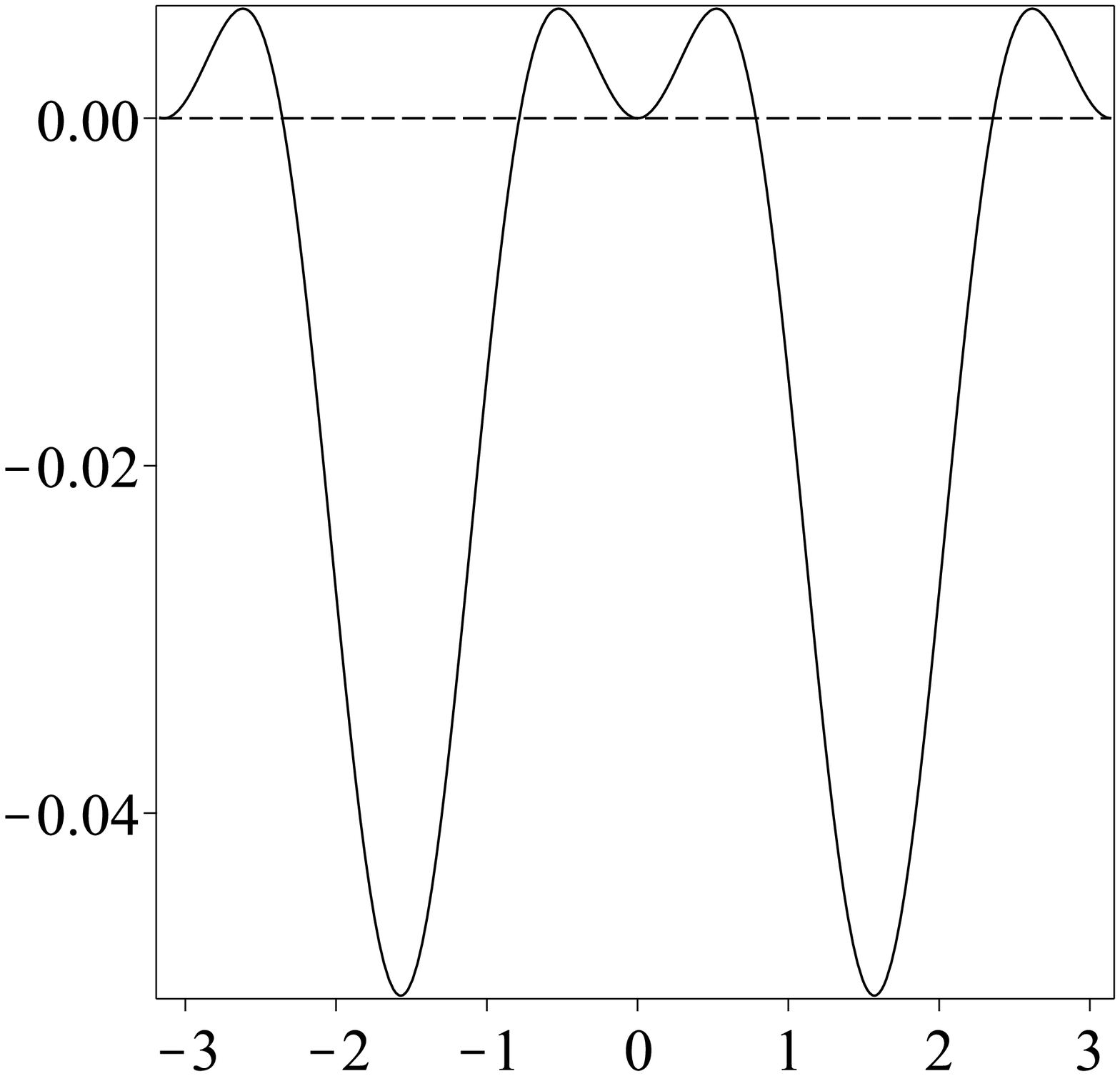}
\caption{$ \det(\cM_{\pa, 1})$}
\end{subfigure}
\caption{
The quasi-tight framelet $\{\phi,\phi; \psi^1,\psi^2,\psi^3\}_{(1,1,-1)}$ in $\Lp{2}$ and the homogeneous quasi-tight framelet $\{\psi^1,\psi^2,\psi^3\}_{(1,1,-1)}$ in $\Lp{2}$ obtained in Example~\ref{ex:ThreeHigh22}.
(A) is the refinable function $\phi\in \Lp{2}$. (B), (C), and (D) are the framelet functions $\psi^1$, $ \psi^2 $ and $ \psi^3$.
(E) is $\det(\cM_{\pa, 1}(e^{-i\xi}))$ for $ \xi\in  [-\pi, \pi] $.
} \label{fig:Three22}
\end{figure}

\section{Proof of Theorem~\ref{thm:const-sig} on Generalized Spectral Factorization for Matrices with Constant Signature}

In this section, we shall prove Theorem~\ref{thm:const-sig} on generalized matrix spectral factorization for Hermite matrices of Laurent polynomials with constant signature. To improve presentation and readability,
the proofs of several auxiliary results for proving Theorem~\ref{thm:const-sig} shall be given in the Appendix.

For an $ n\times n $ square matrix $ \pU(z) $ of Laurent polynomials, if $ \det(\pU(z))\not\equiv 0 $ is a monomial (Laurent polynomial with only one term), we call it \emph{unimodular}.
$ \pU(z) $ is unimodular if and only if there exists a unique $ n \times n $ matrix $ \pU^{-1}(z) = \tfrac{1}{\det(\pU(z))}\operatorname{adj}(\pU(z)) $ of Laurent polynomials such that $ \pU(z)\pU^{-1}(z) = \pU^{-1}(z)\pU(z) = \mathbf{I}_n $.
To prove Theorem~\ref{thm:const-sig},
we first show that Theorem~\ref{thm:const-sig} holds under the additional condition that $\det(\pA(z))$ is a nonzero monomial.
The general case of Theorem~\ref{thm:const-sig} will be then proved by extracting out the nontrivial factors of $\det(\pA(z))$ one by one.

\begin{theorem}[Unimodular Case for Theorem~\ref{thm:const-sig}] \label{thm:unimodular}
For an $n \times n$ Hermite matrix $\pA(z)$ of Laurent polynomials with a nonzero monomial $\det(\pA(z))$,
there exists an $n \times n$ matrix $\pU(z)$ of Laurent polynomials
such that $\pA(z)=\pU(z)\pD\pU^\star(z)$, where $\pD:=\diag(\mathbf{I_{\nu_+}}, -\mathbf{I_{\nu_-}})$ and $n=\nu_++\nu_-$ for some nonnegative integers $ \nu_+$ and $ \nu_- $.
\end{theorem}

Theorem~\ref{thm:unimodular} is known for rings with involution (e.g., see \cite{Lyu:1973Factorization,Lyu:1973Factorizationa,Cop:1972Linear,Djo:1976Hermitian}), including rings of (Laurent) polynomials as special cases. To provide a self-contained proof to Theorem~\ref{thm:unimodular} for completeness,
we present Algorithm~\ref{algo:unimodular} to construct desired matrices $ \pU(z) $ and $ \pD $ in Theorem~\ref{thm:unimodular} by showing that Algorithm~\ref{algo:unimodular} is feasible and will terminate in finitely many steps.

For a Laurent polynomial $ \pu(z) \not\equiv 0 $, we use $ \deg(\pu(z)) $ to denote its highest degree, and use $ \ldeg(\pu(z)) $ to denote its lowest degree. We define the length of $ \pu(z) $ as $ \len(\pu) := \deg(\pu) - \ldeg(\pu) $, and the interval: $ \fs(\pu(z)):=[\ldeg(\pu(z)), ~\deg(\pu(z))] $. If $ \pu(z)\equiv 0 $, then we just define $ \len(\pu) := -\infty $ and $ \fs(\pu) $ to be the empty set.
For a $ k\times k $ matrix $ \pQ(z) $ of Laurent polynomials, we call it \emph{diagonally dominant at the diagonal entry $s$} if
\begin{enumerate}
	\item for all  $ i\neq s $:
\begin{equation} \label{def:DiagDominant1}
	\fs(\pQ_{i,s}(z))\subsetneq \fs(\pQ_{s,s}(z))\quad \mbox{and}~
	\fs(\pQ_{s,i}(z))\subsetneq \fs(\pQ_{s,s}(z));
\end{equation}
	\item for all $ i > s$:
\begin{equation} \label{def:DiagDominant2}
	 \deg(\pQ_{s, i}(z)) < \deg(\pQ_{s,s}(z)).
\end{equation}
\end{enumerate}
$ \pQ(z) $ is called \emph{diagonally dominant} if it is diagonally dominant at all its diagonal entries $ s = 1,\ldots,k $.

The idea adopted in Algorithm~\ref{algo:unimodular} is similar to \cite{GohLanRod:1982Factorization} for the polynomial matrices.
To improve readability for Algorithm~\ref{algo:unimodular}, we provide some auxiliary lemmas with algorithmic proofs given in the Appendix serving as sub-steps in Algorithm~\ref{algo:unimodular}.

\begin{lemma} \label{lemma:EmptySpect2}
Let  $ \pQ(z) $ be a $ k\times k $ Hermite matrix of Laurent polynomials such that its $(1,1)$-entry $ \pQ_{1,1}(z)\not\equiv 0 $
and
\begin{equation} \label{eq:lemmaDiagIncreasing}
\len(\pQ_{1,1}(z))\leqslant \len(\pQ_{2,2}(z))\leqslant\cdots\leqslant \len(\pQ_{k,k}(z)).
\end{equation}
Suppose that $ \pQ(z) $ is diagonally dominant at its first $ s $ diagonal entries for some $ s<k $. (If $ \pQ(z) $ is not diagonally dominant at its first diagonal entry, then just take $ s = 0 $.)
Then there exists a $ k\times k $ unimodular matrix $ \pU(z) $ of Laurent polynomials such that $ \widetilde{\pQ}(z):= \pU(z)\pQ(z)\pU^\star(z) $ is diagonally dominant at its first $ (s+1) $ diagonal entries, while the top left $ (s+1)\times (s+1) $ submatrix of $ \widetilde{\pQ}(z) $ is the same as that of $ \pQ(z) $.
\end{lemma}

\begin{lemma} \label{lemma:EmptySpect3}
Let  $ \pQ(z) $ be a $ k\times k $ unimodular Hermite matrix of Laurent polynomials. If its first diagonal entry $ \pQ_{1,1}(z) \equiv 0 $, then there exist a $ k\times k $ unimodular matrix $ \pU(z) $ of Laurent polynomials and a $ (k-1)\times (k-1) $ matrix $ \widetilde{\pQ}(z) $ of Laurent polynomials such that
\[ \pQ(z) = \pU(z)\begin{bmatrix}
1 & \\ & \widetilde{\pQ}(z)
\end{bmatrix}\pU^\star(z). \]
\end{lemma}

We are now ready to present Algorithm~\ref{algo:unimodular} below to prove Theorem~\ref{thm:unimodular}.
The structure and idea of the following Algorithm~\ref{algo:unimodular} consist of three main steps.
\begin{enumerate}
	\item If the first diagonal entry of the $ n\times n$ Hermite matrix $ \pA(z) $ is identically zero, then we apply Lemma~\ref{lemma:EmptySpect3} to find a unimodular matrix $ \pU(z) $ such that $ \pA(z) =
	\pU(z)\begin{bmatrix}
	1 & \\ & \widetilde{\pA}(z)
	\end{bmatrix}\pU^\star(z) $ holds. Hence, the problem is reduced to solving the generalized matrix spectral factorization of the $ (n-1)\times (n-1) $ matrix $ \widetilde{\pA}(z) $.
	
	\item If the first diagonal entry of $ \pA(z) $ is not identically zero, then we can repeatedly apply Lemma~\ref{lemma:EmptySpect2}, to reduce $ \pA(z) $ to a diagonally dominant matrix.
	
	\item If $ \pA(z) $ is a unimodular diagonally dominant matrix of Laurent polynomials, then it must be a diagonal constant matrix. So we can solve its spectral factorization directly.
\end{enumerate}

\begin{algorithm} \label{algo:unimodular}
Let $\pA(z)$ be an $n \times n$ Hermite matrix of Laurent polynomials such that $\det(\pA(z))$ is a nonzero monomial.

\begin{enumerate}
\item[(S0)] Initialization. Set $\pU(z) := \mathbf{I}_n$ to be the $n \times n$ identity matrix.
Let $\pQ(z) := \pA(z)$ and $k:=n$.

\item[(S1)] Find a permutation matrix  $\widetilde{\pU}$ such that $\widetilde{\pQ}(z):=\widetilde{\pU}\pQ(z)\widetilde{\pU}^\star$ satisfies
$$ \len(\widetilde{\pQ}_{1,1}(z))\leqslant \len(\widetilde{\pQ}_{2,2}(z))\leqslant\cdots\leqslant \len(\widetilde{\pQ}_{k,k}(z)).$$
Update/replace $\pU(z)$ by $\pU(z)
\diag(\mathbf{I}_{n-k}, \widetilde{\pU}^{-1})$ and
$\pQ(z):=\widetilde{\pQ}(z) $.

\item[(S2)] If the first diagonal entry $\pQ_{1,1}(z)\not\equiv 0$, then go to step  (S3). Otherwise,
apply Lemma~\ref{lemma:EmptySpect3} to find a $ k\times k $ unimodular matrix $ \widetilde{\pU}(z) $ of Laurent polynomials such that
$ \widetilde{\pU}(z)\pQ(z)\widetilde{\pU}^\star(z) =\diag(1, \widetilde{\pQ}(z))
$ for some $ (k-1)\times (k-1) $ matrix $ \widetilde{\pQ}(z) $ of Laurent polynomials.
Update/replace $\pU(z)$ by $\pU(z)
\diag(\mathbf{I}_{n-k}, \widetilde{\pU}^{-1}(z))$ and $\pQ(z):=\widetilde{\pQ}(z)$. Set $k:=k-1$ and restart from  (S1).

\item[(S3)] For $\pQ_{1,1}(z) \not\equiv 0$, if $ \pQ(z) $ is a diagonally dominant matrix, then go to step (S5). Otherwise, find the largest number $ s $ such that $ \pQ(z) $ is diagonally dominant at its first $ s $ diagonal entries. If it is not diagonally dominant at the first diagonal entry, then just take $ s = 0 $.
Apply Lemma~\ref{lemma:EmptySpect2} to find a $ k\times k $ unimodular
matrix $ \widetilde{\pU}(z) $ of Laurent polynomials such that
$ \widetilde{\pQ}(z):=
\widetilde{\pU}(z)\pQ(z)\widetilde{\pU}^\star(z) $
is diagonally dominant at its first $ (s+1) $ diagonal entries.
Update/replace $\pU(z)$ by $\pU(z)
\diag(\mathbf{I}_{n-k}, \widetilde{\pU}^{-1}(z))$ and $\pQ(z):=\widetilde{\pQ}(z)$.

\item[(S4)] If the lengths of diagonal entries in $\pQ(z)$ are not non-decreasing any more, that is,
$$ \len(\pQ_{1,1}(z))\leqslant \len(\pQ_{2,2}(z))\leqslant\cdots\leqslant \len(\pQ_{k,k}(z))$$
is not satisfied, then restart from  (S1) to sort them again. Otherwise, repeat from  (S3).

\item[(S5)] If $\pQ(z)$ is diagonally dominant, then $\pQ(z)$ must be a constant diagonal matrix, that is,
$\pQ = \diag(\lambda_1, \ldots, \lambda_k) $, where $ \lambda_j \neq 0 $ for all $ j = 1,\ldots, k $.
Without loss of generality, we assume that the first $ (k-\nu_-) $ of the $ \lambda_j $ are positive and the last $ \nu_- $ of them are negative.
Define $ \widetilde{\pU}:=\diag(\sqrt{|\lambda_1|}, \ldots, \sqrt{|\lambda_k|}) $. We have $ \pQ= \widetilde{\pU}\diag(\mathbf{I}_{k-\nu_-}, -\mathbf{I}_{\nu_-}) \widetilde{\pU}^\star $.
Update/replace $\pU(z)$ by $\pU(z)
\diag(\mathbf{I}_{n-k},\widetilde{\pU})
$, and define
$\pD:=
\diag(\mathbf{I}_{n-k}, \mathbf{I}_{k-\nu_-},
-\mathbf{I}_{\nu_-})$. Such output $\pU(z)$ and $\pD$ must satisfy
$ \pA(z) = \pU(z)\pD\pU^\star(z) $ and all the requirements in Theorem~\ref{thm:unimodular}.
\end{enumerate}

\end{algorithm}

\begin{proof}
It is easy to see that after the initialization step (S0), we have
\begin{equation} \label{eq:UQU}
\pA(z) = \pU(z)\begin{bmatrix}
\mathbf{I}_{n-k} & \\ & \pQ(z)
\end{bmatrix}\pU^\star(z).
\end{equation}
Each time we update $\pU(z)$ and $\pQ(z)$ in steps  (S1),(S2),(S3) and (S5), we are actually factoring out some matrices from the original $\pQ(z)$.
Hence, by induction, \eqref{eq:UQU} will always hold during the whole process of the algorithm.
So if the algorithm can finalize in  (S5), the decomposition $\pA(z)=\pU(z)\pD\pU^\star(z)$ must hold.
We prove that all the steps in the algorithm are feasible and they will terminate after finitely many steps.
The feasibility of steps (S2) and (S3) are proved by Lemmas~\ref{lemma:EmptySpect2} and \ref{lemma:EmptySpect3}.

In (S6), we know that if $\pQ(z)$ is diagonally dominant, then $\len(\det(\pQ(z)))=\sum_{l=1}^{k}\len(\pQ_{l,l}(z))$. By \eqref{eq:UQU}, we deduce
$\det(\pA(z)) = \det(\pU(z))\det(\pQ(z))\det(\pU^\star(z))$, which implies  $\det(\pQ(z))\mid\det(\pA(z))$. Since $\det(\pA(z))$ is a nonzero monomial,  $\det(\pQ(z))$ is a nonzero monomial.
Hence $\sum_{l=1}^{k}\len(\pQ_{l,l}(z))=0$, which forces all the diagonal entries of $ \pQ(z) $ to be monomials. Since $\pQ(z)$ is a Hermite matrix, all the diagonal entries of $ \pQ(z) $ must be nonzero constants. Because $\pQ(z)$ is diagonally dominant, so $\pQ(z)$ must be a diagonal constant matrix.

Finally, we prove that the algorithm will stop after finitely many iterations. The algorithm might restart from (S1) in (S2) and (S4) or restart from (S3) in (S4).

When the restart from (S1) in (S2) occurs, the size $ k $ of $\pQ(z)$ will decrease by $1$. By \eqref{eq:UQU}, it can happen only finite number of times.

In order to show that the algorithm can only restart from (S1) in (S4) for finitely many times, let us use the lexicographic order of sequences of length $k$. For any two sequences of nonnegative integers with length $k$:
$\{\alpha_j\}_{j=1}^k,\{\beta_j\}_{j=1}^k \in \NN^k$, we say that $\{\alpha_j\}_{j=1}^k $ is less than $\{\beta_j\}_{j=1}^k$ if there exists some index $j_0\in \{1,\ldots, k\}$, such that $\alpha_j=\beta_j$ for all $j<j_0$, and $\alpha_{j_0} < \beta_{j_0}$.
$\{\alpha_j\}_{j=1}^k $ is equal to $\{\beta_j\}_{j=1}^k$ if $\alpha_j=\beta_j$ for all $j=1,\ldots,k$.
It is easy to see that $\NN^k$ is a well-ordered set under this lexicographic order. Every time the algorithm restarts from (S1) in (S4), the lexicographic order of $\{\len(\pQ_{i,i}(z))\}_{i=1}^k\in \NN^k$ will decrease. Since the sequence is lower bounded by the sequence $ \{0, \ldots,0\} $, the restarts can occur only finitely many times.

Every time the algorithm restarts from (S3) in (S4), $s$ will increase by at least 1, until the matrix $ \pQ(z) $ becomes diagonally dominant. So these iterations can only happen for finite number of times.
This completes the proof of Algorithm~\ref{algo:unimodular} and Theorem~\ref{thm:unimodular}.
\end{proof}

To complete the proof of Theorem~\ref{thm:const-sig}, we have to
extract out nontrivial factors of $\det(\pA(z))$. To do so,
let us recall some necessary notations first.
An $n \times n$ matrix $\pA(z)$ of Laurent polynomials can be factorized into
$$ \pA(z) = \pE(z)\pD(z)\pF(z), $$
where $\pE(z)$ and $\pF(z)$ are unimodular matrices of Laurent polynomials, and $\pD(z)=\diag(d_1(z), \ldots, d_n(z))$ is a diagonal matrix of Laurent polynomials with $d_j(z)\mid d_{j+1}(z)$ for all $j=1,\ldots,n-1$. $\pD(z)$ is called the \emph{Smith Normal Form} of $\pA(z)$.
Moreover, we can normalize $d_j$ by requiring that its leading coefficient should be $1$ and its constant term be nonzero.
Such polynomials $d_j(z)$ in the Smith normal form are called the \emph{invariant polynomials} of $\pA(z)$.
For all $ k = 1,\ldots, n $, the product $ \prod_{j = 1}^{k} d_j(z)$ is essentially the greatest common divisor (gcd) of all the determinants of $k \times k$ submatrices in $\pA(z)$. Let us write the invariant polynomials in $\C$ as follows:
$$ d_j(z) = \prod_{k=1}^{n_j}(z-z_{j,k})^{\alpha_{j,k}}, \qquad j=1,\ldots, n. $$
The factors $(z-z_{j,k})^{\alpha_{j,k}}$, $k=1,2,\ldots, n_j$, $j=1,2,\ldots, n$, where each factor could repeat as many times as it occurs, are called the \emph{elementary divisors} of $\pA(z)$.
For each $ j=1,2,\ldots, n $, since we require $d_j(z)$ to have a nonzero constant term, $d_j(z)$ has no root at $0$. Thus there won't be any $(z-0)^{\alpha_{j,k}}$ terms in the elementary divisors.
Also, by $d_j(z)\mid d_{j+1}(z)$ for all $j=1,\ldots,n-1$, we see that the Smith Normal Form $\pD(z)$ of $\pA(z)$ is uniquely determined by its elementary divisors.

Observe that $\det(\pA(z)) = \det(\pE(z))\det(\pD(z))\det(\pF(z))$. Since both $\det(\pE(z))$ and $\det(\pF(z))$ are nonzero monomials, we see that the determinant of $\pA(z)$ is essentially the product of all its invariant polynomials or the product of all its elementary divisors, up to some multiplicative nonzero monomials:
\begin{equation} \label{eq:DetEleDiv}
\det(\pA(z)) = c_A z^{k_A} \prod_{j=1}^n d_j(z)
= c_A z^{k_A} \prod_{j=1}^n \prod_{k=1}^{n_j}(z-z_{j,k})^{\alpha_{j,k}},
\end{equation}
for some nonzero constant $c_A\in\C$ and some integer $k_A\in \Z$.

To prove the general case in Theorem~\ref{thm:const-sig}, we need some auxiliary results to show that if $ \pA(z) $ is not unimodular, then its elementary divisors can be factored out. For this purpose, we need the following auxiliary results Theorems~\ref{thm:extract1} and \ref{thm:extract2}. Theorem~\ref{thm:extract1} deals with the elementary divisor $ (z-z_0)^\alpha $
in the case $ z_0 \not\in \T $ or the case $ z_0 \in\T $ but $ \alpha \geqslant 2 $. Theorem~\ref{thm:extract2} handles the elementary divisors with $ z_0\in \T $ and $ \alpha = 1 $.

\begin{theorem} \label{thm:extract1}
Let $\pA(z)$ be an  $n \times n$ Hermite matrix of Laurent polynomials with $\len(\det(A(z)))>0$.
If  $\pA(z)$ has some elementary divisor $(z-z_0)^\alpha$ satisfying either one of the two conditions:
\begin{enumerate}
	\item $z_0\in \left(\C\setminus\{0\}\right)\setminus\T$,
	\item $z_0\in \T$  and $\alpha\geqslant 2$,
\end{enumerate}
then there exist two $n \times n$ matrices $\pU(z)$ and $\widetilde{\pA}(z)$ of Laurent polynomials such that
$\pA(z)=\pU(z)\widetilde{\pA}(z)\pU^\star(z)$,
where $\widetilde{\pA}^\star(z) = \widetilde{\pA}(z)$ and $\len(\det(\widetilde{\pA}(z)))\leqslant \len(\det(\pA(z)))-2$.
\end{theorem}

\begin{proof}
Denote the invariant polynomials of $\pA(z)$ by $d_1(z), \ldots, d_n(z) $. Then there exist unimodular matrices $\pE(z)$ and $\pF(z)$ of Laurent polynomials such that
$$ \pA(z) = \pE(z)\diag(d_1(z), \ldots, d_n(z)) \pF(z). $$
Define $\mathring{\pA}(z)$ as
\begin{equation} \label{eq:smith}
\mathring{\pA}(z):=\pE^{-1}(z)\pA(z)\pE^{-\star}(z) =
\diag(d_1(z), \ldots, d_n(z))
\pF(z)\pE^{-\star}(z).
\end{equation}
Since $(z-z_0)^\alpha$ is an elementary divisor of $\pA(z)$, there exists some $d_k(z)$ such that $(z-z_0)^\alpha \mid d_k(z)$.
Hence $(z-z_0)^\alpha$ divides the $k$-th row of $\mathring{\pA}(z)$.
Also, $\mathring{\pA}(z)$ being a Hermite matrix implies that
$ \left((z-z_0)^\alpha\right)^\star=(z^{-1}-\overline{z_0})^\alpha = (-\overline{z_0})^{\alpha}z^{-\alpha}(z-\overline{z_0}^{-1})^\alpha $ divides the $k$-th column of $\mathring{\pA}(z)$.

In the following, we will show that in the items (1) and (2), we can factor out $ (z-z_0)^\beta $ from the $ k $-th row of $ \mathring{\pA}(z) $, and factor out $ ((z-z_0)^\beta)^\star $ from the $ k $-th column of $ \mathring{\pA}(z) $ simultaneously, where $ \beta = \alpha $ in item (1) and $ \beta =  \lfloor \alpha /2 \rfloor$ in item (2).

For item (1), we have $z_0\not\in \T$, and hence $\overline{z_0}^{-1}\neq z_0$.
So $(z-z_0)^\alpha$ and
$ (z-\overline{z_0}^{-1})^\alpha $ are different polynomials. Since they divide the $k$-th row and the $k$-th column of $\mathring{\pA}(z)$ respectively, we deduce that
$ (z-z_0)^\alpha (z-\overline{z_0}^{-1})^\alpha $
(or equivalently $\left((z-z_0)^\alpha\right)^\star (z-z_0)^\alpha$) divides the $(k, k)$-entry of
the matrix $\mathring{\pA}(z)$.
So we can factor out $(z-z_0)^\alpha$ from the $k$-th row and factor out $\left((z-z_0)^\alpha\right)^\star $ from the $k$-th column of $\mathring{\pA}(z)$ simultaneously.
Use $ \pD_{k, \alpha}(z) $ to denote the $ n\times n $ diagonal matrix with the $k$-th diagonal entry equal to $(z-z_0)^\alpha$, and all other diagonal entries equal to 1, i.e.,
\begin{equation} \label{eq:Diagk}
\pD_{k, \alpha}(z) := \diag(1,\ldots, 1, (z-z_0)^\alpha, 1, \ldots, 1).
\end{equation}
We get
$\mathring{\pA}(z)=
\pD_{k, \alpha}(z)\widetilde{\pA}(z)\pD_{k, \alpha}^\star(z)$,
where $\widetilde{\pA}(z)$ is an $n \times n$ Hermite matrix of Laurent polynomials.
So $\pA(z)$ can be written as
$$ \pA(z)
= \pE(z) \mathring{\pA}(z) \pE^\star(z)
= \pE(z)\pD_{k, \alpha}(z)\widetilde{\pA}(z)\pD_{k, \alpha}^\star(z)\pE^\star(z). $$
Let $\pU(z):=\pE(z)\pD_{k, \alpha}(z)$. Then we get
$ \pA(z)=\pU(z)\widetilde{\pA}(z)\pU^\star(z)$.

Since $\det(\pE(z))$ is a nonzero monomial and
$\det(\pU(z)) = \det(\pE(z))\det(\pD_{k, \alpha}(z))$,
we conclude that $\len(\det(\pU(z)))= \len(\det( \pD_{k, \alpha}(z) )) =\alpha$.
So
\begin{align*}
\len(\det(\widetilde{\pA}(z)))
&= \len(\det(\pA(z))) - \len(\det(\pU(z))) - \len(\det(\pU^\star(z))) \\
&= \len(\det(\pA(z))) -\alpha -\alpha
\leqslant  \len(\det(\pA(z))) -2.
\end{align*}
	
For item (2), we have $ \alpha \geqslant 2 $. Let $\beta:=\lfloor \alpha /2 \rfloor$ be the largest integer that is no larger than $ \alpha/2 $. Then $\beta\geqslant 1$ and $2\beta \leqslant \alpha$.
From $\beta \leqslant \alpha$, we see that $(z-z_0)^\beta$ divides the $k$-th row and $((z-z_0)^\beta)^\star $
divides the $k$-th column of $\mathring{\pA}(z)$.
For $z_0\in \T$, $|z_0|^2=1$ implies $\overline{z_0}^{-1} = z_0 $. So $ (z-\overline{z_0}^{-1})^\beta $ and $(z-z_0)^\beta$ are the same polynomial.
Since $2\beta\leqslant \alpha$ and $ (z-z_0)^\alpha $ divides the
$(k,k)$-entry of $\mathring{\pA}(z)$,
we get
$(z-z_0)^{\beta}((z-z_0)^\beta)^\star =(-\overline{z_0})^{\beta}z^{-\beta}(z-z_0)^{2\beta} $ which divides the $(k,k)$-entry of $\mathring{\pA}(z)$.
So we can factor out $(z-z_0)^\beta$ from the $k$-th row and $((z-z_0)^\beta)^\star$ from the $k$-th column at the same time to get
$\mathring{\pA}(z) = \pD_{k, \beta}(z)\widetilde{\pA}(z)\pD_{k, \beta}^\star(z)$,
where $\widetilde{\pA}(z)$ is an $n \times n$ Hermite matrix of Laurent polynomials
and $\pD_{k, \beta}(z)$ is defined as \eqref{eq:Diagk}.

Using similar arguments as for item (1), we get
$$ \pA(z)
= \pE(z) \mathring{\pA}(z) \pE^\star(z)
= \pE(z)\pD_{k, \beta}(z)\widetilde{\pA}(z)\pD_{k, \beta}^\star(z)\pE^\star(z) =
\pU(z)\widetilde{\pA}(z)\pU^\star(z), $$
where $\pU(z):=\pE(z)\pD_{k, \beta}(z)$.
Because
$\len(\det(\pU(z)))= \len(\det(\pD_{k, \beta}(z))) = \beta$, we have
\begin{align*}
	\len(\det(\widetilde{\pA}(z))) & = \len(\det(\pA(z))) - \len(\det(\pU(z))) - \len(\det(\pU^\star(z))) \\
	& = \len(\det(\pA(z))) -\beta -\beta
	\leqslant  \len(\det(\pA(z))) -2.
	\end{align*}
This completes the proof.
\end{proof}

If the Hermite matrix $ \pA(z) \geqslant 0 $ for all $ z\in \T $, then we can prove that all the elementary divisors of $ \pA(z) $ must be either item (1) or item (2) in Theorem~\ref{thm:extract1}.
Actually, if $ \pA(z) $ is positive semidefinite for all $ z\in\T $, all its elementary divisors $(z-z_0)^\alpha $ with $z_0\in \T$ will have even multiplicity $\alpha$.
See Corollary \ref{cor:SPDEvenEleDiv} later in this paper.
However, $ z_0\in \T $ and $\alpha=1$ can indeed happen if the matrix $ \pA(z) $ is not positive semidefinite.
This is the main difference/difficulty in the proof of the generalized spectral factorization of matrices with constant signature, in comparison to the proof of the standard matrix-valued Fej\'er-Riesz lemma,
as demonstrated by the following example.

\begin{example} \label{ex:SingleRoot}
{\rm
Consider the matrix
\[
\pA(z)=
\left[ \begin {array}{cc}
{ z^{-1}\left( z-1 \right)^{2}}&
\left( z-1 \right)  \left( z+1 \right) \\
\noalign{\medskip}
\left( {z}^{-1}-1 \right)  \left( {z}^{-1}+1 \right) &
-{z^{-1}\left( z-1 \right)^{2}}
\end {array} \right].
\]
By direct calculation we have $\pA^\star(z) =\pA(z)$
and $\det(\pA(z)) = \frac{4(z-1)^2}{z}=-\pd(z)\pd^\star(z) \leqslant 0$ for all $ z\in \T $, where $\pd(z) = 2(z-1)$.
Since the determinant is equal to the product of all the eigenvalues of $ \pA(z) $,
we know that $ \nu_+( \pA(z)) = \nu_-( \pA(z)) = 1$ for all $z\in\T\setminus\sigma(\pA)$. Hence, the signature of $ \pA(z) $ is constant for all
$ z\in\T\setminus\sigma(\pA)$.

As to the Smith Normal Form of $ \pA(z) $, let
$$
\pE(z):=
\left[
\begin {array}{cc}
\frac{-2z^3+4z^2+z-1}{z}&
2z\left(2-z \right) \\
\noalign{\medskip}{\frac {2\,{z}^{3}-z-1}{{
z}^{2}}}&2\,z\end {array} \right]
,\quad
\pF(z):=
\left[
\begin {array}{cc} 1&2\,{z}^{2}-z\\
\noalign{\medskip}-1&
z+1-2z^2\end {array} \right],\quad
\pD(z):=
\left[
\begin{array}{cc}
z-1 & 0\\
\noalign{\medskip}0&{ z-1 }
\end {array}
\right].
$$
We can directly verify that $\pA(z) = \pE(z)\pD(z)\pF(z)$ and $\pE(z), \pF(z)$ are both unimodular matrices. So $\pD(z)$ is the Smith Normal Form of $\pA(z)$. Hence $\pA(z)$ has two elementary divisors being $(z-1)$.
} \end{example}

The following theorem handles the elementary divisors $ (z-z_0)^\alpha $ with $ z_0 \in \T $ and $ \alpha = 1 $.

\begin{theorem} \label{thm:extract2}
Let $\pA(z)$ be an $n \times n$ Hermite matrix of Laurent polynomials.
If $\pA(z)$ satisfies
\begin{enumerate}
\item $\sig(A(z))$ is constant for all $z\in \T \setminus \sigma(\pA)$;
\item there exists some $z_0\in \sigma(\pA) \cap \T$ and all the elementary divisors of $\pA(z)$ with root $ z_0 $ have degree equal to $1$,
\end{enumerate}
then there exist two $n \times n$ matrices $\pU(z)$ and $\widetilde{\pA}(z)$ of Laurent polynomials such that
$\pA(z)=\pU(z)\widetilde{\pA}(z)\pU^\star(z)$,
where $\widetilde{\pA}^\star(z) = \widetilde{\pA}(z)$ and $\len(\det(\widetilde{\pA}(z)))\leqslant \len(\det(\pA(z)))-2$.
\end{theorem}

We need the following result to prove Theorem~\ref{thm:extract2}, which connects the study of the eigenvalues of $ \pA(z) $ and its invariant polynomials.
Let us recall the big $ \bo $ notation to study real analytic functions. For an analytic function $ f(\xi) $,  we say that $ f(\xi) = \bo((\xi - \xi_0)^n) $ as $ \xi \rightarrow \xi_0 $ if the $ k $-th derivative $ f^{(k)}(\xi_0) = 0$ for all $ 0 \leqslant k < n $.
We also abuse the notation for the multiplicity of the root of Laurent polynomials. For an analytic function $ f(\xi) $, we use $ \mz(f, \xi_0) $ to denote the largest integer $ n $ such that $ f(\xi) = \bo((\xi-\xi_0)^n) $ as $\xi\to \xi_0$.

\begin{theorem} \label{thm:partialMultiplicity}
Suppose that $ \pA(z) $ is an $ n\times n $ Hermite matrix of Laurent polynomials and $ z_0 = e^{-i\xi_0} \in \T $ with $ \xi_0\in \R $.
Let $ \pd_1(z), \ldots, \pd_n(z) $ be the invariant polynomials of $ \pA(z) $ and define the sequence $ \{\alpha_j\}_{j=1}^n $ by
\[ \alpha_j := \mz(\pd_j(z), z_0), \qquad j = 1,\ldots, n. \]	
Also, we can find $ n $ analytic functions
$ \lambda_1(\xi), \ldots, \lambda_n(\xi) $ for $ \xi \in \R $, which are the eigenvalues of the analytic matrix $ \pA(e^{-i\xi}) $.
Define the sequence $ \{\beta_j\}_{j=1}^n $ by
\[ \beta_j := \mz(\lambda_j(\xi), \xi_0), \qquad j = 1,\ldots, n. \]
Without loss of generality, we can assume $ \beta_1 \leqslant \cdots \leqslant \beta_n $.	
Then the sequence $ \{\alpha_j\}_{j=1}^n $ and the sequence $ \{\beta_j\}_{j=1}^n $ must be the same.
\end{theorem}

\begin{proof}
The invariant polynomials $\pd_j(z)\mid\pd_{j+1}(z)$ hold for all $j=1,2,\ldots,n-1$.
Hence, $\alpha_1\leqslant\cdots\leqslant\alpha_n $.
There exist $ n\times n $ invertible matrices of Laurent polynomials $\pE(z)$ and $\pF(z)$ such that
\begin{equation} \label{eq:smith2}
\pA(z) = \pE(z)\diag(\pd_1(z), \pd_2(z), \ldots, \pd_n(z)) \pF(z)
\end{equation}
holds.
Take $z=e^{-i\xi}$, $ \xi\in\R $. We see that the invariant polynomials $\pd_j(e^{-i\xi})$ are analytic functions of $\xi\in \R$
and $\mz(\pd_j(e^{-i\xi}), \xi_0) = \mz(\pd_j(z), z_0) = \alpha_j$ for all $j=1,2,\ldots,n$.

Write $\pd_j(e^{-i\xi}) = (\xi-\xi_0)^{\alpha_j}\widetilde{d_j}(\xi)$ with $\widetilde{d_j}(\xi_0)\neq 0$. We can rewrite equation \eqref{eq:smith2} as follows,
\begin{align*}
\pA(e^{-i\xi}) =& \pE(e^{-i\xi})\diag(\pd_1(e^{-i\xi}), \ldots, \pd_n(e^{-i\xi})) \pF(e^{-i\xi}) \\
= & \pE(e^{-i\xi})
\diag((\xi-\xi_0)^{\alpha_1},\ldots,(\xi-\xi_0)^{\alpha_n})
\diag(\widetilde{d_1}(\xi),\ldots,\widetilde{d_n}(\xi))
\pF(e^{-i\xi})\\
= &
E_{\xi_0}(\xi)
\diag((\xi-\xi_0)^{\alpha_1},\ldots,(\xi-\xi_0)^{\alpha_n})
F_{\xi_0}(\xi),
\end{align*}
where $E_{\xi_0}(\xi):=\pE(e^{-i\xi})$ and $ F_{\xi_0}(\xi):= \diag(\widetilde{d_1}(\xi),\ldots, \widetilde{d_n}(\xi))\pF(e^{-i\xi})$.
From the definition, $E_{\xi_0}(\xi)$ and $F_{\xi_0}(\xi)$ are both analytic matrices, and $\det(E_{\xi_0}(\xi_0))\neq 0$,
$\det(F_{\xi_0}(\xi_0))\neq 0 $.
Hence, the matrices $E_{\xi_0}(\xi), F_{\xi_0}(\xi)$ and the sequence $\{\alpha_j\}_{j=1}^n$ satisfy all the conditions of Lemma \ref{lemma:PartialMultiplicity} in the Appendix. So the partial multiplicities of $\pA(e^{-i\xi})$ at $\xi_0$ are
	$\{\alpha_j\}_{j=1}^n$.

Since $\pA(e^{-i\xi})$ is an analytic Hermite matrix for $ \xi \in \R $,
by \cite[Theorem~S6.3]{GohLanRod:1982Matrix}, it can also be factorized as
\begin{equation} \label{eq:AnalyticEVD}
\pA(e^{-i\xi}) = W(\xi)
\diag(\lambda_1(\xi), \ldots, \lambda_n(\xi))
(W(\xi))^\star,
\end{equation}
where $W(\xi)$ is a unitary analytic matrix and
the eigenvalues $\lambda_1(\xi), \ldots, \lambda_n(\xi)$ are analytic functions of $\xi\in\R$.
Without loss of generality, we can assume that $\beta_1\leqslant \cdots \leqslant \beta_n$.
From $ \beta_j = \mz(\lambda_j(\xi), \xi_0) $, we can write $\lambda_j(\xi) = (\xi - \xi_0)^{\beta_j} f_j(\xi) $
with $f_j(\xi_0)\neq 0 $ for all $j=1,\ldots, n$. The factorization \eqref{eq:AnalyticEVD} becomes
\begin{align*}
\pA(e^{-i\xi}) =& W(\xi)
\diag((\xi-\xi_0)^{\beta_1},\ldots,(\xi-\xi_0)^{\beta_n})
\diag(f_1(\xi),\ldots,f_n(\xi))
	(W(\xi))^\star\\
	=&\widetilde{E}_{\xi_0}(\xi)
\diag((\xi-\xi_0)^{\beta_1},\ldots,(\xi-\xi_0)^{\beta_n})
	\widetilde{F}_{\xi_0}(\xi),
	\end{align*}
where $\widetilde{E}_{\xi_0}(\xi):= W(\xi) $ and $ \widetilde{F}_{\xi_0}(\xi):= \diag(f_1(\xi),\ldots, f_n(\xi))(W(\xi))^\star$.
From the definition, $\widetilde{E}_{\xi_0}(\xi)$ and $\widetilde{F}_{\xi_0}(\xi)$ are both analytic matrices, and $\det(\widetilde{E}_{\xi_0}(\xi_0))\neq 0$, $\det(\widetilde{F}_{\xi_0}(\xi_0))\neq 0 $.
Hence, the matrices $\widetilde{E}_{\xi_0}(\xi), \widetilde{F}_{\xi_0}(\xi)$ and the sequence $\{\beta_j\}_{j=1}^n$ satisfy all the conditions in Lemma~\ref{lemma:PartialMultiplicity} in the Appendix. By Lemma~\ref{lemma:PartialMultiplicity}, we must have
$
\{\beta_j\}_{j=1}^n = \{\alpha_j\}_{j=1}^n
$.
This completes the proof.
\end{proof}

We now prove Theorem~\ref{thm:extract2} using Theorem~\ref{thm:partialMultiplicity}.

\begin{proof}[Proof of Theorem \ref{thm:extract2}]
Denote the invariant polynomials of the matrix $ \pA(z) $ by $ \pd_1(z),\ldots, \pd_n(z) $.
Define the sequence $ \{\alpha_j\}_{j=1}^n $ by
$ \alpha_j := \mz(\pd_j(z), z_0)$, $ j=1,\ldots, n $.
From the condition in item (2), all $\alpha_j\leqslant 1$. Also, by $\pd_j(z)\mid\pd_{j+1}(z)$ for all $j=1,\ldots,n-1$, we have $\alpha_1\leqslant\cdots\leqslant\alpha_n $. Thus
$$\{\alpha_j\}_{j=1}^n = \{0,\ldots,0,1,\ldots,1\}.$$
	
Taking $z=e^{-i\xi}$, we get a matrix $ \pA(e^{-i\xi}) $ that is analytic of $ \xi \in\R $.
By \cite[Theorem~S6.3]{GohLanRod:1982Matrix}, the analytic Hermite matrix $\pA(e^{-i\xi})$ can also be factorized as
\begin{equation} \label{eq:AnalyticEVDnew}
\pA(e^{-i\xi}) = W(\xi)
\diag(\lambda_1(\xi),\ldots,\lambda_n(\xi))
(W(\xi))^\star,
\end{equation}
	where $W(\xi)$ is a unitary analytic matrix and $\lambda_1(\xi), \ldots, \lambda_n(\xi)$ are analytic functions of $\xi\in\R$.
	Since $z_0\in \T$, we can find some $\xi_0\in \left[ -\pi,\pi \right)$ such that $z_0 = e^{-i\xi_0}$.
	Define the sequence $ \{\beta_j\}_{j=1}^n $ by
	$\beta_j:= \mz(\lambda_j(\xi), \xi_0)$ for all $j=1,\ldots, n$.
	Without loss of generality, we can choose the factorization such that $\beta_1\leqslant \cdots \leqslant \beta_n$.
	According to Theorem~\ref{thm:partialMultiplicity}, we must have
	\[ \{\beta_j\}_{j=1}^n = \{\alpha_j\}_{j=1}^n = \{0,\ldots,0,1,\ldots,1\}. \]

	Let $K$ be the number of times $``1"$ appearing in $ \{\beta_j\}_{j=1}^n$ or $\{\alpha_j\}_{j=1}^n$.
	Recall from the definition of $ \{\alpha_j\}_{j=1}^n$, each $``1"$ corresponds to an elementary divisor $(z-z_0)$. So $K>0$ is the number of times that the elementary divisor $(z-z_0)$ appears. Let us see how the signs of the eigenvalues $\lambda_j(\xi)$ change from the left to the right side of $\xi_0$.
	
	For $j=1,\ldots, n-K$, we have $\beta_j=0$. So $\lambda_j(\xi_0) \neq 0$. Since the eigenvalue $\lambda_j(\xi)$ is a continuous function of $\xi\in\R$, it will not change its sign between the two sides of $\xi_0$, i.e., $\sign(\lambda_j(\xi_0-)) = \sign(\lambda_j(\xi_0+))$.
	
	For $j=n-K+1, \ldots, n$, we have $\beta_j=1$. In this case, $ \lambda_j(\xi_0) = 0 $ and $ \lambda_j'(\xi_0)\neq 0 $. We know that the eigenvalues of a Hermite matrix are all real, so $ \lambda_j(\xi), \lambda_j'(\xi) $ are both real functions of $ \xi\in\R $.
	Hence, $ \lambda_j'(\xi_0) $ is a nonzero real number.
	We have the following two possible situations.
	\begin{enumerate}
		\item If $\lambda_j'(\xi_0) > 0$, then $\lambda_j(\xi)$ is increasing near $\xi_0$. So
		$\lambda_j(\xi_0-) <0$ and $  \lambda_j(\xi_0+)> 0$.
		\item If $\lambda_j'(\xi_0) < 0$, then $\lambda_j(\xi)$ is decreasing near $\xi_0$. So
		$\lambda_j(\xi_0-) >0 $ and $ \lambda_j(\xi_0+)< 0$.
	\end{enumerate}
	
	Since the signature of $\pA(z)$ is constant for all $ z\in\T\setminus\sigma(\pA) $,
	we know that the number of positive eigenvalues
	and the number of negative eigenvalues of $ \pA(e^{-i\xi}) $ will remain unchanged between the two sides of $ \xi_0 $.
	So the above two cases must happen exactly the same number of times.
	That is, $K$ has to be an even integer. And there are exactly $K/2$ number of $\lambda_j(\xi)$ such that
	$ \lambda_j(\xi_0) = 0 $ and $\lambda_j'(\xi_0)>0$.
	Meanwhile, there are exactly $K/2$ number of $\lambda_j(\xi)$ such that $ \lambda_j(\xi_0)=0 $ and $\lambda_j'(\xi_0)<0$.
	The sign of $\lambda_j'(\xi_0)$ here are called the \emph{sign characteristic},
	which was firstly studied in \cite{GohLanRod:1980Spectral} for matrices of polynomials.
	
	Since $ K > 0 $, there exist some $j_1, j_2 \geqslant n-K+1$ such that
	$$\lambda_{j_1}(\xi) = \gamma_1^2(\xi - \xi_0) + \bo((\xi-\xi_0)^2),\qquad
	\lambda_{j_2}(\xi) = -\gamma_2^2(\xi - \xi_0) + \bo((\xi-\xi_0)^2),\qquad
	\mbox{as}~ \xi \rightarrow \xi_0.$$
	for some real $ \gamma_1, \gamma_2 \neq 0 $.
	
	In the eigenvalue decomposition \eqref{eq:AnalyticEVDnew}, $W(\xi)$ being a unitary and analytic matrix on $\xi\in\R$ implies that $W^{-1}(\xi_0)W(\xi) = \mathbf{I}_n + \bo((\xi-\xi_0))$
	as $ \xi \rightarrow \xi_0 $. So, there exists an $n \times n$ analytic matrix $G(\xi)$ such that
	$$ W^{-1}(\xi_0)W(\xi) = \mathbf{I}_n + (\xi-\xi_0)G(\xi), \qquad
	\left( W^{-1}(\xi_0)W(\xi) \right)^\star = \mathbf{I}_n + (\xi-\xi_0)(G(\xi))^\star. $$
	Multiplying constant matrices  $W^{-1}(\xi_0) $ and $W^{-\star}(\xi_0)$ on the left and the right side of \eqref{eq:AnalyticEVDnew} respectively,
	we define $\mathring{\pA}(e^{-i\xi}) $ as
	\begin{align} \label{eq:diagfirstder}
	\mathring{\pA}(e^{-i\xi}):= &
	W^{-1}(\xi_0) \pA(e^{-i\xi})(W(\xi_0))^{-\star}
	= W^{-1}(\xi_0)W(\xi)
\diag(\lambda_1(\xi),\ldots,\lambda_n(\xi))
	(W^{-1}(\xi_0)W(\xi))^\star  \notag \\
	=& (\mathbf{I}_n + (\xi-\xi_0)G(\xi))
\diag(\lambda_1(\xi),\ldots,\lambda_n(\xi))
	(\mathbf{I}_n + (\xi-\xi_0)(G(\xi))^\star) \notag \\
	=& \Lambda(\xi)
	+ (\xi - \xi_0) G(\xi) \Lambda(\xi)
	+ (\xi - \xi_0)\Lambda(\xi) (G(\xi))^\star
	+ (\xi - \xi_0)^2 G(\xi) \Lambda(\xi) (G(\xi))^\star,
	\end{align}
	where $\Lambda(\xi)  := \diag(\lambda_1(\xi), \ldots, \lambda_n(\xi))$.
	Plugging in $ \xi = \xi_0 $, we can directly get
	\begin{equation} \label{eq:EVDFirstOrder}
	\mathring{\pA}(e^{-i\xi_0}) = \Lambda(\xi_0) =
\diag(\lambda_1(\xi_0),\ldots,\lambda_{n-K}(\xi_0), \mathbf{0}_{n\times n})
	\end{equation}
	As we picked $j_1, j_2 \geqslant n-K+1$, the $j_1$-th and the $j_2$-th rows, as well as the $j_1$-th and the $j_2$-th columns of
	$\mathring{\pA}(e^{-i\xi})$ are equal to $\bo((\xi-\xi_0))$ as $\xi \rightarrow \xi_0$.
	
	Now, we will check the lower right $K \times K$ submatrix of $ \mathring{\pA}(e^{-i\xi})$ from \eqref{eq:diagfirstder}.
	Since $ \lambda_{n-K+1}(\xi),\ldots,\lambda_n(\xi) $ are equal to $ \bo((\xi-\xi_0)) $ as $ \xi \rightarrow \xi_0 $,
	the lower right $K \times K$ submatrices of the second and the third term on the right hand side of \eqref{eq:diagfirstder}
	are both $ \bo((\xi-\xi_0)^2) $ as $ \xi \rightarrow \xi_0 $.
	Hence, the summation of the four terms on the right hand side of \eqref{eq:diagfirstder} yields:
	\begin{align} \label{eq:EVDSecondOrder}
	\mathring{\pA}_{(n-K+1):n, (n-K+1):n}(e^{-i\xi})
	= &
\diag(\lambda_{n-K+1}(\xi),\ldots,\lambda_n(\xi))
	+ \bo((\xi-\xi_0)^2) + \bo((\xi-\xi_0)^2) + \bo((\xi-\xi_0)^2) \notag \\
	= &
\diag(\lambda_{n-K+1}(\xi),\ldots,\gamma_1^2(\xi - \xi_0),\ldots,-\gamma_2^2(\xi - \xi_0),\ldots,\lambda_n(\xi))
	+ \bo((\xi-\xi_0)^2),
	\end{align}
	as $\xi \rightarrow \xi_0$. The $ \gamma_1^2(\xi - \xi_0) $ and $-\gamma_2^2(\xi - \xi_0) $ terms appear at the $j_1$-th and the $j_2$-th diagonal positions respectively.
	Now, we can use the following matrix $ V $ to cancel the first order term at the $(j_1, j_1)$ position of $ \mathring{\pA}(e^{-i\xi}) $. Define the $n\times n$  matrix $V$ as
	\begin{equation} \label{eq:V}
	V:=
	\begin{bmatrix}
	1 &        &               &        &               &        &  \\
	& \ddots &               &        &               &        &  \\
	&        & \gamma_1^{-1} &        & \gamma_2^{-1} &        &  \\
	&        &               & \ddots &               &        &  \\
	&        &               &        & 1             &        &  \\
	&        &               &        &               & \ddots &  \\
	&        &               &        &               &        & 1
	\end{bmatrix},
	\end{equation}
	where $ V\mathring{\pA} $ corresponds to dividing the $j_1$-th row of $ \mathring{\pA} $ by $\gamma_1$, then adding $\gamma_2^{-1}$ times the $j_2$-th row to the $j_1$-th row of $ \mathring{\pA} $.
	Taking symmetric operations on both rows and columns of $\mathring{\pA}(e^{-i\xi})$, we define
	$ \breve{\pA}(e^{-i\xi}) := V \mathring{\pA}(e^{-i\xi}) V^\star$.
	Then the lower-right $K \times K$ submatrix of $ \breve{\pA}(e^{-i\xi})$ becomes:
	\begin{align*}
	&\breve{\pA}_{(n-K+1):n, (n-K+1):n}(e^{-i\xi})
	=
	\begin{bmatrix}
	\lambda_{n-K+1}(\xi) &        &                        &        &                          &        &  \\
	& \ddots &                        &        &                          &        &  \\
	&        & 0                      &        & -\gamma_2(\xi - \xi_0)   &        &  \\
	&        &                        & \ddots &                          &        &  \\
	&        & -\gamma_2(\xi - \xi_0) &        & -\gamma_2^2(\xi - \xi_0) &        &  \\
	&        &                        &        &                          & \ddots &  \\
	&        &                        &        &                          &        & \lambda_n(\xi)
	\end{bmatrix}
	+ \bo((\xi-\xi_0)^2).
	\end{align*}
	Thus, the $(j_1, j_1)$-diagonal entry of $\breve{\pA}(e^{-i\xi}) $ is $\bo((\xi-\xi_0)^2)$ as $\xi\rightarrow \xi_0$.
	From the definition of $ \breve{\pA} $, we see that similar to $\mathring{\pA}(e^{-i\xi})$, the $j_1$-th and the $j_2$-th rows, as well as the $j_1$-th and the $j_2$-th columns of $\breve{\pA}(e^{-i\xi})$ are still $\bo((\xi-\xi_0))$ as $\xi \rightarrow \xi_0$.

	Finally, we can change back to Laurent polynomials. The matrix $ \breve{\pA}(z) $ of Laurent polynomials is written as
	$$\breve{\pA}(z) = V \mathring{\pA}(z) V^\star = V W^{-1}(\xi_0) \pA(z)W^{-\star}(\xi_0) V^\star. $$
	Since the $j_1$-th row and the $j_1$-th column of $ \breve{\pA}(e^{-i\xi}) $ are $ \bo((\xi-\xi_0)) $, we know that $(z-z_0)$ divides both the $j_1$-th row and the $j_1$-th column of $ \breve{\pA}(z) $.
	Also, the fact that the $(j_1, j_1)$ entry of $\breve{\pA}(e^{-i\xi}) $ is $\bo((\xi-\xi_0)^2)$ implies that
	$(z-z_0)^2$ divides the $(j_1, j_1)$ entry of $\breve{\pA}(z)$. So we can factor out $(z-z_0)$ from the $j_1$-th row and $(z-z_0)^\star$ from the $j_1$-th column simultaneously to get
	$$ \breve{\pA}(z) = \pD_{j_1, 1}(z)\widetilde{\pA}(z)\pD_{j_1, 1}^\star(z), $$
	for some $n \times n$ Hermite matrix $\widetilde{\pA}(z)$ of Laurent polynomials and $\pD_{j_1, 1}(z)$ is defined as \eqref{eq:Diagk}. Thus, we have
	\begin{align*}
	\pA(z) =& W(\xi_0)V^{-1}\breve{\pA}(z)V^{-\star}W^\star(\xi_0)
	= W(\xi_0)V^{-1} \pD_{j_1, 1}(z)\widetilde{\pA}(z)\pD_{j_1, 1}^\star(z)  V^{-\star}W^\star(\xi_0)
	= \pU(z) \widetilde{\pA}(z) \pU^\star(z),
	\end{align*}
	where $\pU(z):= W(\xi_0)V^{-1}\pD_{j_1, 1}(z) $. Observe that
	\begin{align*}
	\len(\det(\widetilde{\pA}(z)))
	= \len(\det(\pA(z))) - \len(\det(\pU(z))) - \len(\det(\pU^*(z)))
	= \len(\det(\pA(z))) - 2.
	\end{align*}
	So $ \pU(z) $ and $ \widetilde{\pA}(z) $ satisfy all the requirements. This proves Theorem~\ref{thm:extract2}.
\end{proof}

Now we are ready to prove Theorem~\ref{thm:const-sig}.

\begin{proof}[Proof of Theorem \ref{thm:const-sig}]
If  $\len(\det(\pA(z)))>0$, then $\pA(z)$ has some elementary divisor $(z-z_0)^\alpha$ with $z_0\neq 0$ and $ \alpha \in \N $. Let $\pA_0(z):=\pA(z)$. For $j\geqslant 0$, if $\pA_j(z)$ has some elementary divisor $(z-z_0)^\alpha$ with $z_0\in \C\setminus\T\setminus\{0\}$ or $\alpha>1$, apply Theorem~\ref{thm:extract1} to get a factorization of $ \pA_j(z) $ as
$\pA_j(z)=\pU_{j+1}(z)\pA_{j+1}(z)\pU^\star_{j+1}(z) $, for some $n\times n$ matrices $\pU_{j+1}(z)$ and $\pA_{j+1}(z)$ of Laurent polynomials satisfying $\pA_{j+1}^\star(z) = \pA_{j+1}(z)$ and $ \len(\pA_{j+1}(z)) < \len(\pA_{j}(z)) $.
If all the elementary divisors $(z-z_0)^\alpha$ of $\pA_j(z)$ has degree $\alpha =1$ and $z_0\in\T$, we can apply Theorem~\ref{thm:extract2} to still get the factorization $\pA_j(z)=\pU_{j+1}(z)\pA_{j+1}(z)\pU^\star_{j+1}(z) $ with  $\pA_{j+1}^\star(z) = \pA_{j+1}(z)$ and $ \len(\pA_{j+1}(z)) < \len(\pA_{j}(z)) $.
Reset $ j$ by $j+1 $ and repeat the steps, until $ \len(\det(\pA_j(z)))=0 $.
This iteration will stop after finite number of steps, since $\len(\det(\pA(z)))$ is finite and $\len(\det(\pA_j(z)))$ is strictly decreasing after each step. Hence, we can get a factorization as
$$ \pA(z) =  \pU_{1}(z)\cdots\pU_{k}(z)
\pA_{k}(z)\pU^\star_{k}(z)\cdots\pU^\star_{1}(z),$$
where $ \pA_{k}(z)$ has no elementary divisors, i.e., $\len(\det(\pA_{k}(z)))=0$.

In this case, it is proved by Theorem \ref{thm:unimodular} that $\pA_k(z)$ can be factorized as $\pA_k(z)=\pU_{k+1}(z)\pD\pU^\star_{k+1}(z) $
for some $n\times n$ matrix $\pU_{k+1}(z)$ of Laurent polynomials
and $\pD=\diag(\mathbf{I}_{\nu_1}, -\mathbf{I}_{\nu_2})$ is an $n \times n$ constant diagonal matrix
for some nonnegative integers $ \nu_1$ and $ \nu_2 $ satisfying $ \nu_1 + \nu_2 = n $.
Define
$\pU(z):= \prod_{j=1}^{k+1}\pU_j(z)$. Then
$ \pA(z)=\pU(z)\pD\pU^\star(z) $
holds.

Also, notice that for all $ z_0 \in \T\setminus\sigma(\pA(z)) $, $ \pU(z_0) $ is a nonsingular matrix. By Sylvester's law of inertia,
$ \nu_1 = \nu_+(\pA(z_0)) = \nu_+ $ and $ \nu_2 = \nu_-(\pA(z_0)) = \nu_- $.
This completes the proof.
\end{proof}

All the steps, except finding $W(\xi)$ in \eqref{eq:AnalyticEVDnew}, in the above proof of Theorem~\ref{thm:const-sig} are constructive.
The existence of $W(\xi)$ in \eqref{eq:AnalyticEVDnew} is guaranteed by \cite[Theorem~S6.3]{GohLanRod:1982Matrix},
which is not constructive and is very complicated.
We now provide the following simple constructive algorithm to
realize the generalized spectral factorization in Theorem~\ref{thm:const-sig}.
Steps (S3) and (S5) simply follow the proof of Theorem~\ref{thm:extract1} and Algorithm~\ref{algo:unimodular}, respectively.
We use step (S4) to find the factorization in Theorem~\ref{thm:extract2}. The idea of step (S4) is that for $ z_0 = e^{-i\xi_0}\in\T $, where all the elementary divisors have single root, we can easily calculate the first two coefficient matrices of the Taylor expansion $ \pA(e^{-i\xi}) = C_0 + C_1 (\xi - \xi_0) + \bo((\xi-\xi_0)^2) $ as $ C_0 = \pA(z_0) $ and $ C_1 = -i z_0 \pA'(z_0) $. Then if restricted to the null space of $ C_0 $, the matrix $ C_1 $ must have half of the eigenvalues being positive and the other half of the eigenvalues being negative. Thus, we can find a nonsingular matrix $ V $ such that $ VC_1 V^\star $ has one zero on some diagonal position, and hence we can factor out $ (z - z_0) $ from the row and $ (z-z_0)^\star $ from the column of $ V\pA(z)V^\star $ simultaneously.
See the Appendix for the proof of the following algorithm.

\begin{algorithm} \label{algo:Const-Sig}
	Input an $n\times n$ Hermite matrix $\pA(z)$ of Laurent polynomials with constant signature on $z\in\T\bs\sigma(\pA)$ such that
	$\det(\pA(z))\not\equiv 0$.

\begin{enumerate}	
\item[(S0)] Initialization. Set $\widetilde{\pA}(z):=\pA(z)$ and $\pU(z):= \mathbf{I}_n$.

\item[(S1)] Compute the Smith Normal Form $\pD(z) = \diag(\pd_1(z), \ldots, \pd_n(z))$ of $\widetilde{\pA}(z)$ and get a decomposition $\widetilde{\pA}(z) = \pE(z)\pD(z)\pF(z)$,
where $\pE(z)$ and $\pF(z)$ are unimodular matrices of Laurent polynomials.

\item[(S2)] If $\pD(z)$ is a constant matrix, then go to  (S5).
Otherwise, redefine
\begin{equation} \label{eq:SNFchange}
\widetilde{\pA}(z):=\pE^{-1}(z)\widetilde{\pA}(z)\pE^{-\star}(z) =
\diag(\pd_1(z), \ldots, \pd_n(z))\pF(z)\pE^{-\star}(z),
\end{equation}
and update/replace $\pU(z)$ by $\pU(z)\pE(z)$.

\item[(S3)] {\bfseries For $j$ from $1$ to $n$:}\\
\hspace*{0.5 cm} Factorize $\pd_j(z)= \prod_{k=1}^{n_j}(z-z_{j,k})^{\alpha_{j,k}}. $\\
\hspace*{0.5 cm} {\bfseries If} there exists some factor $ (z-z_{j,k})^{\alpha_{j,k}} $ with $z_{j,k} \in (\C\setminus\{0\})\setminus \T$:
\begin{enumerate} [leftmargin=0.7 in]
	\item redefine $\pU(z) $ by multiplying its $j$-th column by $(z-z_{j,k})^{\alpha_{j,k}}$;
	\item redefine $\widetilde{\pA}(z)$ by dividing its $j$-th row by $(z-z_{j,k})^{\alpha_{j,k}} $,  and dividing its $j$-th column by $(z^{-1}-\overline{z_{j,k}})^{\alpha_{j,k}}$;
	\item {\bfseries break} the {\bfseries for} loop, and go back to  (S1);
\end{enumerate}

{\bfseries else if} there exists some factor $ (z-z_{j,k})^{\alpha_{j,k}} $ with $z_{j,k} \in \T$ and $ \alpha_{j,k} \geqslant 2$:
	\begin{enumerate}[leftmargin=0.7 in]
		\item redefine $\pU(z) $ by multiplying its $j$-th column by $(z-z_{j,k})^{\lfloor\alpha_{j,k}/2\rfloor}$;
		\item redefine $\widetilde{\pA}(z)$ by dividing its $j$-th row by $(z-z_{j,k})^{\lfloor\alpha_{j,k}/2\rfloor} $,  and dividing its $j$-th column by $(z^{-1}-\overline{z_{j,k}})^{\lfloor\alpha_{j,k}/2\rfloor}$;
		\item {\bfseries break} the {\bfseries for} loop, and go back to  (S1);
	\end{enumerate}
\hspace*{0.5 cm} {\bfseries end if;}\\
{\bfseries end for;}
		
\item[(S4)] If the {\bfseries for} loop doesn't break from any conditions in  (S3), then all the elementary divisors will have roots on $\T$ with degree equal to $1$.
Pick one of the elementary divisors $(z-z_0)$.
Suppose that it is contained in the last $K$ invariant polynomials $d_{n-K+1}(z), \ldots, d_n(z)$:

\begin{enumerate}
	\item From \eqref{eq:SNFchange}, we see the last $K$ columns and the last $K$ rows of $\widetilde{\pA}(z_0)$ have to be $0$. Consider the constant Hermite matrix $-iz_0\widetilde{\pA}'(z_0)$. Take its lower right $K \times K$ submatrix, denoted as $\pA_K$, and find its eigenvalue decomposition as $\pA_K:=\pU_1\Gamma\pU^\star_1$, for some unitary matrix $\pU_1$ and $\Gamma=\diag(\gamma_1^2, -\gamma_2^2,\ldots, \gamma_K)$. Then the eigenvalues in $\Gamma$ must be all nonzero, while $K/2$ of them are positive and $K/2$ of them are negative. Arrange them such that the first one is positive and the second one is negative.
			Redefine $\widetilde{\pA}(z):=\diag(\mathbf{I}_{n-K}, \pU_1^{-1})\widetilde{\pA}(z)\diag(\mathbf{I}_{n-K}, \pU_1^{-\star})$
			and $ \pU(z):=\pU(z)\diag(\mathbf{I}_{n-K}, \pU_1)$.
			
			\item Take $\pU_2:=\diag\left(\mathbf{I}_{n-K},
			\begin{bmatrix}
			\gamma_1^{-1} & \gamma_2^{-1} \\
			0 & 1
			\end{bmatrix}, \mathbf{I}_{K-2} \right)$.
			Redefine $\widetilde{\pA}(z):=\pU_2\widetilde{\pA}(z)\pU^\star_2$ and $\pU(z):=\pU(z)\pU_2^{-1}$.
			
			\item Redefine $\widetilde{\pA}(z)$ by dividing its $(n-K+1)$-th row by $(z-z_0)$  and dividing its $(n-K+1)$-th column by $(z^{-1}-\overline{z_0})$. Redefine $\pU(z) $ by multiplying its $(n-K+1)$-th column by $(z-z_0)$.
		\end{enumerate}
		Go back to  (S1).

		\item[(S5)] Finalize: Since $\widetilde{\pA}(z)$ has no elementary divisor, apply Algorithm \ref{algo:unimodular} to get the factorization
		$\widetilde{\pA}(z)= \widetilde{\pU}(z)\pD\widetilde{\pU}^\star(z)$.
		Redefine $\pU(z):=\pU(z)\widetilde{\pU}(z)$.
		Output $\pU(z)$ and $\pD$. Then $\pA(z)=\pU(z)\pD\pU^\star(z)$ must hold.
	\end{enumerate}
\end{algorithm}

Let us make some interesting remarks and consequences about Theorem~\ref{thm:partialMultiplicity}.
For a Hermite matrix  $ \pA(z) $ of Laurent polynomials, although we know from Theorem~\ref{thm:partialMultiplicity} that the analytic eigenvalues $ \lambda_1(\xi),\ldots, \lambda_n(\xi) $ of $ \pA(e^{-i\xi}) $ have some relationship to the invariant polynomials of $ \pA(z) $,
we cannot expect $ \lambda_1(\xi),\ldots, \lambda_n(\xi) $ to be Laurent polynomials in general.
Actually, the following example shows that the analytic functions $ \lambda_1(\xi),\ldots, \lambda_n(\xi) $ might not be even $ 2\pi$-periodic functions of $ \xi\in\R $.

\begin{example} \label{ex:AnalyticalEig}
{\rm
	Consider the same matrix $ \pA(z) $ as in Example~\ref{ex:SingleRoot}.
	Solving $ \det(\pA(e^{-i\xi}) - \lambda \mathbf{I}_{2})=0 $, we can find two  analytic functions that are eigenvalues of $ \pA(e^{-i\xi}) $ as
	$ \lambda_1(\xi) = -\lambda_2(\xi) = 4\sin(\xi/2) $. They are both $ 4\pi $-periodic functions of $ \xi\in \R $, and we cannot find two eigenvalues of $ \pA(e^{-i\xi}) $ that are both analytic and $ 2\pi $-periodic functions of $ \xi\in\R $.
	Also, as calculated in Example~\ref{ex:SingleRoot}, the two invariant polynomials of $ \pA(z) $ are $ \pd_1(z) = \pd_2(z) = z-1 $.
	Take $ \xi_0=0 $ and $ z_0 = e^{-i\xi_0}= 1 $, we can calculate $ \alpha_j := \mz(\pd_j(z), 1) = 1 $, $ \beta_j := \mz(\lambda_j(e^{-i\xi}), 0) = 1 $ for $ j = 1,2 $.
}  \end{example}

Since the sequence $ \{\beta_j\}_{j=1}^n $ in Theorem~\ref{thm:partialMultiplicity} is related to the sign change of the eigenvalues $ \lambda_j(\xi) $, we have the following corollary for the positive semidefinite matrix $ \pA(z) $ of Laurent polynomials.

\begin{cor} \label{cor:SPDEvenEleDiv}
	Suppose that $\pA(z)$ is a Hermite matrix of Laurent polynomials such that $ \pA(z)\geqslant 0 $ for all $z\in \T$. Then all its elementary divisors $(z-z_0)^{\alpha}$ with $z_0\in \T$ must have even degree, i.e., $\alpha\in 2\Z$.
\end{cor}

\begin{proof}
	Since $ z_0\in \T $, we can find some $ \xi_0\in\R $ such that $ z_0 = e^{-i\xi_0} $.
	Suppose that $ \lambda_1(\xi),\ldots,\lambda_n(\xi) $ are the eigenvalues of $ \pA(e^{-i\xi}) $ which
	are also analytic functions of $ \xi\in\R $.
	Define the sequences $ \{\alpha_j\}_{j=1}^n $ and $ \{\beta_j\}_{j=1}^n $ as in Theorem~\ref{thm:partialMultiplicity}.
	By Theorem~\ref{thm:partialMultiplicity}, we must have $ \{\beta_j\}_{j=1}^n = \{\alpha_j\}_{j=1}^n $.
	
	Since $ \pA(e^{-i\xi}) $ is positive semidefinite for all $ \xi\in\R $, that is,
	$ \lambda_j(\xi) $ will not change sign across $ \xi_0 $ for all $ j=1,\ldots,n $, we conclude that
	\[ \beta_j = \mz(\lambda_j(\xi), \xi_0) \in 2\Z, \qquad \forall~ j = 1,\ldots, n. \]
	So $ \alpha_j \in 2\Z $ for all $ j = 1, \ldots, n $.
	From the definition of $ \alpha_j $, we know that $ \{\alpha_j\}_{j=1}^n $
	are just the degrees of elementary divisors $ (z-z_0)^\alpha $ in each invariant polynomial.
	So all such $ \alpha $ satisfy $ \alpha \in 2\Z $.
\end{proof}

\section{Proof of Theorem~\ref{thm:nonconst-sig} on Generalized Matrix Spectral Factorization}

In this section, we prove Theorem~\ref{thm:nonconst-sig}.
To prove the necessity part of Theorem~\ref{thm:nonconst-sig}, we need the following result.

\begin{lemma} \label{thm:SylvesterInertia}
Suppose that an $n \times n$ Hermite matrix $A$ can be decomposed in the following way
\begin{equation} \label{eq:SylvesterConst}
A = U
\begin{bmatrix}
\mathbf{I}_{m_+} &  \\
 & -\mathbf{I}_{m_-}
\end{bmatrix} U^\star,
\end{equation}
where $U$ is an $n \times m$ matrix and $\mathbf{I}_{m_+}$, $\mathbf{I}_{m_-}$ are the identity matrices of size $m_+$ and $m_-$, respectively, such that $m_++m_- = m$. Then
$$m_+ \geqslant \nu_+(A), \qquad m_- \geqslant \nu_-(A).$$
\end{lemma}

\begin{proof}
First, we consider the case that $A$ is nonsingular. In this case, the decomposition \eqref{eq:SylvesterConst} forces that all the three matrices on the right hand side of \eqref{eq:SylvesterConst} must have rank at least $ n $.
So $m\geqslant n$ and $U$ must have full row rank.

If $m=n$, then $U$ is a nonsingular square matrix. By Sylvester's law of inertia,
$$m_+ = \nu_+(A),\qquad m_- = \nu_-(A).$$

If $m > n$, since $U$ has full row rank, we can add $ m-n $ more rows to $U$ to get $ \widetilde{U} $ such that
$\widetilde{U} := \begin{bmatrix}
U \\ V
\end{bmatrix}$
is an $m \times m$ nonsingular square matrix. Then the $m \times m$ matrix
$\widetilde{A} := \widetilde{U}
\begin{bmatrix}
\mathbf{I}_{m_+} &  \\
 & -\mathbf{I}_{m_-}
\end{bmatrix}
\widetilde{U}^\star$
has $A$ on the top left corner:
\begin{equation} \label{eq:Sylvestertilde}
\widetilde{A} := \widetilde{U}\begin{bmatrix}
\mathbf{I}_{m_+} &  \\
 & -\mathbf{I}_{m_-}
\end{bmatrix}
\widetilde{U}^\star
= \begin{bmatrix}
U \\ V
\end{bmatrix}
\begin{bmatrix}
\mathbf{I}_{m_+} &  \\
 & -\mathbf{I}_{m_-}
\end{bmatrix}
\begin{bmatrix}
U^\star & V^\star
\end{bmatrix}
=\begin{bmatrix}
A & B^\star \\
B & C
\end{bmatrix}
\end{equation}
for some $(m-n)\times n$ matrix $B$ and some $(m-n)\times (m-n)$ matrix $C$.
Define nonsingular $m\times m$ matrix
$W := \begin{bmatrix}
\mathbf{I}_n & \mathbf{0} \\
-BA^{-1} & \mathbf{I}_{m-n}
\end{bmatrix}$, and let $ \mathring{A} := W\widetilde{A}W^\star $.
Plugging \eqref{eq:Sylvestertilde} in $ \mathring{A} $, we can directly calculate that
\begin{equation} \label{eq:Sylvesterring}
\mathring{A} := W\widetilde{A}W^\star
= \begin{bmatrix}
\mathbf{I}_n & \mathbf{0} \\
-BA^{-1} & \mathbf{I}_{m-n}
\end{bmatrix}
\begin{bmatrix}
A & B^\star \\
B & C
\end{bmatrix}
\begin{bmatrix}
\mathbf{I}_n & -A^{-\star}B^\star \\
\mathbf{0} & \mathbf{I}_{m-n}
\end{bmatrix}
= \begin{bmatrix}
A & \mathbf{0} \\
\mathbf{0} & D
\end{bmatrix},
\end{equation}
where the $(m-n)\times (m-n)$ matrix $D := C-BA^{-1}B^\star$.
From \eqref{eq:Sylvesterring}, we see that the eigenvalues of $\mathring{A}$ are just the eigenvalues of $A$ combined with the eigenvalues of $D$.
So
\begin{equation} \label{eq:sigAring}
\nu_+(\mathring{A}) \geqslant \nu_+(A),\qquad
\nu_-(\mathring{A}) \geqslant \nu_-(A).
\end{equation}
Also, from the definition of $ \widetilde{A} $ and $ \mathring{A} $ in \eqref{eq:Sylvestertilde} and \eqref{eq:Sylvesterring},
we deduce that
\begin{align} \label{eq:Sylvester}
\mathring{A} =
W \widetilde{A} W^\star
= W\widetilde{U}
\begin{bmatrix}
\mathbf{I}_{m_+} &  \\
 & -\mathbf{I}_{m_-}
\end{bmatrix}
\widetilde{U}^\star W^\star
=
W\widetilde{U}
\begin{bmatrix}
\mathbf{I}_{m_+} &  \\
 & -\mathbf{I}_{m_-}
\end{bmatrix}
(W\widetilde{U})^\star.
\end{align}
Since $W\widetilde{U}$ is an $ m \times m $ nonsingular matrix, by Sylvester's law of inertia again,
\eqref{eq:Sylvester} implies that
\begin{equation} \label{eq:sigAring2}
\nu_+(\mathring{A}) = m_+ ,\qquad  \nu_-(\mathring{A}) = m_-.
\end{equation}
Combining \eqref{eq:sigAring} and \eqref{eq:sigAring2}, we get
$m_+ \geqslant \nu_+(A) $  and $ m_- \geqslant \nu_-(A)$.
This proves the lemma for the case that $A$ is nonsingular.

For the case that $A$ is singular, we can find its eigenvalue decomposition first:
$$PAP^\star =
\begin{bmatrix}
\Lambda & \\
 & \mathbf{0}
\end{bmatrix},$$
where $\Lambda$ is a $k\times k$ nonsingular diagonal matrix containing all the nonzero eigenvalues of $A$
and $P$ is an $n \times n$ unitary matrix. Plugging \eqref{eq:SylvesterConst} into the above decomposition:
$$ \begin{bmatrix}
\Lambda & \\
 & \mathbf{0}
\end{bmatrix}
=PAP^\star
=PU
\begin{bmatrix}
\mathbf{I}_{m_+} &  \\
 & -\mathbf{I}_{m_-}
\end{bmatrix} U^\star
P^\star
= Q\begin{bmatrix}
\mathbf{I}_{m_+} &  \\
 & -\mathbf{I}_{m_-}
\end{bmatrix} Q^\star,
$$
where $Q:=PU$. We define $\widetilde{Q}$ by removing the last $n-k$ rows of $Q$. Then the above equation implies:
$$ \Lambda = \widetilde{Q}
\begin{bmatrix}
\mathbf{I}_{m_+} &  \\
 & -\mathbf{I}_{m_-}
\end{bmatrix}
\widetilde{Q}^\star. $$
Since $\Lambda$ is nonsingular, we know from the previously proved case that
$$m_+ \geqslant \nu_+(\Lambda) = \nu_+(A),\qquad m_- \geqslant \nu_-(\Lambda) = \nu_-(A).$$
This proves the lemma for the case that $A$ is a singular matrix.
\end{proof}

\begin{proof}[Proof of Theorem~\ref{thm:nonconst-sig}]
Necessity. $ \pA^\star(z) = \pA(z) $ implies that for all $ z_0\in \T $, $ \pA(z_0) $ is a Hermite matrix and $ \pU^\star(z_0) = (\pU(z_0))^\star $ holds. Hence, we know from Lemma~\ref{thm:SylvesterInertia} that the decomposition $ \pA(z) = \pU(z)\diag(\mathbf{I}_{m_1}, -\mathbf{I}_{m_2})\pU^\star(z) $ yields $ m_1 \geqslant \nu_+(\pA(z_0)) $ and $ m_2 \geqslant \nu_-(\pA(z_0)) $. Considering all $ z_0\in \T $, we see that \eqref{eq:LargeSig} holds. This proves the necessity part of Theorem~\ref{thm:nonconst-sig}.

To prove the sufficiency part, we first consider the case that $ \det(\pA(z))\not\equiv 0 $, where $ \sigma(\pA) $ is a finite subset of $ \C\bs\{0\} $. The degenerate case is proved later using the Smith Normal Form of $ \pA(z) $.

Suppose that the claim holds for
$$ m_1 =  \max_{z\in\T}\nu_+(\pA(z)), \qquad
m_2 = \max_{z\in\T}\nu_-(\pA(z)). $$
Then $ \pA(z) = \widetilde{\pU}(z)\widetilde{\pD}\widetilde{\pU}^\star(z) $ is obviously true with
$ \widetilde{\pU}(z) :=
[\mathbf{0}_{n\times s_1}, \pU(z), \mathbf{0}_{n\times s_2}] $
and $ \widetilde{\pD}:=
\diag(\mathbf{I}_{s_1+m_1},-\mathbf{I}_{s_2+m_2})
$, for any integers $ s_1, s_2 \geqslant 0 $.
Therefore, we only need to prove the claim for $m_1$ and $m_2$ equal to the lower bounds in \eqref{eq:LargeSig}.
Define
\begin{equation} \label{eq:n+n-}
n_+:= \max_{z\in\T}\nu_+(\pA(z)),\qquad n_-:= \max_{z\in\T}\nu_-(\pA(z)).
\end{equation}
If the signature of $ \pA(z) $ is constant on $\T\bs\sigma(\pA)$, that is, $ \nu_+(\pA(z)) $ and $ \nu_-(\pA(z)) $ are both constant on $ z\in  \T\bs\sigma(\pA)$, then by Lemma~\ref{lem:maxminT}, for all $ z_0\in \T\bs\sigma(\pA) $,
\[ \nu_+(\pA(z_0)) = \max_{z\in \T \bs\sigma(\pA)}\nu_+(\pA(z)) = n_+, \qquad
\nu_-(\pA(z_0)) = \max_{z\in \T \bs\sigma(\pA)}\nu_-(\pA(z)) = n_-. \]
Hence, $ n_+ + n_- = n $ and the result is proved by Theorem \ref{thm:const-sig}.
If $\sig(\pA(z))$ is not constant on $\T\bs\sigma(\pA)$, we have $m_0:= n_+ + n_- > n$. In the following, we will construct $(m_0 - n)$ Laurent polynomials $\mu_1(z),\ldots, \mu_{m_0 - n}(z)$ such that the Hermite matrix
\begin{equation} \label{eq:extend}
\widetilde{\pA}(z):=
\diag(\pA(z),\mu_1(z),\ldots,\mu_{m_0 - n}(z))
\end{equation}
has constant signature on $\T\setminus\sigma(\widetilde{\pA})$.

Since $ \det(\pA(z)) $ is a Laurent polynomial that is not identically zero, $\{z_1,\ldots, z_K\}:=\sigma(\pA) \cap \T $ contains only finite number of points on $\T$. So $ \{z_1,\ldots, z_K\} $ cuts $\T$, which is the unit circle in the complex plane, into $K$ connected open segments: $\Gamma_1, \ldots, \Gamma_K $,
such that
\begin{enumerate}
\item $\bigcup_{j=1}^K \Gamma_j \bigcup \{z_l\}_{l=1}^K = \T $;
\item Pairwise disjoint: $\Gamma_j \cap \{z_l\}_{l=1}^K $ is empty, $\Gamma_j \cap \Gamma_k$ is empty for all $j,k=1,\ldots,K$, $j\neq k$;
\item Both endpoints of $\Gamma_j $ are contained inside $\{z_l\}_{l=1}^K$, denote them by $z_{j,1}$ and $z_{j,2}$, for all $j=1,2,\ldots,K$.
\end{enumerate}

We can choose all the eigenvalues $\lambda_1(\xi), \ldots, \lambda_n(\xi)$ of  $\pA(e^{-i\xi})$ to be  analytic functions of $\xi\in\R$. In each $ \Gamma_j $,
since $ \det(\pA(e^{-i\xi})) = \prod_{k=1}^{n}\lambda_k(\xi) \neq 0 $, none of the $ \lambda_k(\xi) $ will attain zero. As nonzero continuous functions on an open interval, all $ \lambda_k(\xi) $ will not change signs within each $ \Gamma_j $. Thus $\nu_+(\pA(z))$ and $\nu_-(\pA(z))$ remain constant on each $\Gamma_j$.

For each $\Gamma_j$, define a function
$$ \eta_j(z):= (z_{j,1}z_{j,2})^{-\frac{1}{2}}z^{-1}(z-z_{j,1})(z-z_{j,2}), \qquad
j = 1,\ldots, K. $$
The square root of $ z_{j,1}z_{j,2} $ is chosen in the complex plane, where the two possible solutions only differ by a $``-"$ sign.
For both solutions, we can directly verify that $ \eta_j^\star(z) = \eta_j(z) $. So $ \eta_j(z) $ is a real function for all $ z\in \T $.

Since the signature of $ \pA(z) $ is not constant for all $ z\in\T\bs \sigma(\pA) $, $ \T $ contains more than one open segments $ \Gamma_j $. So $ z_{j,1}\neq z_{j,2} $ and both $ z_{j,1} $ and $ z_{j,2} $ are single roots of $ \eta_j(z) $.
Hence $ \eta_j(z) $ will have different signs between two sides of $ z_{j,1} $ and $ z_{j,2} $ on $ \T $.
Therefore, in calculation of the square root of $ z_{j,1}z_{j,2} $,
we can just choose the solution such that $ \eta_j(z)>0 $ for all $ z\in \Gamma_j $
and $ \eta_j(z) < 0 $ for all $ z\in \T\setminus \Gamma_j\setminus \{z_{j,1}, z_{j,2}\} $.
In summary, $ \eta_j(z) $ satisfies
\begin{enumerate}
\item $\eta_j(z)$ is real for all $z\in \T$;

\item $\eta_j(z)>0 $ for all $z\in \Gamma_j$ and $\eta_j(z)<0 $ for all $z\in \Gamma_k $, $ k\neq j $.
\end{enumerate}

Let us construct functions $\mu_k(z)$ recursively for $k=1,\ldots, m_0-n $, such that \eqref{eq:extend} has constant signature on $ z\in\T\setminus\sigma(\widetilde{\pA}) $. Start with $\pA_0(z):=\pA(z)$ and $k=1$.
In order to have our following construction work, we only need to verify two conditions before the start of each new iteration:
\begin{enumerate}[label=(\roman*)]
\item $ \pA_{k-1}(z) $ is a Hermite matrix of Laurent polynomials satisfying
$ \max_{z\in\T}\nu_+(\pA_{k-1}(z)) = n_+$ and $\max_{z\in\T}\nu_-(\pA_{k-1}(z)) = n_- $,
where $n_+$ and $n_-$ are defined in \eqref{eq:n+n-}.
\item $ k \leqslant m_0 - n $.
\end{enumerate}
They are obviously true for $ k=1 $.

Define an index set $J:=\{j\setsp \nu_-(\pA_{k-1}(z))=n_- ~\mbox{for all}~ z\in \Gamma_j \}$.
Now, take
$$ \mu_k(z):= (-1)^{|J|+1}\prod_{j\in J}\eta_j(z), \qquad
\pA_k(z):=\begin{bmatrix}
\pA_{k-1}(z) &  \\
 & \mu_k(z)
\end{bmatrix}.  $$
Since all $ \eta_j(z) $ are real functions on $ z\in\T $, $\mu_k^\star(z)=\mu_k(z)$ is also real on $\T$. From $ \pA_{k-1}^\star(z) = \pA_{k-1}(z) $ in item (i), the matrix $\pA_k(z)$ is also a Hermite matrix of Laurent polynomials.
By the definition of $\mu_k(z)$, we can directly verify from the sign of $ \eta_j(z) $ that
$\mu_k(z)> 0$ for all $z\in \cup_{j\in J} \Gamma_j$,
and $\mu_k(z)<0$ for all $z\in \cup_{j\notin J}\Gamma_j$.
For $z\in\T$, the eigenvalues of $\pA_k(z)$ are just all the eigenvalues of $\pA_{k-1}(z)$, combined with $\mu_k(z)$.
Now, let us calculate $\nu_+(\pA_k(z))$ and $\nu_-(\pA_k(z))$ on each $\Gamma_j$.
\begin{itemize}
\item For $z\in \bigcup_{j\in J} \Gamma_j$, since $\mu_k(z)>0$, we have
$\nu_-(\pA_k(z)) = \nu_-(\pA_{k-1}(z)) = n_-$.
By item (ii), we know that $k\leqslant m_0-n = n_++n_--n$, and hence
\begin{align*}
\nu_+(\pA_{k}(z)) &= (n+k)-\nu_-(\pA_{k}(z)) = (n+k)-n_-
\leqslant n+ (n_++n_--n)-n_- = n_+ .
\end{align*}

\item For $z\in\bigcup_{j\notin J}\Gamma_j$, since $\mu_k(z)<0$ and $\nu_-(\pA_{k-1}(z))<n_- $,
we have
$ \nu_-(\pA_k(z)) = \nu_-(\pA_{k-1}(z))+1\leqslant n_-. $
Meanwhile, $ \nu_+(\pA_{k}(z)) = \nu_+(\pA_{k-1}(z))\leqslant n_+ $.
\end{itemize}
Combining the two cases, we showed that
$$ \max_{z\in\T\bs\sigma(\pA)}\nu_+(\pA_{k}(z)) \leqslant n_+ \qquad \mbox{and}\quad
\max_{z\in\T\bs\sigma(\pA)}\nu_-(\pA_{k}(z)) \leqslant n_-. $$
The inequalities of the other direction is obvious, since
$$ \max_{z\in\T\bs\sigma(\pA)}\nu_+(\pA_{k}(z)) \geqslant
\max_{z\in\T\bs\sigma(\pA)}\nu_+(\pA_{k-1}(z)) = n_+,
\qquad
\max_{z\in\T\bs\sigma(\pA)}\nu_-(\pA_{k}(z)) \geqslant
\max_{z\in\T\bs\sigma(\pA)}\nu_-(\pA_{k-1}(z)) = n_-. $$
So,
$ \max_{z\in\T\bs\sigma(\pA)}\nu_+(\pA_{k}(z)) = n_+$ and $\max_{z\in\T\bs\sigma(\pA)}\nu_-(\pA_{k}(z)) = n_- $.
According to Lemma~\ref{lem:maxminT}, we have
\begin{equation}  \label{eq:Sig+-}
\max_{z\in\T}\nu_+(\pA_{k}(z)) =  n_+,\qquad
\max_{z\in\T}\nu_-(\pA_{k}(z)) =  n_-.
\end{equation}

Now we can take $k:=k+1$ and repeat the above procedure recursively to construct all the Laurent polynomials
$\mu_1(z),\ldots,\mu_{m_0 - n}(z)$.
Equalities in \eqref{eq:Sig+-} guarantees that the item (i) will always hold in the new iteration.
We can repeat our constructions until the item (ii) is violated.
Take $\widetilde{\pA}(z):=\pA_{m_0-n}(z)$ to be the last matrix constructed.
It is an $m_0 \times m_0$ Hermite matrix of Laurent polynomials still satisfying
$$ \max_{z\in\T\setminus\sigma(\widetilde{\pA})}\nu_+(\widetilde{\pA}(z)) = n_+, \qquad
\max_{z\in\T\setminus\sigma(\widetilde{\pA})}\nu_-(\widetilde{\pA}(z)) = n_-. $$
Since $n_+ + n_- = m_0$, both $ \nu_+(\widetilde{\pA}(z)) $ and $\nu_-(\widetilde{\pA}(z))$ must be constant for all
$ z\in\T\setminus\sigma(\widetilde{\pA})$.
Hence, $\sig(\widetilde{\pA}(z))$ is constant on $\T\setminus\sigma(\widetilde{\pA})$.
By Theorem \ref{thm:const-sig}, there exists an $m_0 \times m_0$ matrix $\widetilde{\pU}(z)$ of Laurent polynomials such that
\begin{equation*}
\widetilde{\pA}(z) = \widetilde{\pU}(z)\pD\widetilde{\pU}^\star(z)
\end{equation*}
holds with $\pD = \diag(\mathbf{I}_{n_+}, -\mathbf{I}_{n_-})$ being the $m_0\times m_0$ constant diagonal matrix.

From the structure of $\widetilde{\pA}(z)$ in \eqref{eq:extend},
we conclude that $\pA(z)$ can be reconstructed by deleting the last $(m_0-n)$ rows and last $ (m_0 - n) $ columns of $\widetilde{\pA}(z)$. So, define $\pU(z)$ to be the $ n \times m_0$ matrix of Laurent polynomials constructed by deleting the last $(m_0 - n)$ rows of $\widetilde{\pU}(z)$, we get the desired factorization
$ \pA(z)=\pU(z)\pD\pU^\star(z)$.
This proves the sufficiency part of Theorem~\ref{thm:nonconst-sig} for the case $ \det(\pA(z))\not\equiv 0 $.

Now we consider the degenerate case that $  \det(\pA(z))\equiv 0 $.
For a matrix $ \pA(z) $ of Laurent polynomials, if its invariant polynomials are
$ \pd_1(z), \cdots, \pd_n(z) $, then we call the number of $ \pd_j(z) $ that are not identically zero
the \emph{general rank} of $ \pA(z) $.

Let us write $ \pA(z) $ into its Smith Normal Form:
\[ \pA(z) = \pE(z)
\diag(\pd_1(z),\ldots,\pd_r(z),\mathbf{0}_{(n-r)\times (n-r)})
\pF(z),
\]
where $ r $ is the general rank of $ \pA(z) $, $ \pd_1(z),\ldots,\pd_r(z) $ are the first $ r $ invariant polynomials of $ \pA(z) $ that are not identically zero
and $ \pE(z), \pF(z) $ are unimodular matrices of Laurent polynomials.
Define
\[ \mathring{\pA}(z):=\pE^{-1}(z)\pA(z)\pE^{-\star}(z) =
\diag(\pd_1(z),\ldots,\pd_r(z),\mathbf{0}_{(n-r)\times (n-r)})
\pF(z)\pE^{-\star}(z).
\]
Then $ \mathring{\pA}(z) $ is Hermite and its last $ (n-r) $ rows are zero. This implies that its last $ (n-r) $ columns must also be zero. Hence,
$ \mathring{\pA}(z) =
\diag(\widetilde{\pA}(z), \mathbf{0}) $,
where $\widetilde{\pA}(z) $ is an $ r \times r $ Hermite matrix of Laurent polynomials.
Since the invariant polynomials of $ \mathring{\pA}(z) $ are the same as that of $ \pA(z) $, which are
$ \pd_1(z), \ldots, \pd_r(z), 0, \ldots, 0 $, the invariant polynomials of $ \widetilde{\pA}(z) $ must be $ \pd_1(z), \ldots, \pd_r(z) $. So $ \det(\widetilde{\pA}(z)) $ is not identically zero.
Also, for all $ z\in\T $, since $ \pE^{-1}(z) $ is nonsingular, we get
\[ \nu_+(\widetilde{\pA}(z)) = \nu_+(\mathring{\pA}(z)) =  \nu_+(\pA(z)),  \qquad
\nu_-(\widetilde{\pA}(z)) = \nu_-(\mathring{\pA}(z)) =  \nu_-(\pA(z)).\]
Using the previously proved non-degenerate case, we know that for every
$$	m_1 \geqslant
\max_{z\in\T}\nu_+(\widetilde{\pA}(z))
= \max_{z\in\T}\nu_+(\pA(z)),\qquad
m_2 \geqslant \max_{z\in\T}\nu_-(\widetilde{\pA}(z))
= \max_{z\in\T}\nu_-(\pA(z)),$$
there exists an $r \times (m_1 + m_2)$ matrix of Laurent polynomials $\widetilde{\pU}(z)$ and a constant diagonal matrix $\pD=\diag(\mathbf{I}_{m_1}, -\mathbf{I}_{m_2})$,
such that
$$ \widetilde{\pA}(z)=\widetilde{\pU}(z)\pD\widetilde{\pU}^\star(z) .$$
Adding $ (n-r) $ more rows of zeros to $ \widetilde{\pU}(z) $ yields an $ n \times (m_1 + m_2) $ matrix
$ \pV(z) :=\begin{bmatrix}
\widetilde{\pU}(z) \\ \mathbf{0}_{(n-r)\times (m_1+m_2)}
\end{bmatrix} $. We can directly verify that
$\mathring{\pA}(z) = \pV(z)\pD\pV^\star(z) $. Define $ \pU(z) :=\pE(z)\pV(z) $, we know
$ \pA(z)=\pU(z)\pD\pU^\star(z)  $ holds. This proves the sufficiency part of Theorem~\ref{thm:nonconst-sig} for the case that $ \det(\pA(z))\equiv 0 $.
\end{proof}

\section{Quasi-tight Framelets with a General Dilation Factor}

Since the proof of Theorem~\ref{thm:qtf} on quasi-tight framelets is built on Theorems~\ref{thm:const-sig} and \ref{thm:nonconst-sig} for the generalized matrix spectral factorization of Hermite matrices of Laurent polynomials, we can easily generalize Theorem~\ref{thm:qtf} to quasi-tight framelets with an arbitrary dilation factor.

Let $ \dm $ be an integer such that $ \dm \geqslant 2 $. Suppose that $ \Theta, a, b_1, \ldots, b_s \in \lp{0} $ and $ \epsilon_1,\ldots, \epsilon_s \in \{-1, 1\} $. We say that $ \{a; b_1, \ldots, b_s\}_{\Theta, (\epsilon_1,\ldots,\epsilon_s)} $ is a \emph{quasi-tight $\dm$-framelet filter bank} if
\be \label{Mqtffb}
\left[\begin{matrix}
\pb_1(\omega^0 z) &\cdots &\pb_s(\omega^0 z)\\
\pb_1(\omega^1 z) &\cdots &\pb_s(\omega^1 z)\\
\vdots & \ddots & \vdots \\
\pb_1(\omega^{\dm-1} z) &\cdots &\pb_s(\omega^{\dm-1} z)
\end{matrix}\right]
\left[\begin{matrix} \eps_1 & &\\
&\ddots &\\ & &\eps_s\end{matrix}\right]
\left[\begin{matrix}
\pb_1(\omega^0 z) &\cdots &\pb_s(\omega^0 z)\\
\pb_1(\omega^1 z) &\cdots &\pb_s(\omega^1 z)\\
\vdots & \ddots & \vdots \\
\pb_1(\omega^{\dm-1} z) &\cdots &\pb_s(\omega^{\dm-1} z)
\end{matrix}\right]^\star
=\cM_{\pa,\pTh}(z),
\ee
where $ \omega := e^{-i 2\pi/\dm} $ and the $ \dm \times \dm $ matrix $\cM_{\pa,\pTh}$ is defined to be
\begin{equation} \label{McM}
\cM_{\pa,\pTh}(z):=
\left[\begin{matrix}
\pTh(\omega^0 z) & & \\
 & \ddots & \\
 &        & \pTh(\omega^{\dm-1} z)
\end{matrix}\right] - \pTh(z^{\dm})
\left[ \begin{matrix}
\pa(\omega^0 z) \\ \vdots \\ \pa(\omega^{\dm-1}z)
\end{matrix}
\right]
\left[ \begin{matrix}
\pa(\omega^0 z) \\ \vdots \\ \pa(\omega^{\dm-1}z)
\end{matrix}
\right]^\star.
\end{equation}
Assume that $\pa(1)=1$ and $\phi\in \Lp{2}$ with $\phi$ being defined by
\be \label{Mphi}
\wh{\phi}(\xi):=\prod_{j=1}^\infty \pa(e^{-i \dm^{-j}\xi}),\qquad \xi\in \R.
\ee
Write $\pTh(z)=\tilde{\pth}(z)\pth^\star(z)$ for some $\tilde{\theta}, \theta\in \lp{0}$. Define $\eta, \tilde{\eta}, \psi^1,\ldots,\psi^s$ by
\be \label{Meta:psi}
\eta:=\sum_{k\in \Z} \theta(k) \phi(\cdot-k), \;
\tilde{\eta}:=\sum_{k\in \Z} \tilde{\theta}(k) \phi(\cdot-k),\quad
\psi^\ell:=\dm\sum_{k\in \Z} b_\ell(k)\phi(\dm\cdot-k),\qquad \ell=1,\ldots,s.
\ee
If in addition $\pTh(z)\ge 0$ for all $z\in \T$, then by Fej\'er-Riesz lemma we can always choose $\tilde{\theta}=\theta$ so that $\tilde{\eta}=\eta$.
If $\pTh(1)=1$ and $\pb_1(1)=\cdots=\pb_s(1)=0$,
then
$\{\eta,\tilde{\eta};\psi^1,\ldots,\psi^s\}_{(\epsilon_1, \ldots, \epsilon_s)}$ is \emph{a quasi-tight $\dm$-framelet} in $\Lp{2}$, that is, for every $ f\in \Lp{2} $,
\be \label{Mqtf}
f=\sum_{k\in \Z} \la f, \tilde{\eta}(\cdot-k)\ra \eta(\cdot-k)+
\sum_{j=0}^\infty \sum_{\ell=1}^s \sum_{k\in \Z} \eps_\ell \la f, \psi^\ell_{\dm^j;k}\ra \psi^\ell_{\dm^j;k}
\ee
with the series converging unconditionally in $\Lp{2}$ and the underlying system being a Bessel sequence in $\Lp{2}$.
Moreover, if $\{a;b_1,\ldots,b_s\}_{\Theta,(\epsilon_1,\ldots, \epsilon_s)}$ is a quasi-tight $\dm$-framelet filter bank, then
\be \label{qtf:vm:srM}
\min(\vmo(b_1),\ldots,\vmo(b_s))\le \min(\sr(a, \dm), \tfrac{1}{2}\vmo(\pTh(z)-\pTh(z^\dm)\pa(z)\pa^\star(z))),
\ee
where $ \sr(a, \dm) $ is the largest integer $ n $ such that $ (1+z+\cdots+z^{\dm-1})^n \mid \pa(z) $.

\begin{theorem}\label{thm:Mqtf}
Let $a,\Theta\in \lp{0}\bs\{0\}$ be two finitely supported
not-identically-zero filters such that $\pTh^\star=\pTh$.
Let $n_b$ be any positive integer satisfying
\be \label{Mnb}
1\le n_b\le \min(\sr(a, \dm), \tfrac{1}{2} \vmo(\pTh(z)-\pTh(z^\dm) \pa(z)\pa^\star(z))).
\ee
Let $\cM_{\pa,\pTh}(z)$ be defined in \eqref{McM} and
the quantities $s_{a,\Theta}^+, s_{a,\Theta}^-, s_{a,\Theta}$ be defined in \eqref{saTheta}.
Then there exist $b_1,\ldots,b_s\in \lp{0}$ with $s=s_{a,\Theta}$
and $ \eps_1=\cdots= \eps_{s_{a,\Theta}^+} = 1 $,  $\eps_{s_{a,\Theta}^+ +1}=\cdots=\eps_s=-1$ such that $\{a;b_1,\ldots,b_s\}_{\Theta, (\epsilon_1,\ldots, \epsilon_s)}$ is a quasi-tight $\dm$-framelet filter bank with
$ \min\{\vmo(b_1),\ldots, \vmo(b_s)\}\geqslant n_b $.
Also, for $1\le s<s_{a,\Theta}$, there does not exist a quasi-tight $\dm$-framelet filter bank
$\{a;b_1,\ldots,b_s\}_{\Theta, (\epsilon_1, \ldots, \epsilon_s)}$ with $b_1,\ldots,b_s\in \lp{0}$ and $\eps_1,\ldots,\eps_s\in\{-1,1\}$.
Moreover, if $\pa(1)=\pTh(1)=1$ and $\phi\in \Lp{2}$ with $\phi$ being defined in \eqref{Mphi}, then
$\{\eta,\tilde{\eta};\psi^1,\ldots,\psi^s\}_{(\epsilon_1, \ldots, \epsilon_s)}$ is a quasi-tight $\dm$-framelet in $\Lp{2}$, where $\eta,\tilde{\eta},\psi^1,\ldots,\psi^s\in \Lp{2}$ are defined in \eqref{Meta:psi} with $\tilde{\pth}(z)\pth^\star(z)=\pTh(z)$.
\end{theorem}

\begin{proof}
Since all high-pass filters must have at least $n_b$ vanishing moments, we can write
\begin{equation} \label{eq:MbCoset}
\pb_\ell(z) = (1-z^{-1})^{n_b}\mathring{\pb}_\ell(z), \qquad \ell = 1,\ldots, s
\end{equation}
for some Laurent polynomials $\mathring{\pb}_0,\ldots,\mathring{\pb}_s$.
Then $\{a;b_1,\ldots,b_s\}_{\Theta,(\eps_1,\ldots,\eps_s)}$ is a quasi-tight $\dm$-framelet filter bank
satisfying
\eqref{Mqtffb} and \eqref{Mnb}
if and only if
\begin{equation} \label{eq:Mqtfbnb}
\left[\begin{matrix}
\mathring{\pb_1}(\omega^0 z) &\cdots &\mathring{\pb_s}(\omega^0 z)\\
\mathring{\pb_1}(\omega^1 z) &\cdots &\mathring{\pb_s}(\omega^1 z)\\
\vdots & \ddots & \vdots \\
\mathring{\pb_1}(\omega^{\dm-1} z) &\cdots &\mathring{\pb_s}(\omega^{\dm-1} z)
\end{matrix}\right]
\left[\begin{matrix} \eps_1 & &\\
&\ddots &\\ & &\eps_s\end{matrix}\right]
\left[\begin{matrix}
\mathring{\pb_1}(\omega^0 z) &\cdots &\mathring{\pb_s}(\omega^0 z)\\
\mathring{\pb_1}(\omega^1 z) &\cdots &\mathring{\pb_s}(\omega^1 z)\\
\vdots & \ddots & \vdots \\
\mathring{\pb_1}(\omega^{\dm-1} z) &\cdots &\mathring{\pb_s}(\omega^{\dm-1} z)
\end{matrix}\right]^\star
= \cM_{\pa, \pTh|n_b}(z),
\end{equation}
where $ \cM_{\pa, \pTh|n_b}(z) $ is an $ \dm\times \dm $ matrix defined by
\begin{align*} 
[\cM_{\pa, \pTh|n_b}]_{j,j}(z):=&
\frac{\pTh(\omega^{j-1} z) -
\pTh(z^\dm)\pa(\omega^{j-1} z)\pa^\star(\omega^{j-1} z)}
{(1-\omega^{j-1} z)^{n_b}(1-(\omega^{j-1}z)^{-1})^{n_b}}, \qquad j = 1,\ldots, \dm, \\
[\cM_{\pa, \pTh|n_b}]_{j,k}(z):=&
\frac{-\pTh(z^\dm)\pa(\omega^{j-1} z)\pa^\star(\omega^{k-1} z)}
{(1 - \omega^{k-1}z)^{n_b}(1-(\omega^{j-1}z)^{-1})^{n_b}},
\qquad j,k = 1,\ldots, \dm, \quad j\neq k.
\end{align*}
Note that according to the upper bound of $ n_b $ in \eqref{Mnb},
$ \cM_{\pa, \pTh|n_b}(z) $ is a well-defined $ \dm\times\dm $ matrix of Laurent polynomials.
For $ \gamma \in \{0, \dots, \dm-1\} $, define the \emph{$ \gamma$-coset sequence} of a filter $ u\in \lp{0} $ as $ u^{[\gamma]} := \{u(\gamma + \dm k)\}_{k\in \Z} $.
Then we have
\[
\begin{bmatrix}
\pu(\omega^0 z) \\
\vdots \\
\pu(\omega^{\dm-1} z)
\end{bmatrix} =
\begin{bmatrix}
1 & 1 & \ldots & 1 \\
1 & \omega & \ldots & \omega^{\dm-1} \\
\vdots & \vdots & \ddots & \vdots \\
1 & \omega^{\dm-1} & \ldots & \omega^{(\dm-1)(\dm-1)}
\end{bmatrix}
\begin{bmatrix}
1 &   & & \\
  & z & & \\
  &   & \ddots & \\
  &   & & z^{\dm-1}
\end{bmatrix}
\begin{bmatrix}
\pu^{[0]}(z^{\dm}) \\
\vdots \\
\pu^{[\dm-1]}(z^{\dm})
\end{bmatrix}
= F(z)
\begin{bmatrix}
\pu^{[0]}(z^{\dm}) \\
\vdots \\
\pu^{[\dm-1]}(z^{\dm})
\end{bmatrix},
 \]
where $ F(z) $ is defined as $ F_{j,k}(z) = \omega^{(j-1)(k-1)}z^{k-1} $, $ j, k = 1,\dots \dm $.
Notice that $ F(z) $ satisfies $ F(z)F^\star(z) = \dm \mathbf{I_\dm} $, hence, $ F(z) $ is invertible: $ F^{-1}(z) = \frac{1}{\dm}F^\star(z)
= \frac{1}{\dm}
[\omega^{-(j-1)(k-1)}z^{1-j}]_{1\leqslant j \leqslant \dm, 1\leqslant k \leqslant \dm} $.
Define an $ \dm\times\dm $ matrix of Laurent polynomial $ N(z):= F^{-1}(z)\cM_{\pa, \pTh|n_b}(z)F^{-\star}(z) $.
That is, for $ j, k = 1,\dots, \dm $,
\begin{align} \label{eq:Njk}
N_{j,k}(z)
=& \sum_{q = 1}^{\dm}\sum_{p = 1}^{\dm}
[F^{-1}]_{j,p}(z) [\cM_{\pa, \pTh|n_b}]_{p,q}(z)
[F^{-\star}]_{q,k}(z) \notag \\
=& \frac{1}{\dm^2}z^{k-j} \left(
\sum_{p=0}^{\dm-1}
\sum_{\substack{q=0 \\ q\neq p}}^{\dm-1}
\omega^{(k-1)q-(j-1)p}
\frac{-\pTh(z^\dm)\pa(\omega^{p} z)\pa^\star(\omega^{q} z)}
{(1 - \omega^{q}z)^{n_b}(1-(\omega^{p}z)^{-1})^{n_b}}
\right. \notag \\
& \qquad \qquad\qquad \qquad \left.
+ \sum_{\ell=0}^{\dm-1} \omega^{(k-j)\ell}
\frac{\pTh(\omega^{\ell} z) -
\pTh(z^\dm)\pa(\omega^{\ell} z)\pa^\star(\omega^{\ell} z)}
{(1-\omega^{\ell} z)^{n_b}(1-(\omega^{\ell}z)^{-1})^{n_b}}
\right).
\end{align}
It is easy to verify that $ N(\omega^{r}z) = N(z) $ for all $ r = 0, \ldots, \dm-1 $.
Hence, $ N(z) $ only depends on $ z^\dm $, and we can write $ N(z) = \cN_{\pa, \pTh|n_b}(z^\dm) $, where $ \cN_{\pa, \pTh|n_b}(z) $ is an $ \dm \times \dm$ Hermite matrix of Laurent polynomials.

Multiplying $ F^{-1}(z) $ and $ F^{-\star}(z) $ on the left and right side of \eqref{eq:Mqtfbnb} respectively, we see that \eqref{eq:Mqtfbnb} is equivalent to
\begin{equation} \label{eq:MPRpolyphase}
\left[\begin{matrix}
\mathring{\pb_1}^{[0]}(z) &\cdots &\mathring{\pb_s}^{[0]}(z)\\
\vdots & \ddots & \vdots \\
\mathring{\pb_1}^{[\dm-1]}(z) &\cdots &\mathring{\pb_s}^{[\dm-1]}(z)
\end{matrix}\right]
\left[\begin{matrix} \eps_1 & &\\
&\ddots &\\ & &\eps_s\end{matrix}\right]
\left[\begin{matrix}
\mathring{\pb_1}^{[0]}(z) &\cdots &\mathring{\pb_s}^{[0]}(z)\\
\vdots & \ddots & \vdots \\
\mathring{\pb_1}^{[\dm-1]}(z) &\cdots &\mathring{\pb_s}^{[\dm-1]}(z)
\end{matrix}\right]^\star
=\cN_{\pa, \pTh|n_b}(z),
\end{equation}
Hence, the existence of a quasi-tight $\dm$-framelet filter bank
$ \{a; b_1, \ldots, b_s\}_{\Theta, (\epsilon_1, \ldots, \epsilon_s)} $ with $ n_b $ order of vanishing moments necessarily implies a generalized spectral factorization \eqref{eq:MPRpolyphase} for the Hermite matrix $ \cN_{\pa, \pTh|n_b}(z) $ of Laurent polynomials.

According to Theorem~\ref{thm:nonconst-sig},  the existence of the generalized spectral factorization \eqref{eq:MPRpolyphase} implies that the number $ s_+ $ of times that ``$ +1 $" appears in $ \{\epsilon_1,\ldots,\epsilon_s \} $ and the number $ s_- $ of times that ``$ -1 $" appears in $ \{\epsilon_1,\ldots,\epsilon_s \} $ satisfy
\begin{equation} \label{eq:Ms+s-}
s_+ \geqslant \max_{z\in \T}\nu_+(\cN_{\pa, \pTh|n_b}(z)) , \qquad \mbox{and} \quad
s_- \geqslant \max_{z\in \T}\nu_-(\cN_{\pa, \pTh|n_b}(z)).
\end{equation}
By \eqref{eq:MbCoset}, \eqref{eq:Mqtfbnb} and \eqref{eq:Njk}, we deduce that
\begin{align*}
\cM_{\pa, \pTh}(z)=
\pP(z)  \cN_{\pa, \pTh|n_b}(z^{\dm}) \pP^\star(z),
\end{align*}
where $ \pP(z) := \diag((1-(\omega^0 z)^{-1})^{n_b},
\ldots,(1-(\omega^{\dm - 1} z)^{-1})^{n_b})
$.
Hence,
$ \sigma(\pP) \subseteq \{1, \omega, \ldots, \omega^{\dm}\} $, which is a finite set.
Similar to the proof of Theorem~\ref{thm:qtf}, we conclude that
\begin{equation*}
s_{a, \Theta}^+
= \max_{z\in \T} \nu_+(\cN_{\pa, \pTh|n_b}(z)), \qquad
\mbox{and}\quad
s_{a, \Theta}^- = \max_{z\in \T} \nu_-(\cN_{\pa, \pTh|n_b}(z)).
\end{equation*}
Therefore, from \eqref{eq:Ms+s-} we see the generalized spectral factorization \eqref{eq:MPRpolyphase} implies
\begin{equation*}
s_+ \geqslant s_{a, \Theta}^+ , \qquad
s_- \geqslant s_{a, \Theta}^-, \qquad
\mbox{and} \quad
s = s_+ + s_-
\geqslant s_{a, \Theta}^+ + s_{a, \Theta}^- = s_{a, \Theta} .
\end{equation*}
Hence, for $1\le s<s_{a,\Theta}$, there does not exist a quasi-tight framelet filter bank
$\{a;b_1,\ldots,b_s\}_{\Theta, (\epsilon_1, \ldots, \epsilon_s)}$ with $b_1,\ldots,b_s\in \lp{0}$ and $\eps_1,\ldots,\eps_s\in\{-1,1\}$.

On the other hand, given filters $ a, \Theta \in \lp{0}\bs\{0\} $, $ \Theta^\star = \Theta $, and positive integer $ n_b $ satisfying \eqref{Mnb}, we can calculate the matrix $ \cN_{\pa, \pTh|n_b}(z) $ of Laurent polynomials from $ \cN_{\pa, \pTh|n_b}(z^\dm):= N(z) $, where $ N(z) $ is defined as \eqref{eq:Njk}.
By $ \Theta^\star(z) = \Theta(z) $, we see from \eqref{eq:Njk} that
$ N(z) $ and $ \cN_{\pa, \pTh|n_b}(z) $ are both Hermite matrices of Laurent polynomials.
Take $ s := s_{a, \Theta} = s_{a, \Theta}^+ + s_{a, \Theta}^- $.
According to Theorem~\ref{thm:nonconst-sig},
we can choose
$ \eps_1=\ldots= \eps_{s_{a,\Theta}^+} = 1 $, and $\eps_{s_{a,\Theta}^+ +1}=\ldots=\eps_{s}=-1$,
and find a generalized spectral factorization of $ \cN_{\pa, \pTh|n_b}(z) $ as
$ \cN_{\pa, \pTh|n_b}(z) = \pU(z)
\diag(\eps_1,\ldots,\eps_s)
\pU^\star(z)$,
where $ \pU(z) $ is an $ \dm\times s $ matrix of Laurent polynomials.
Then we can define Laurent polynomials $ \mathring{\pb}_1(z), \ldots, \mathring{\pb}_s(z) $ as
\[
\left[\begin{matrix}
\mathring{\pb_1}^{[0]}(z) &\cdots &\mathring{\pb_s}^{[0]}(z)\\
\vdots & \ddots & \vdots \\
\mathring{\pb_1}^{[\dm-1]}(z) &\cdots &\mathring{\pb_s}^{[\dm-1]}(z)
\end{matrix}\right] := \pU(z).
\]
Thus, \eqref{eq:MPRpolyphase} holds.
Multiplying
$ F(z) $ and
$ F^\star(z) $ on the left and right side of $ \cN_{\pa, \pTh|n_b}(z^\dm) $ respectively, we see that \eqref{eq:MPRpolyphase} is equivalent to \eqref{eq:Mqtfbnb}.
Define Laurent polynomials $ \pb_1(z)\ldots, \pb_s(z) $ in \eqref{eq:MbCoset}, we can conclude from \eqref{eq:Mqtfbnb} that $ \{a; b_1, \ldots, b_s\}_{\Theta, (\epsilon_1, \ldots, \epsilon_s)} $ is a quasi-tight framelet filter bank with
$ \min\{\vmo(b_1),\ldots, \vmo(b_s)\}\geqslant n_b $.
This proves the existence of the quasi-tight framelet filter bank with minimum number of high-pass filters and high vanishing moments.
\end{proof}

\begin{example} \label{ex:3Band}
{ \rm
Let $ \dm = 3 $ be a dilation factor.
Consider $ \pTh(z) = 1 $ and the low-pass filter
\[ \pa(z) = -\tfrac{1}{27}z^{-3}(1+z+z^2)^2 (2z^2 - 7z + 2). \]
By the definition of $ \cM_{\pa, 1}(z) $ in \eqref{McM}, the three eigenvalues of $ \cM_{\pa, 1}(z) $ are $ 1, 1 $ and $ \det(\cM_{\pa, 1}(z)) $.
We see from Figure~\ref{fig:3Band} that $ \det(\cM_{\pa, 1}(z)) \leqslant 0 $ on $ \T $.
Hence $ s_{a, \Theta}^+ = 2 $ and $ s_{a, \Theta}^- = 1 $.
Note that $ \sr(\pa, 3) = 2 $ and $ \vmo(1-\pa\pa^\star) = 4 $. Therefore, the maximum order of vanishing moments is two. Taking $ n_b = 2 $, we obtain a quasi-tight $3$-framelet filter bank $ \{a; b_1, b_2, b_3\}_{\Theta, (1, 1, -1)} $ as follows:
\begin{align*}
\pb_1(z) = \tfrac{\sqrt{6}}{6}(z-1)^2(z+1),\qquad
\pb_2(z) = \tfrac{\sqrt{6}}{18}(z-1)^3, \qquad
\pb_3(z) = \tfrac{1}{27}z^{-3}(z-1)^4(2z^2 + 5z + 2),
\end{align*}
with $ \vmo(b_1) = 2 $, $ \vmo(b_2) = 3 $ and $ \vmo(b_3) = 4 $.
Since $ \sm(a, 3)\approx 0.6599 $ (see \cite[(7.2.2)]{hanbook} for its definition), the refinable function $ \phi $ defined in \eqref{Mphi} belongs to $ \Lp{2} $. Therefore, $ \{\phi,\phi; \psi^1, \psi^2, \psi^3\}_{(1, 1, -1)} $ is a quasi-tight $3$-framelet in $ \Lp{2} $ and $ \{\psi^1, \psi^2, \psi^3\}_{(1, 1, -1)} $ is a homogeneous quasi-tight $3$-framelet in $ \Lp{2} $, where $ \psi^1, \psi^2 $ and $ \psi^3 $ are defined in \eqref{Meta:psi} and have at least two vanishing moments.
} \end{example}

\begin{figure}[h!]
	\centering	
	\begin{subfigure}[]{0.18\textwidth}			 \includegraphics[width=\textwidth, height=0.8\textwidth]{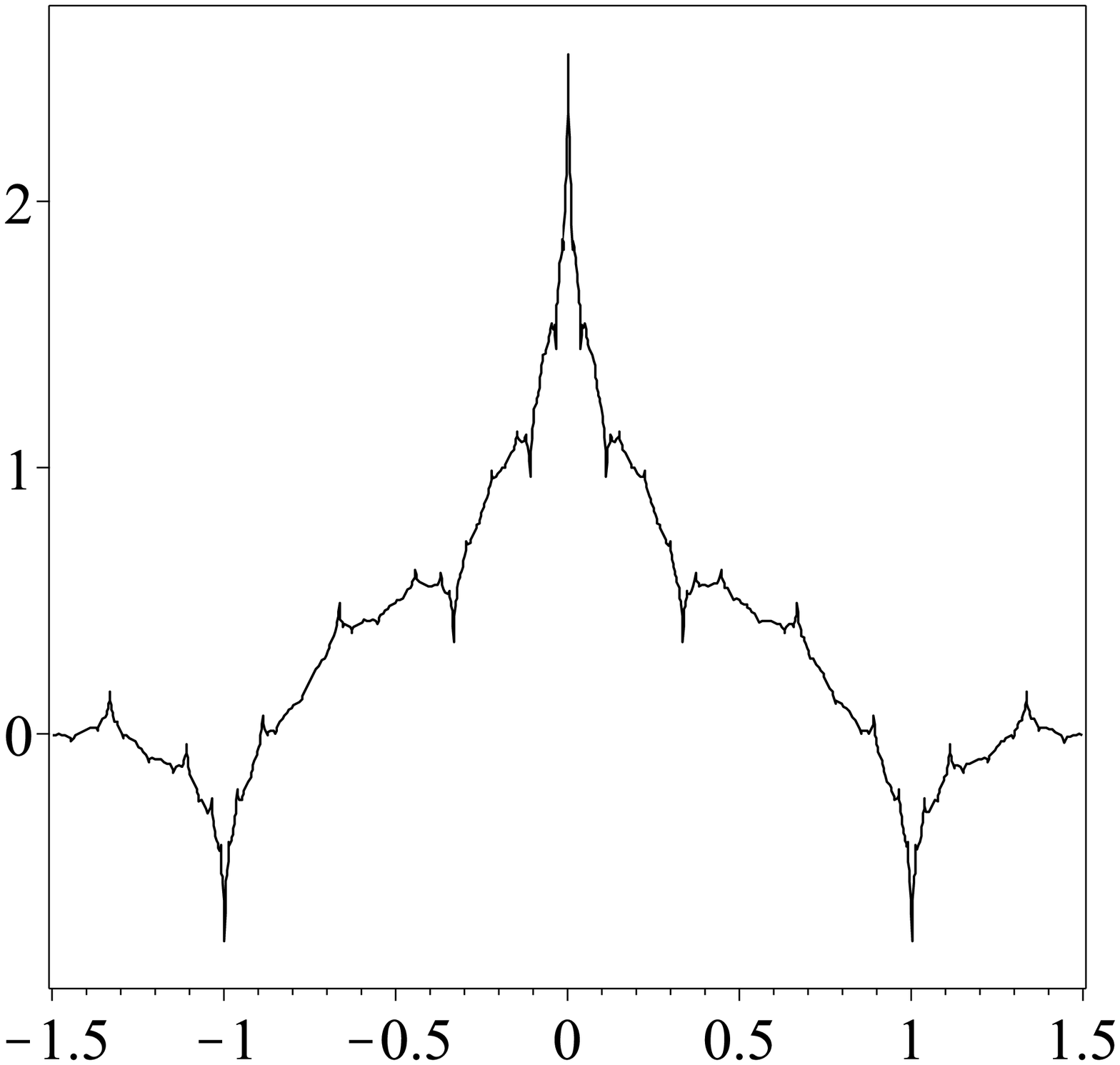}
		\caption{$\phi$}
	\end{subfigure}
	\begin{subfigure}[]{0.18\textwidth}		 \includegraphics[width=\textwidth, height=0.8\textwidth]{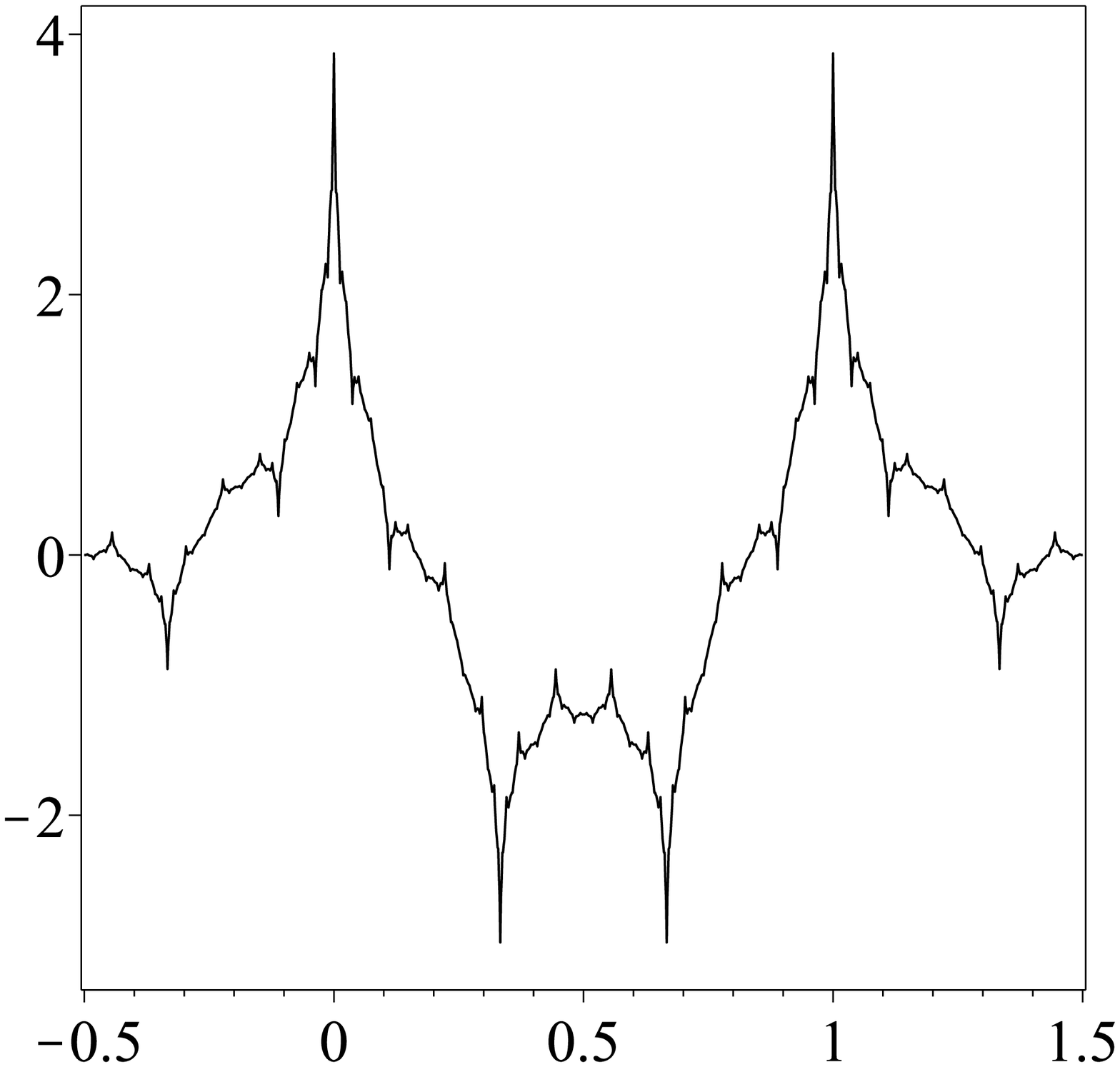}
		\caption{$\psi^1$}
	\end{subfigure}
	\begin{subfigure}[]{0.18\textwidth}		 \includegraphics[width=\textwidth, height=0.8\textwidth]{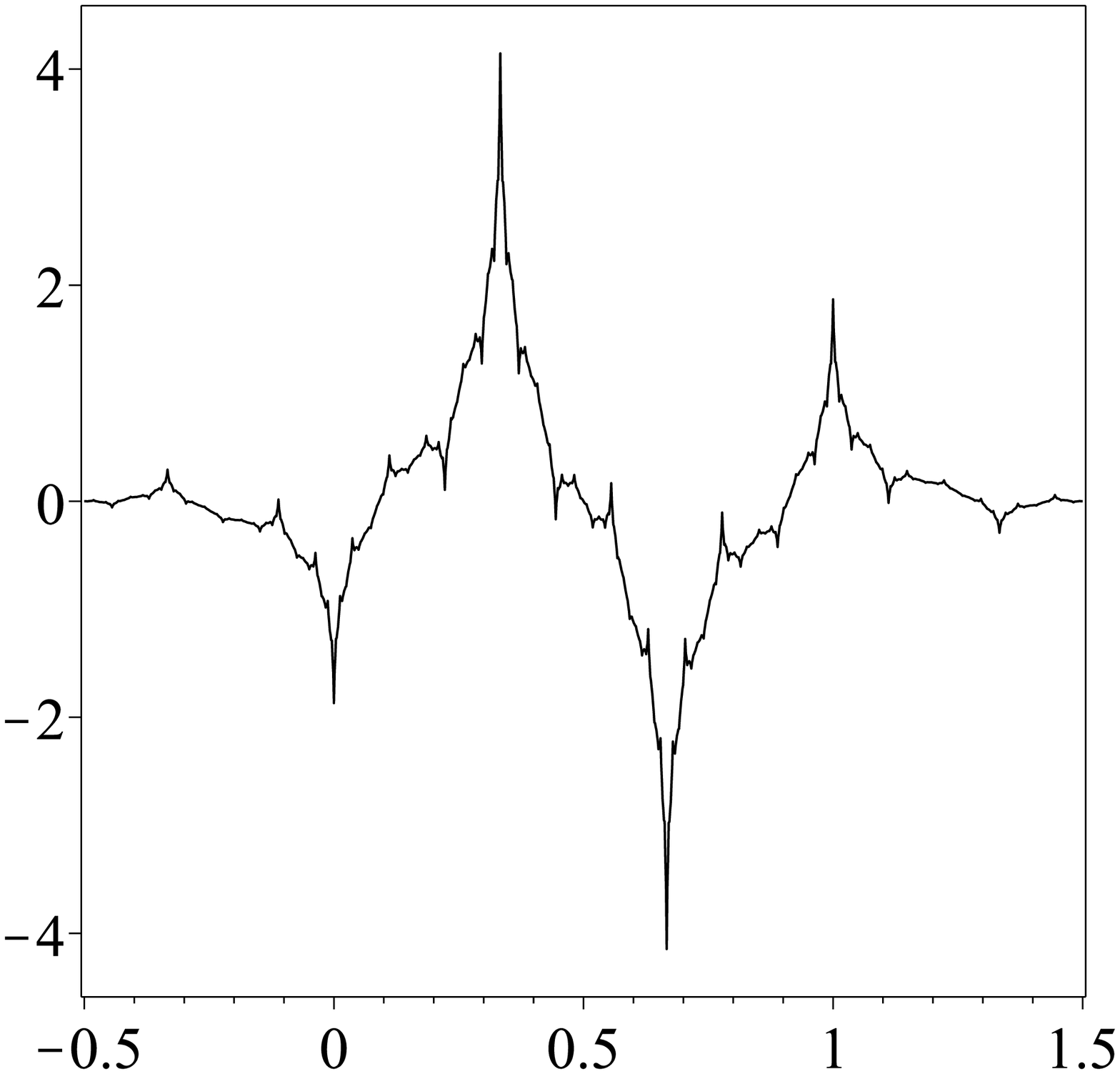}
		\caption{$\psi^2$}
	\end{subfigure}		 \begin{subfigure}[]{0.18\textwidth}		 
	\includegraphics[width=\textwidth, height=0.8\textwidth]{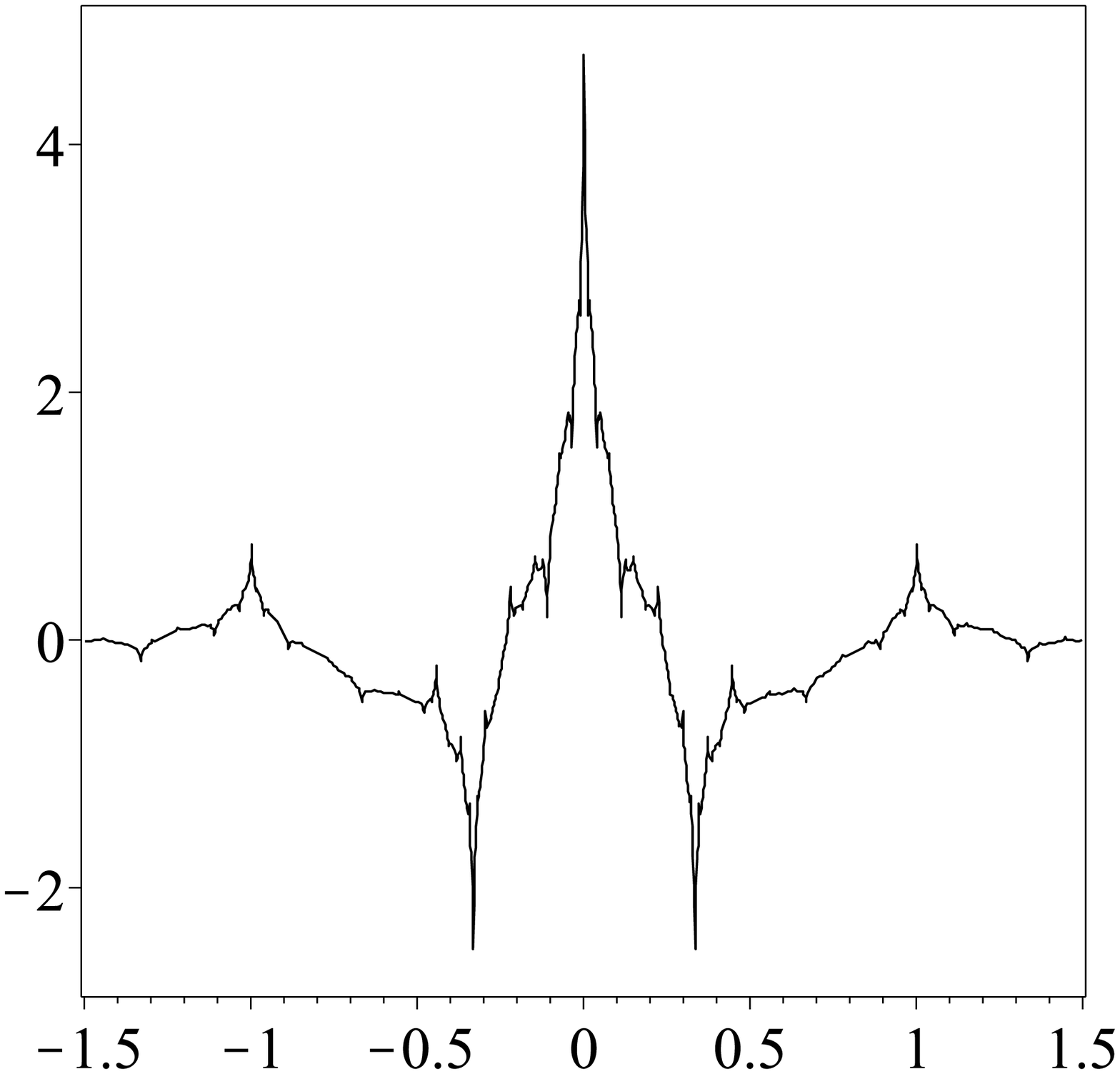}
		\caption{$\psi^3 $ }
	\end{subfigure}
	\begin{subfigure}[]{0.18\textwidth}
		\includegraphics[width=\textwidth, height=0.8\textwidth]{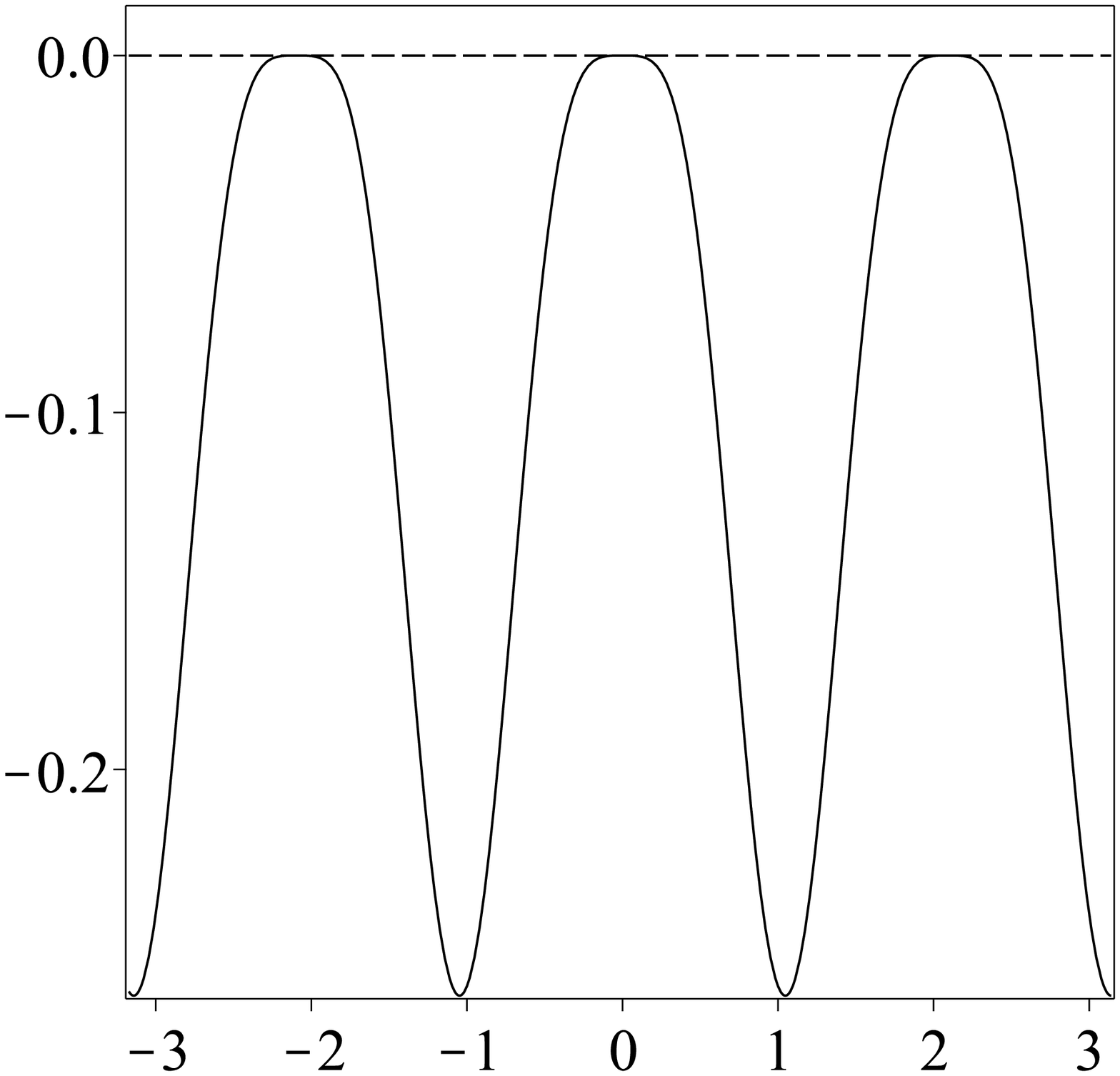}
		\caption{$ \det(\cM_{\pa, 1})$}
	\end{subfigure}
	\caption{
		The quasi-tight $3$-framelet $\{\phi,\phi; \psi^1,\psi^2,\psi^3\}_{(1,1,-1)}$ in $\Lp{2}$ and the homogeneous quasi-tight $3$-framelet $\{\psi^1,\psi^2,\psi^3\}_{(1,1,-1)}$ in $\Lp{2}$ obtained in Example~\ref{ex:3Band}.
		(A) is the refinable function $\phi\in \Lp{2}$. (B), (C) and (D) are the framelet functions $\psi^1$, $ \psi^2 $ and $ \psi^3$.
		(E) is $\det(\cM_{\pa, 1}(e^{-i\xi}))$ for $ \xi\in  [-\pi, \pi] $.
	} \label{fig:3Band}
\end{figure}

\appendix
\setcounter{section}{1}

\section*{Appendix. Proofs of Auxiliary Results for Proving Theorem~\ref{thm:const-sig}}

In this Appendix, we provide the proofs for several auxiliary results used in the proof of Theorem~\ref{thm:const-sig} in Section~4.
To prove Lemma~\ref{lemma:EmptySpect2}, we need the following result.

\begin{lemma} \label{lemma:EmptySpect1}
Let $ \pQ(z)\not \equiv 0 $ be a $ k\times k $ Hermite matrix of Laurent polynomials, which is diagonally dominant and the lengths of the diagonal entries are nondecreasing as in \eqref{eq:lemmaDiagIncreasing}.
Let $ \pb(z) $ be a column vector of Laurent polynomials with size $ k $ satisfying
\begin{equation} \label{eq:lemmaldeg}
 \ldeg(\pb_\ell(z)) \geqslant \ldeg(\pQ_{\ell,\ell}(z)), \qquad
\ell = 1,\ldots,k.
\end{equation}
Then there exists a column vector $ X(z) $ of Laurent polynomials with size $ k $ such that $ Y(z):= \pb(z) - \pQ(z)X(z) $ satisfies
\begin{equation} \label{eq:lemmaReq}
\fs(Y_\ell(z)) \subsetneq \fs(\pQ_{\ell,\ell}(z)), \qquad
\deg(Y_\ell(z)) < \deg(\pQ_{\ell,\ell}(z)),
\qquad \forall ~\ell = 1,\ldots,k.
\end{equation}
\end{lemma}

\begin{proof}
If $ \pb(z) $ already satisfies $ \deg(\pb_\ell(z)) < \deg(\pQ_{\ell,\ell}(z))$ for all $\ell = 1,\ldots,k $, then we can simply take $ X(z) = 0 $, and the result is true. So we just need to consider the case that there exists some $ s \in \{ 1,\ldots,k\} $, such that
\begin{equation}\label{eq:lemmadeg}
\deg(\pb_s(z)) \geqslant \deg(\pQ_{s,s}(z)) .
\end{equation}

Since $ \pQ(z) $ is Hermite, the filter supports of its diagonal elements must be symmetric intervals:
\[ [-n_\ell, n_\ell]:=\fs(\pQ_{\ell,\ell}), \qquad \ell = 1,\ldots,k. \]
From \eqref{eq:lemmaDiagIncreasing}, we know that $ n_1 \leqslant \cdots \leqslant n_k $.
Define
\[ \pD(z):=\diag\Big(\big[(z+1)(z^{-1}+1)\big]^{n_k-n_1},
\big[(z+1)(z^{-1}+1)\big]^{n_k-n_2},\ldots,1\Big)
\]
$ \widetilde{\pQ}(z):=\pD(z)\pQ(z) $, and $ \widetilde{\pb}(z):=\pD(z)\pb(z) $. We see from \eqref{eq:lemmaldeg} that for all $ \ell = 1,\ldots,k $:
\begin{equation} \label{eq:lemmafsupp}
\fs(\widetilde{\pQ}_{\ell,\ell}(z)) = [-n_k, n_k], \qquad
\ldeg(\widetilde{\pb}_\ell(z))\geqslant
\ldeg(\widetilde{\pQ}_{\ell,\ell}(z))= -n_k.
\end{equation}
Since $ \pQ(z) $ is diagonally dominant, from \eqref{def:DiagDominant1} we see that $ \fs(\widetilde{\pQ}_{\ell,i}(z))\subsetneq \fs(\widetilde{\pQ}_{\ell,\ell}(z)) = [-n_k, n_k] $ for all $ \ell = 1,\ldots,k,~ i\neq \ell $.
Thus we can write $ \widetilde{\pQ}(z) $ as
\begin{equation} \label{eq:lemmaCoeffs}
 \widetilde{\pQ}(z) = \sum_{\ell=-n_k}^{n_k}\widetilde{Q}_\ell z^\ell.
\end{equation}
Also, from \eqref{def:DiagDominant2} we know that
$ \deg(\widetilde{\pQ}_{\ell, i}(z))<\deg(\widetilde{\pQ}_{\ell,\ell}(z))=n_k $ for all $ \ell = 1,\ldots,k, ~ i>\ell $. So the coefficient matrix $ \widetilde{Q}_{n_k} $ in \eqref{eq:lemmaCoeffs} is lower triangular, which by \eqref{eq:lemmafsupp} also has nonzero diagonal elements. Therefore $ \widetilde{Q}_{n_k} $ is nonsingular.

From \eqref{eq:lemmadeg}, we also know that there exists some $ s\in \{1,\ldots,k\} $, such that $ \deg(\widetilde{\pb}_s(z))\geqslant \deg(\widetilde{\pQ}_{s,s}(z))=n_k $. So by \eqref{eq:lemmafsupp}, we can write $ \widetilde{\pb}(z) $ as
$\widetilde{\pb}(z) = \sum_{\ell = -n_k}^{M}\widetilde{b}_\ell z^\ell$
with $ M\geqslant n_k $.
Let us parameterize the unknown
$ X(z)= \sum_{\ell=0}^{M-n_k}X_\ell z^\ell $, and take $ \widetilde{Y}(z):=\widetilde{\pb}(z) - \widetilde{\pQ}(z)X(z) $.
By this definition,
$ \fs(\widetilde{Y}(z))\subset [-n_k, M] $.
Write $ \widetilde{Y}(z) = \sum_{\ell=-n_k}^{M}\widetilde{Y}_\ell z^\ell $, we want to solve for $ X(z) $ such that the coefficients $ \widetilde{Y}_\ell = 0 $ for all $ \ell = n_k, n_k+1, \ldots,M $.
Notice that we have $ M-n_k +1 $ matrix equations to solve for $ M-n_k+1 $ unknowns.
The equations can be formulated as the following Toeplitz form
\[ \begin{bmatrix}
\widetilde{Q}_{n_k} & \widetilde{Q}_{n_k-1} & \cdots & \widetilde{Q}_{2n_k-M}\\
& \ddots & \ddots & \vdots \\
& & \widetilde{Q}_{n_k} & \widetilde{Q}_{n_k-1} \\
& & & \widetilde{Q}_{n_k}
\end{bmatrix}
\begin{bmatrix}
X_0 \\ \vdots \\ X_{M - n_k-1} \\ X_{M - n_k}
\end{bmatrix}
= \begin{bmatrix}
\widetilde{b}_{n_k} \\ \vdots \\ \widetilde{b}_{M-1} \\ \widetilde{b}_{M}
\end{bmatrix}, \]
where we use $ \widetilde{Q}_j = 0 $ if $ j<-n_k $.
Since $ \widetilde{Q}_{n_k} $ is nonsingular, we can solve the above system from the last equation and use backward substitution to find all $ X_0, \ldots,X_{M-n_k} $.

Now, we found a vector $ X(z) $ of Laurent polynomials such that
$$
\widetilde{Y}(z) = \widetilde{\pb}(z) - \widetilde{\pQ}(z)X(z) =  \pD(z)\Big(\pb(z) - \pQ(z)X(z)\Big)
$$
satisfies $ \fs(\widetilde{Y}_\ell(z))\subset [-n_k, ~n_k-1] $, for all $ \ell = 1,\ldots,k $. Take $ Y(z):= \pb(z) - \pQ(z)X(z) = \pD^{-1}(z)\widetilde{Y}(z)$, we will prove that it satisfies \eqref{eq:lemmaReq}.
Notice that $ \pD(z) $ is a diagonal matrix with diagonals $ \pD_\ell(z) = \big[(z+1)(z^{-1}+1)\big]^{n_k - n_\ell} $, for all $ \ell = 1,\ldots,k $. So the filter support of $ Y_\ell(z) = \widetilde{Y}_\ell(z)/\pD_\ell(z) $ satisfies $ \fs(Y_\ell(z))\subset [-n_\ell, n_\ell-1] $ for all $ \ell = 1,\ldots,k $. Therefore \eqref{eq:lemmaReq} holds.
\end{proof}

\begin{proof}[Proof of Lemma~\ref{lemma:EmptySpect2}]
Write $ \pQ(z) = \begin{bmatrix}
\pA(z) & \pB(z) \\ \pB^\star(z) & \pC(z)
\end{bmatrix} $,
where $ \pA(z) $ is an $ (s+1)\times (s+1) $ Hermite matrix of Laurent polynomials.
Since $ \pQ(z) $ is diagonally dominant at diagonals $ 1,\ldots,s $, from \eqref{def:DiagDominant1} and \eqref{eq:lemmaDiagIncreasing} we know that for all $ i < (s+1) $,
$ \fs(\pQ_{i,(s+1)}(z))\subsetneq \fs(\pQ_{i,i}(z))\subset \fs(\pQ_{(s+1),(s+1)}) $ and
$  \fs(\pQ_{(s+1),i}(z))\subsetneq \fs(\pQ_{i,i}(z))\subset \fs(\pQ_{(s+1),(s+1)}(z)) $. So $ \pA(z) $ is a diagonally dominant matrix.
Hence, for $ s = k-1 $, the lemma is true with $ \pU(z) = \mathbf{I}_{k} $.

For $ s < k-1 $, since $ \pQ_{1,1}(z)\not\equiv 0 $, we know that $ \pA(z)\not \equiv 0 $.
We can find integers $ \lambda_{s+2},\ldots,\lambda_{k} $ such that
$$ \widetilde{\pB}(z):= \pB(z)
\diag(z^{-\lambda_{s+2}},\ldots,z^{-\lambda_{k}})
$$

satisfies $ \ldeg(\widetilde{\pB}_{\ell,i}(z))\geqslant \ldeg(\pA_{\ell,\ell}(z)) $ for all $ \ell = 1,\ldots,(s+1) $, and $ i = 1,\ldots,k-(s+1) $.
Write $ \widetilde{\pB}(z) $ as column vectors
$ \widetilde{\pB}(z) = \begin{bmatrix}
\pb^{(s+2)}(z) & \ldots & \pb^{(k)}(z)
\end{bmatrix} $,
we have
\[ \ldeg(\pb^{(i)}_{\ell}(z)) \geqslant \ldeg(\pA_{\ell,\ell}(z)), \qquad
\mbox{for all}~ i = (s+2),\ldots,k, \quad \ell = 1,\ldots,(s+1). \]
Using Lemma~\ref{lemma:EmptySpect1}, for each $ i = (s+2),\ldots,k $, we can solve for a vector $ x^{(i)}(z) $ such that
$ y^{(i)}(z):= \pb^{(i)}(z) - \pA(z) x^{(i)}(z)$ satisfies
\begin{equation} \label{eq:lemma2Vector}
\fs(y^{(i)}_\ell(z)) \subsetneq \fs(\pA_{\ell,\ell}(z)), ~
\deg(y^{(i)}_\ell(z)) < \deg(\pA_{\ell,\ell}(z)),
\qquad \ell = 1,\ldots,k.
\end{equation}
Define $ \widetilde{X}(z):= \begin{bmatrix}
x^{(s+2)}(z) & \ldots & x^{(k)}(z)
\end{bmatrix}$,
$ Y(z):=\begin{bmatrix}
y^{(s+2)}(z) & \ldots & y^{(k)}(z)
\end{bmatrix} = \widetilde{\pB}(z) - \pA(z)\widetilde{X}(z) $, and
$ \Lambda(z):= \diag(z^{\lambda_{s+2}}, \ldots, z^{\lambda_{k}})$,
we know that
\begin{align} \label{eq:lemma2Final}
&\begin{bmatrix}
\mathbf{I}_{s+1} & \\
-\widetilde{X}^\star(z) & \mathbf{I}_{k-(s+1)}
\end{bmatrix}
\begin{bmatrix}
\mathbf{I}_{s+1} & \\
& \Lambda(z)
\end{bmatrix}
\pQ(z)
\begin{bmatrix}
\mathbf{I}_{s+1} & \\
& \Lambda^\star(z)
\end{bmatrix}
\begin{bmatrix}
\mathbf{I}_{s+1} & -\widetilde{X}(z)\\
 & \mathbf{I}_{k-(s+1)}
\end{bmatrix}  \notag \\
=&
\begin{bmatrix}
\mathbf{I}_{s+1} & \\
-\widetilde{X}^\star(z) & \mathbf{I}_{k-(s+1)}
\end{bmatrix}
\begin{bmatrix}
\pA(z) & \widetilde{\pB}(z) \\
\widetilde{\pB}^\star(z) & \Lambda(z)\pC(z)\Lambda^\star(z)
\end{bmatrix}
\begin{bmatrix}
\mathbf{I}_{s+1} & -\widetilde{X}(z)\\
& \mathbf{I}_{k-(s+1)}
\end{bmatrix}  \notag \\
=&
\begin{bmatrix}
\pA(z) & \widetilde{\pB}(z) - \pA(z)\widetilde{X}(z) \\
\widetilde{\pB}^\star(z) - \widetilde{X}^\star(z)\pA^\star(z) &  \pE(z)
\end{bmatrix}
=
\begin{bmatrix}
\pA(z) & Y(z) \\
Y^\star(z) &  \pE(z)
\end{bmatrix},
\end{align}
where $ \pE(z):= \Lambda(z)\pC(z)\Lambda^\star(z) - \widetilde{\pB}^\star(z)\widetilde{X}(z) - \widetilde{X}^\star(z)\widetilde{\pB}(z)$.
From \eqref{eq:lemma2Vector}, we deduce that the above matrix
$ \widetilde{\pQ}(z):=\begin{bmatrix}
\pA(z) & Y(z) \\
Y^\star(z) &  \pE(z)
\end{bmatrix} $
is diagonally dominant at the first $ (s+1) $ diagonal entries.
Taking
$$ \pU(z):= \begin{bmatrix}
\mathbf{I}_{s+1} & \\
-\widetilde{X}^\star(z) & \mathbf{I}_{k-(s+1)}
\end{bmatrix}
\begin{bmatrix}
\mathbf{I}_{s+1} & \\
& \Lambda(z)
\end{bmatrix} =
\begin{bmatrix}
\mathbf{I}_{s+1} & \\ -\widetilde{X}^\star(z) & \Lambda(z)
\end{bmatrix}, $$
the equality \eqref{eq:lemma2Final} implies that
$ \widetilde{\pQ}(z) = \pU(z)\pQ(z)\pU^\star(z) $.
Also, $ \det(\pU(z)) = z^{\lambda_{s+2}+\cdots+\lambda_{k}} $. Hence, $ \pU(z) $ is unimodular. Also, the top left $ (s+1)\times (s+1) $ submatrix of $ \widetilde{Q}(z) $, which is $ \pA(z) $, is the same as that of $ \pQ(z) $.
This completes the proof.
\end{proof}

\begin{proof}[Proof of Lemma~\ref{lemma:EmptySpect3}]
Since $ \pQ(z) $ is invertible, we can write $ \pQ(z) $ and $ \pQ^{-1}(z) $ as
$$\pQ(z) =
\begin{bmatrix}
0 & \pa^\star(z) \\
\pa(z) & \pE(z)
\end{bmatrix}
,\qquad
\pQ^{-1}(z) =
\begin{bmatrix}
\pb(z) & \pc^\star(z) \\
\pc(z) & \pF(z)
\end{bmatrix},$$
where $ \pa(z) $ and $ \pc(z) $
are both vectors of Laurent polynomials of size $ (k-1) $, $ \pE(z) $ and $ \pF(z) $ are matrices of Laurent polynomials of size $ (k-1)\times (k-1) $ and $ \pb(z) $ is a scalar Laurent polynomial satisfying $ \pb^\star(z) = \pb(z) $.
Set
$$y(z) := \tfrac{1}{2}\left(\pb(z) + 1 \right),\quad
x(z) := \tfrac{1}{2}\left(\pb(z)-1\right)\pa(z),\quad
\pV(z) :=
\begin{bmatrix}
y & \pc^\star \\
x & \mathbf{I}_{k-1}
\end{bmatrix}.$$
We have
\begin{align*}
\pV(z)\pQ(z)\pV^\star(z)
=&
\begin{bmatrix}
y & \pc^\star \\
x & \mathbf{I}_{k-1}
\end{bmatrix}
\begin{bmatrix}
0 & \pa^\star(z) \\
\pa(z) & \pE(z)
\end{bmatrix}
\begin{bmatrix}
y &x^\star \\
\pc & \mathbf{I}_{k-1}
\end{bmatrix}
=
\begin{bmatrix}
\pc^\star\pa y + y \pa^\star\pc + \pc^\star\pE\pc &
\pc^\star \pa x^\star + y \pa^\star + \pc^\star \pE \\
x \pa^\star\pc + y \pa + \pE \pc & \pa x^\star + x \pa^\star + \pE
\end{bmatrix}.
\end{align*}
By calculation,
\begin{align*}
&\pc^\star\pa y + y \pa^\star\pc + \pc^\star\pE\pc =
2 y +  \pc^\star\pE\pc =
\pb + 1 - \pb = 1. \\
&\pc^\star \pa x^\star + y \pa^\star + \pc^\star \pE
=  x^\star + y \pa^\star + \pc^\star \pE
= \frac{1}{2}(\pb-1)\pa^\star +  \frac{1}{2}(\pb+1) \pa^\star + \pc^\star \pE
= \pb\pa^\star + \pc^\star\pE
= (\pa\pb + \pE\pc)^\star
= 0. \\
&\pa x^\star + x \pa^\star + \pE
= \frac{1}{2}\left(\pb-1\right)\pa\pa^\star
+ \frac{1}{2}\left(\pb-1\right)\pa\pa^\star + \pE
= (\pb - 1) \pa\pa^\star + \pE.
\end{align*}
Thus, $ \pV(z)\pQ(z)\pV^\star(z) =
\diag(1,  (\pb(z) - 1) \pa(z)\pa(z)^\star + \pE(z) ) $.
Also, notice that
$ \det(\pV(z)) = y(z) - \pc^\star(z)x(z) =
\frac{1}{2}\left(\pb(z) + 1 \right)
- \frac{1}{2}\left(\pb(z) - 1 \right) = 1$, so $ \pV(z) $ is unimodular. We define $ \pU(z):=\pV^{-1}(z) $, and $ \widetilde{\pQ}(z):= (\pb(z) - 1) \pa(z)\pa(z)^\star + \pE(z)$.
It is straightforward to see that $ \pQ(z) = \pU(z)
\diag(1, \widetilde{\pQ}(z))
\pU^\star(z) $
satisfies all the requirements in the lemma.
\end{proof}

The following result is used in the proof of Theorem~\ref{thm:partialMultiplicity}.

\begin{lemma} \label{lemma:PartialMultiplicity}
Let $A(\xi)$ be an $n \times n$ matrix of analytic functions. Suppose that $A(\xi)$ can be factorized in some neighborhood of $\xi_0\in \C$ as follows:
$$ A(\xi)=E_{\xi_0}(\xi)
\diag((\xi-\xi_0)^{\alpha_1},\ldots,(\xi-\xi_0)^{\alpha_n})
F_{\xi_0}(\xi),$$
where
\begin{enumerate}
\item $E_{\xi_0}(\xi)$ and $F_{\xi_0}(\xi)$ are both $n \times n$ analytic matrices in some neighborhood of $ \xi_0 $;
\item $E_{\xi_0}(\xi_0)$ and $F_{\xi_0}(\xi_0)$ are both nonsingular;
\item the integer sequence $ \{\alpha_j\}_{j=1}^n $ is nondecreasing,
i.e., $0\leqslant \alpha_1 \leqslant \cdots \leqslant \alpha_n$.
\end{enumerate}
Then the sequence $\{\alpha_j\}_{j=1}^n$ is unique (independent of the factorization we use). We call it the \emph{partial multiplicities} of $A(\xi)$ at $\xi_0$.
\end{lemma}

\begin{proof}
If $C(\xi)$ is an $n \times n$ matrix, which is analytic in some neighborhood of $\xi_0$, and $\det(C(\xi_0))\neq 0$, then
$C^{-1}(\xi) = \frac{1}{\det(C(\xi))}\operatorname{adj}(C(\xi))$ is also an analytic matrix in some neighborhood of $\xi_0$.

Suppose that we have the following two different factorizations of $ A(\xi) $, both satisfy the three conditions in the lemma:
\begin{equation*}
A(\xi)  = E_{\xi_0}(\xi)
\diag((\xi-\xi_0)^{\alpha_1},\ldots,(\xi-\xi_0)^{\alpha_n})
F_{\xi_0}(\xi)
= \widetilde{E}_{\xi_0}(\xi)
\diag((\xi-\xi_0)^{\widetilde{\alpha}_1},\ldots,(\xi-\xi_0)^{\widetilde{\alpha}_n})
\widetilde{F}_{\xi_0}(\xi).
\end{equation*}
Then we have
\begin{equation} \label{eq:UniquenessSeq}
\diag((\xi-\xi_0)^{\widetilde{\alpha}_1},\ldots,(\xi-\xi_0)^{\widetilde{\alpha}_n})
=  P(\xi)
\diag((\xi-\xi_0)^{\alpha_1},\ldots,(\xi-\xi_0)^{\alpha_n})
Q(\xi),
\end{equation}
where $P(\xi):= \widetilde{E}^{-1}_{\xi_0}(\xi) E_{\xi_0}(\xi)$ and $Q(\xi):= F_{\xi_0}(\xi)\widetilde{F}^{-1}_{\xi_0}(\xi)$ are both analytic matrices in some neighborhood of $\xi_0$.
For all $k \in \{1, \ldots, n\}$, check the top left $ k\times k $ submatrix of \eqref{eq:UniquenessSeq}:
\begin{equation} \label{eq:UniquessSub}
\diag((\xi-\xi_0)^{\widetilde{\alpha}_1},\ldots,(\xi-\xi_0)^{\widetilde{\alpha}_n})
= P_{r,1:k} (\xi)
\diag((\xi-\xi_0)^{\alpha_1},\ldots,(\xi-\xi_0)^{\alpha_n})
Q_{c, 1:k}(\xi)
= R_k(\xi)Q_{c, 1:k}(\xi),
\end{equation}
where $P_{r,1:k} (\xi) $ is the $ k\times n $ submatrix of $P(\xi)$, constructed by taking the first $k$ rows of $ P(\xi) $, and $ Q_{c, 1:k}(\xi) $ is the $ n \times k $ submatrix of $Q(\xi)$, constructed by taking the first $k$ columns of $ Q(\xi) $. $R_k(\xi) :=P_{r,1:k} (\xi) \diag((\xi-\xi_0)^{\alpha_1}, \ldots , (\xi-\xi_0)^{\alpha_n})$ is a $k \times n$ matrix. From the definition, we see that the $s$-th column of $R_k(\xi)$ is $\bo((\xi - \xi_0)^{\alpha_s})$ as $\xi\rightarrow\xi_0$, for all $s=1,\ldots,n$.

Taking the determinant of \eqref{eq:UniquessSub}, by the Cauchy-Binet Formula, we have
\begin{equation} \label{eq:UniquenessCB}
(\xi-\xi_0)^{\widetilde{\alpha}_1+\ldots+\widetilde{\alpha}_k}
= \det(\diag((\xi-\xi_0)^{\widetilde{\alpha}_1} ,\ldots,(\xi-\xi_0)^{\widetilde{\alpha}_k}))
= \sum_{J\subseteq \{1,2,\ldots, n\}, }^{|J| = k}
\det(\left[R_k\right]_{c, J}(\xi))
\det(\left[Q_{c, 1:k}\right]_{r, J}(\xi)),
\end{equation}
where $\left[R_k\right]_{c, J}(\xi) $ is the $ k\times k $ submatrix of $R_k(\xi)$,
constructed by taking the columns with indices belonging to $J$;
$\left[Q_{c, 1:k}\right]_{r, J}(\xi) $ is the $ k\times k $ submatrix of $Q_{c, 1:k}(\xi)$,
constructed by taking the rows with indices belonging to $J$.
The summation is taken over all indices sets $J$, whose size is equal to $k$. Since all the elements in the $s$-th column of $\left[R_k\right]_{c, J}(\xi) $ are $\bo((\xi - \xi_0)^{\alpha_s})$ as $\xi\rightarrow \xi_0$, for all $s = 1,\ldots, n$
and the sequence $\{\alpha_j\}_{j=1}^n$  is nondecreasing,
we have
$$ \det(\left[R_k\right]_{c, J}(\xi)) =\bo((\xi - \xi_0)^{\alpha_1+\ldots+\alpha_k}), \qquad
\mbox{for all}~~ J\subseteq \{1,\ldots,n\},~ |J|=k. $$
Hence, each term in the summation on the right hand side of \eqref{eq:UniquenessCB} is $ \bo((\xi - \xi_0)^{\alpha_1+\ldots+\alpha_k}) $, as $ \xi \rightarrow \xi_0 $.
So
$ (\xi-\xi_0)^{\widetilde{\alpha}_1+\ldots+\widetilde{\alpha}_k} = \bo((\xi - \xi_0)^{\alpha_1+\ldots+\alpha_k}) $, which implies
$$ \alpha_1+\cdots+\alpha_k \leqslant \widetilde{\alpha}_1+\cdots+\widetilde{\alpha}_k, \qquad
\mbox{for all}~~ k=1,\ldots, n. $$
Similarly, we can prove
$\widetilde{\alpha}_1+\cdots+\widetilde{\alpha}_k \leqslant \alpha_1+\cdots+\alpha_k$
also holds for all $k=1,\ldots, n$. The two inequalities give that
$$\widetilde{\alpha}_1+\cdots+\widetilde{\alpha}_k = \alpha_1+\cdots+\alpha_k, \qquad \forall~ k=1,\ldots, n.$$
So $\{\alpha_j\}_{j=1}^n$ and $\{\widetilde{\alpha}_j\}_{j=1}^n$ must be the same sequence.
\end{proof}

Finally, we prove Algorithm~\ref{algo:Const-Sig}.

\begin{proof}[Proof of Algorithm~\ref{algo:Const-Sig}] Steps  (S1)--(S3) simply follow the proof of Lemma~\ref{thm:extract1}, while step  (S5) follows the proof of Algorithm \ref{algo:unimodular}. We only need to prove that step (S4) is feasible.

Suppose $z_0 = e^{-i\xi_0}$ for some $\xi_0\in\R$. By \eqref{eq:diagfirstder} and \eqref{eq:EVDFirstOrder} from the proof of Theorem \ref{thm:extract2}, we see that there exists some constant unitary matrix $W_0 := W^{-1}(\xi_0)$ such that
$ \mathring{\pA}(z):=W_0\widetilde{\pA}(z)W_0^\star $ satisfies:
\begin{equation} \label{eq:algo1}
W_0\widetilde{\pA}(z_0)W_0^\star = \mathring{\pA}(z_0) =
\diag(\lambda_1,\ldots,\lambda_{n-K},\mathbf{0}_{K\times K})
\end{equation}
where $\lambda_1,\ldots,\lambda_{n-K}\neq 0$ are the nonzero eigenvalues of $\widetilde{\pA}(z_0)$.
Write the Taylor expansion of $ \widetilde{\pA}(e^{-i\xi}) $ and
$ \mathring{\pA}(e^{-i\xi}) $
at $ \xi_0 $ as
$$ \widetilde{\pA}(e^{-i\xi}) = C_0 + C_1 (\xi - \xi_0) + \bo((\xi - \xi_0)^2), \quad
\mathring{\pA}(e^{-i\xi}) = D_0 + D_1 (\xi - \xi_0) + \bo((\xi - \xi_0)^2).
$$
where $ C_0 = \widetilde{\pA}(z_0) $, $ C_1 = -iz_0 \widetilde{\pA}'(z_0) $, $ D_0 = \mathring{\pA}(z_0) $ and
$ D_1 = -iz_0 \mathring{\pA}'(z_0) $.
We know that \eqref{eq:algo1} implies
\begin{equation} \label{eq:algoC0D0}
D_0 = W_0 C_0 W_0^\star = \diag(\lambda_1,\ldots,\lambda_{n-K},\mathbf{0}_{K\times K}).
\end{equation}
Also, \eqref{eq:EVDSecondOrder} implies
\begin{equation} \label{eq:algoC1D1}
 D_1 = W_0 C_1 W_0^\star =
\left[
\begin{array}{c|ccccc}
* & \multicolumn{5}{c}{*} \\
\hline
\multirow{5}{*}{*} &
\gamma_1^2 &             &        &              &  \\
&	           & -\gamma_2^2 &        &              &  \\
& 	           &             & \ddots &              &  \\
& 	           &             &        & \gamma_{K-1} &  \\
& 	           &             &        &              & \gamma_K
\end{array}
\right],
\end{equation}
where the lower right $K \times K$ submatrix is diagonal with $K/2$ positive and $K/2$ negative diagonal entries.
From \eqref{eq:algoC0D0} we see that the eigenspace of $C_0$ corresponding to the eigenvalue $0$ has dimension $K$. It must also be the span of the last $K$ column vectors of $W_0^\star$:
$$E_0= \mbox{span}\{w_{n-K+1},\ldots, w_n \}.$$

Also, by the construction of $\widetilde{\pA}(z)$ in \eqref{eq:SNFchange}, we see that the last $K$ columns and last $K$ rows of $\widetilde{\pA}(z_0)$ must be zero:
$$ C_0 = \widetilde{\pA}(z_0)=
\begin{bmatrix}
\mathbf{*} & \textbf{0}_{(n-K)\times K} \\
\mathbf{0}_{K \times (n-K)} & \mathbf{0}_{K \times K}
\end{bmatrix}. $$
So $E_0$ is also the span of $K$ natural basis vectors $E_0=\mbox{span}\{e_{n-K+1},\ldots,e_n\}$. This implies that $ \mbox{span}\{w_{n-K+1},\ldots, w_n \} = \mbox{span}\{e_{n-K+1},\ldots,e_n\} $,
so the matrix  $W_0^\star  $ of the eigenvectors of $ C_0 $ has the form
$W_0^\star=\begin{bmatrix}
W_{1} & \mathbf{0}_{(n-K)\times K} \\
W_{2} & \widetilde{W}_0
\end{bmatrix}$
for some matrices $ W_{1}, W_{2} $ and $ \widetilde{W}_0 $,
while $\widetilde{W}_0$ is a $K\times K$ unitary matrix.

From \eqref{eq:algoC1D1}, we deduce that
\begin{align*}
W_0 C_1 W_0^\star
=\begin{bmatrix}
W^\star_{1} & W^\star_{2}   \\
\textbf{0}_{ K\times (n-K)} & \widetilde{W}^\star_0
\end{bmatrix}
\begin{bmatrix}
* & *   \\
* & \pA_K
\end{bmatrix}
\begin{bmatrix}
W_{1} & \textbf{0}_{(n-K)\times K} \\
W_{2} & \widetilde{W}_0
\end{bmatrix}
= \left[
\begin{array}{c|ccccc}
* & \multicolumn{5}{c}{*} \\
\hline
\multirow{5}{*}{*} &
\gamma_1^2 &             &        &              &  \\
 &	           & -\gamma_2^2 &        &              &  \\
 & 	           &             & \ddots &              &  \\
 & 	           &             &        & \gamma_{K-1} &  \\
 & 	           &             &        &              & \gamma_K
\end{array}
\right].
\end{align*}
Therefore, $ \widetilde{W_0}^\star \pA_K \widetilde{W_0} = \diag(\gamma_1^2, -\gamma_2^2,\ldots, \gamma_{K-1}, \gamma_K)$.
Hence, the lower right $K\times K$ submatrix of $ C_1 $, i.e., $\pA_K$, has $K/2$ positive and $K/2$ negative eigenvalues. This proves item ($a$) in step  (S4) is feasible.
From the construction, we see that the redefined $\widetilde{\pA}(z)$ after item ($a$) in step (S4) satisfies:
$$ \widetilde{\pA}(z_0) = \begin{bmatrix}
\mathbf{*} & \textbf{0}_{(n-K)\times K} \\
\mathbf{0}_{K \times (n-K)} & \mathbf{0}_{K \times K}
\end{bmatrix},\qquad
-iz_0\widetilde{\pA}'(z_0) =
\left[
\begin{array}{c|ccccc}
* & \multicolumn{5}{c}{*} \\
\hline
\multirow{5}{*}{*} &
\gamma_1^2 &             &        &              &  \\
 &	           & -\gamma_2^2 &        &              &  \\
 & 	           &             & \ddots &              &  \\
 & 	           &             &        & \gamma_{K-1} &  \\
 & 	           &             &        &              & \gamma_K
\end{array}
\right]. $$
The design of $\pU_2$ in item (b) of step (S4) is similar to the matrix $V$ in \eqref{eq:V}. We can verify that the redefined $\widetilde{\pA}(z)$ after item (b) satisfies:
$$ \widetilde{\pA}(z_0) = \begin{bmatrix}
\mathbf{*} & \textbf{0}_{(n-K)\times K} \\
\mathbf{0}_{K \times (n-K)} & \mathbf{0}_{K \times K}
\end{bmatrix},\qquad
-iz_0\widetilde{\pA}'(z_0) =
\left[
\begin{array}{c|ccccc}
* & \multicolumn{5}{c}{*} \\
\hline
\multirow{5}{*}{*} &
0 &   -\gamma_2          &        &              &  \\
 &	  -\gamma_2         & -\gamma_2^2 &        &              &  \\
 & 	           &             & \ddots &              &  \\
 & 	           &             &        & \gamma_{K-1} &  \\
 & 	           &             &        &              & \gamma_K
\end{array}
\right]. $$
The above equality shows that $(z-z_0)$ divides both the $(n-K+1)$-th row and the $(n-K+1)$-th column of $\widetilde{\pA}(z)$, meanwhile, $(z-z_0)^2$ divides the $(n-K+1)$-th diagonal element of $\widetilde{\pA}(z)$. Thus the item (c) in  (S4) is feasible.
\end{proof}

\end{document}